\documentclass{dcds} 
\usepackage{amsmath}
 \usepackage{amssymb} 
 \usepackage{booktabs}
 \usepackage{mathtools}
 \usepackage{paralist}
 \usepackage{graphicx}  
 \usepackage{amsmath, amsfonts, amssymb, amsthm}
 \usepackage[T1]{fontenc}
 \usepackage{float}
\floatstyle{plaintop}
\restylefloat{table}
\usepackage{wrapfig}

 \usepackage{graphics} 
 \usepackage{epsfig}      
\usepackage{epstopdf}  
 \usepackage[colorlinks=true]{hyperref} 
 \hypersetup{urlcolor=blue, citecolor=red}
\usepackage{commath}               
 
\usepackage{multirow}
\usepackage{algorithmic}
\usepackage{algorithm}
\usepackage{graphicx}
\usepackage{float}
\usepackage{wrapfig}
\usepackage{mdframed}
\usepackage{xcolor}
\usepackage{bm}
\usepackage[title]{appendix}
\usepackage{resizegather}
\usepackage{bbm}
\usepackage{subcaption}

\def\currentvolume{ }
 \def\currentissue{ }
  \def\currentyear{ }
  \def\currentmonth{ }
    \def\ppages{ }
     \def\DOI{ }

\newtheorem{theorem}{Theorem} 

\newtheorem{lemma}[theorem]{Lemma}
\newtheorem{proposition}{Proposition}

\theoremstyle{definition}
\newtheorem{comp}[theorem]{Computational Result}

\newcommand{\xmark}{{\color{red}\ding{55}}}
\newcommand{\cmark}{{\color{blue}\ding{51}}}
\newcommand{\qmark}{{\fontfamily{cyklop}\selectfont \textit{?}}}
\newcommand{\tmark}{{\color{brown} $\thicksim$}}
\newcommand{\eg}{\textit{e.g.}}
\newcommand{\ie}{\textit{i.e.}}

\usepackage{xspace}
\newcommand{\so}[1]{{\text{SO}\negmedspace\left(#1\right)}}

\newcommand{\SNR}{\operatorname{SNR}}

\title[Method of moments for \textit{ab initio} modeling] 
      {Method of moments for 3-D single particle \textit{ab initio} modeling with non-uniform distribution of viewing angles}

\author[Sharon, Kileel, Khoo, Landa and Singer]{}


 \email{nsharon@tauex.tau.ac.il}
 \email{jkileel@math.princeton.edu}
 \email{yuehaw.khoo@gmail.com}
 \email{sboris20@gmail.com}
 \email{amits@math.princeton.edu}
 
\thanks{$^*$Corresponding author: Nir Sharon.}  
\thanks{$^{\dagger}$The first two authors contributed equally.}

\begin{document}

\maketitle


\centerline{\scshape Nir Sharon$^{*, \dagger}$}
\medskip
{\footnotesize
 \centerline{School of Mathematical Sciences, Tel Aviv University}
} 

\medskip

\centerline{\scshape Joe Kileel$^{\dagger}$}
\medskip
{\footnotesize
 \centerline{Program in Applied and Computational Mathematics, Princeton University}}
   
   \medskip
 
 \centerline{\scshape Yuehaw Khoo}
\medskip
{\footnotesize
 \centerline{Department of Statistics and the College, The University of Chicago}}
   
   \medskip
   
    \centerline{\scshape Boris Landa}
\medskip
{\footnotesize
 \centerline{Applied Mathematics Program, Yale University}}
   
   \medskip
   
    \centerline{\scshape Amit Singer}
\medskip
{\footnotesize
 \centerline{Program in Applied and Computational Mathematics and Department of Mathematics, Princeton University}}

\bigskip

\begin{abstract} %
	Single-particle reconstruction in cryo-electron microscopy (cryo-EM) is an increasingly popular technique for determining the 3-D structure of a molecule from several noisy 2-D projections images taken at unknown viewing angles.  Most reconstruction algorithms require a low-resolution initialization for the 3-D structure, which is the goal of \textit{ab initio} modeling.  Suggested by Zvi Kam in 1980, the method of moments (MoM) offers one approach, wherein low-order statistics of the 2-D images are computed and a 3-D structure is estimated by solving a system of polynomial equations. Unfortunately, Kam's method suffers from restrictive assumptions, most notably that viewing angles should be distributed uniformly. Often unrealistic, uniformity entails the computation of higher-order correlations, as in this case first and second moments fail to determine the 3-D structure. In the present paper, we remove this hypothesis, by permitting an unknown, non-uniform distribution of viewing angles in MoM.  Perhaps surprisingly, we show that this case is \textit{statistically easier} than the uniform case, as now first and second moments generically suffice to determine low-resolution expansions of the molecule.  In the idealized setting of a known, non-uniform distribution, we find an efficient provable algorithm inverting first and second moments.  For unknown, non-uniform distributions, we use non-convex optimization methods to solve for both the molecule and distribution.
\end{abstract}	

\section{Introduction}

Single-particle cryo-electron microscopy (cryo-EM) is an imaging method for determining the high-resolution 3-D structure of biological macromolecules without crystallization~\cite{frank2006three, kuhlbrandt2014resolution}.  The reconstruction process in cryo-EM determines the 3-D structure of a molecule from its noisy 2-D tomographic projection images.  By virtue of the experimental setup, each projection image is taken at an unknown viewing direction and has a very high level of noise, due to the small electron dose one can apply to the specimen before inflicting severe radiation damage, \eg,~\cite{biyani2017focus, fischer2015structure, merk2016breaking}. The computational pipeline that leads from the raw data, given many large unsegmented micrographs of projections, to the 3-D model consists of the following stages. The first step is particle picking, in which 2-D projection images are selected from micrographs. The selected particle images typically undergo 2-D classification to assess data quality and further improve particle picking. At this point, the 3-D reconstruction process begins, where often it is divided into two substeps of low-resolution modeling and 3-D refinement.  In this paper, we focus on the mathematical aspects of the former, namely the modeling part. In particular, we suggest using the method of moments (MoM) for \textit{ab initio} modeling. We illustrate this workflow with an overview given in Figure~\ref{fig:schematic}.

The last step in the reconstruction, also known as the refinement step, aims to improve the resolution as much as possible. This refinement process is typically a variant of the expectation-maximization (EM) algorithm which seeks the maximum likelihood estimator (MLE) via an efficient implementation, \eg,~\cite{scheres2012relion}. As such, 3-D refinement requires an initial structure that is close to the correct target structure~\cite{grigorieff2007frealign, scheres2012bayesian}. Serving this purpose, an \textit{ab initio} model is the result of a reconstruction process which depends solely on the data at hand with no \textit{a priori} assumptions about the 3-D structure of the molecule~\cite{reboul2018single}. We remark that the two primary challenges for cryo-EM reconstruction are the high level of noise and the unknown viewing directions. Mathematically, without the presence of noise, the unknown viewing directions could be recovered using common lines~\cite{vainshtein1986determination, van1987angular}. Then, the 3-D structure follows, for example, by tomographic inversion, see, \eg,~\cite{anden2018structural}. Reliable detection of common lines is limited however to high signal-to-noise (SNR) ratio.  As a result, the application of common lines based approaches is often limited to 2-D class averages rather than the original raw images~\cite{singer2011three}. 
Other alternatives such as frequency marching~\cite{barnett2017} and optimization using stochastic gradient have been suggested~\cite{punjani2017cryosparc}.
As optimization processes are designed to minimize highly non-convex cost functions, methods like SGD are not guaranteed to succeed.  In addition, as in the case of EM, it is not \textit{a priori} clear how many images are required.

\begin{figure}[h]  
	\begin{center}
		\includegraphics[width=0.9\textwidth]{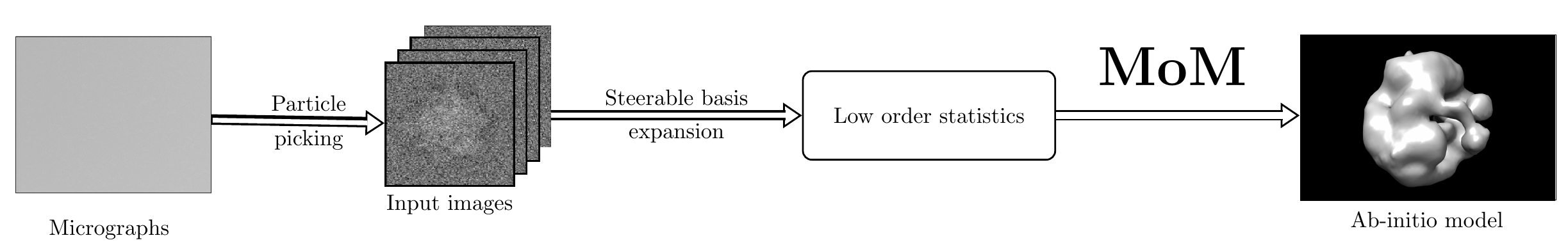}
	\end{center}	
	\caption{A schematic flowchart of 3-D reconstruction using method of moments (MoM).} 
	\label{fig:schematic}
\end{figure}

Approximately forty years ago, Zvi Kam proposed a method for 3-D \textit{ab initio} reconstruction based on computing the mean and covariance of the 2-D noisy images~\cite{kam1980reconstruction}. In order to uniquely determine the volume, the third moment (triple correlation) is also used besides the mean and covariance. In this approach, known as Kam's method, the 3-D volume is reconstructed without estimating the viewing directions. In this sense, Kam's method is strikingly different from common lines based approaches and maximum likelihood and other optimization methods that rely on orientation estimation for each image. 
 Crucially, Kam's method is effective at \textit{arbitrary} levels of noise, given sufficiently many picked particles for accurate estimation of the moment statistics. Additionally, Kam's method does not require any starting model, and it requires only one pass through the data to compute moments (contrary to other approaches needing access to the measurements multiple times). Despite the aforementioned advantages, Kam's method relies on the restrictive assumption that the viewing directions for the images are distributed uniformly over the sphere. This hypothesis, alongside other technical issues, has so far prevented a direct application of Kam's method to experimental cryo-EM data, for which viewing angles are typically non-uniform \cite{baldwin2019non, glaeser2017opinion, naydenova2017measuring, tan2017addressing}. 
This situation motivates us to explore generalizations of Kam's method better suited to cryo-EM data.\footnote{We remark that Kam's method, assuming uniform rotations, is of significant current interest in X-ray free electron laser (XFEL) single molecule imaging, where the assumption of uniformity more closely matches experimental reality \cite{donatelli2015, pande2014xfel, natureCommXfel}.}

In this paper, we generalize Kam's theory to the case of non-uniform distribution of viewing directions. We regard Kam's original approach with uniform distribution of viewing angles as a degenerate instance of MoM. In our formulation, we estimate both the 3-D structure and the unknown distribution of viewing angles jointly from the first two moments of the Fourier transformed images.  More precisely, for $n$ images $I_j,j=1,\ldots,n$, the first and second empirical moments of the Fourier transformed images, given in polar coordinates, $\widehat{I}_j(r,\varphi), j=1,\ldots,n$, are 
\begin{equation} \label{eqn:empirical_moments}
\widetilde{m}_1(r,\varphi) =  \frac{1}{n} \sum_{j=1}^n \widehat{I}_j(r,\varphi),\quad  \text{ and } \quad \widetilde{m}_2(r,\varphi,r',\varphi') = \frac{1}{n} \sum_{j=1}^n \widehat{I}_j(r,\varphi)\widehat{I}_j(r',\varphi'),
\end{equation}
which upon the above discretization become 2-D and 4-D tensors, respectively. Our basic rationale for trying to obtain the volume from the first two moments is as follows. Supposing the distribution of rotations of the \textit{image plane} to be uniform, then in the limit $n \rightarrow \infty$ the first moment is radially symmetric, that is, it is only a function of $r$ but is independent of $\varphi$.  Therefore, $\widetilde{m}_1$ may be regarded as a 1-D vector. Similarly, the second moment is a 3-D tensor (rather than 4-D) since it will only depend on $\varphi$ and $\varphi'$ through $\varphi-\varphi'$ as $n \rightarrow \infty$.  Also $I_j(r',\varphi')$ is linearly related to the molecule's volume via a tomographic projection. Thus, for images of size $N\times N$ pixels, the first and second moments should give rise to $\mathcal{O}(N^3)$ polynomial equations for the unknown volume and distribution. Assuming the volume is of size $N\times N \times N$ (and the distribution is of lower dimensionality), then the first and second moments have ``just'' the right number of equations (in terms of leading order) to determine the unknowns. Unfortunately, when the distribution of viewing directions is uniform, as noted by Kam~\cite{kam1980reconstruction}, the information encoded in the second moment is algebraically redundant; essentially it is the autocorrelation function (or equivalently, the power spectrum), and this information is insufficient for determining the structure of the molecule.  As we will see, a non-uniform distribution of viewing directions introduces additional terms in both the first and second moments, and extends the number of independent equations beyond the autocorrelation case.  In particular, we will show that non-uniformity guarantees uniqueness from the analytical counterparts of $\widetilde{m}_1$ and $\widetilde{m}_2$ in cases of a known distribution, and it guarantees finitely many solutions in other, more realistic, cases of an unknown distribution.

Our work is inspired by several earlier studies on simplified models in a setting called Multi-Reference Alignment (MRA). In MRA, a given group of transformations acts on a vector space of signals \cite{bandeira2017estimation}. For example, the group $\so{2}$ acts on the space of band-limited signals over the unit circle by rotating them counterclockwise (as a 1-D analog of cryo-EM). The task then is to estimate a ground truth signal from multiple noisy samples, corresponding to unknown group elements of a finite cyclic subgroup of $\so{2}$ acting on the signal.  The papers \cite{bandeira2017optimal, bendory2017bispectrum} show that for a uniform distribution over the group, the signal can be estimated from the third moment, and the number of samples required scales like the third power of the noise variance. On the other hand, for a non-uniform and also aperiodic distributions over the group, the signal can be estimated from the second moment, and the required number of samples scales quadratically with the noise variance~\cite{abbe2018multireference}. 

Despite the success of signal recovery in MRA from the first two moments under the action of the cyclic group, it is not apparent that such a strategy is still applicable in the case of cryo-EM. First, in cryo-EM, each image is obtained from the ground truth volume not just by applying a rotation in $\so{3}$, but also a tomographic projection. Moreover, the studies mentioned above (of MRA) consider finite abelian groups, whereas, in the case of cryo-EM, the group under consideration is the continuous non-commutative group $\so{3}$. The goal of this paper is then to investigate whether the first and second moment of the images is also sufficient for solving the inverse problem of structure determination in the cryo-EM setting.

\subsection{Our contribution}
We formulate the reconstruction problem in cryo-EM as an inverse problem of determining the volume and the distribution of viewing directions from the first two moments of the images. Assuming the volume and distribution are band-limited functions, they are discretized by Prolate Spheroidal Wave Functions (PSWFs) and Wigner matrices, respectively. The moments give rise to a polynomial system in which the unknowns are the coefficients of the volume and the distribution.  Using computational algebraic geometry techniques~\cite{CLO-book1,eisenbud-book,NAG-book}, we exhibit a range of band limits for the volume and the distribution such that the polynomial system  has only finitely many solutions, pointing to the possibility of exact recovery in these regimes.  Additionally, we comment on numerical stability issues, by providing condition number formulas for moment inversion. In the setting where the rotational distribution is known, we prove that the number of solutions is generically 1 and present an efficient algorithm for recovering the volume using ideas from tensor decomposition~\cite{harshman1970foundations}. For the practical case of an unknown distribution, we rely on methods from non-convex optimization and demonstrate, with synthetic data, successful {\em ab initio} model recovery of a molecule from the first two moments.

\subsection{Organization}
The paper is organized as follows. In Section~\ref{section: mom}, we present discretizations for the volume and distribution and derive the polynomial system obtained from the first two moments. In Section~\ref{sec:algebraic geometry}, we demonstrate that there exists a range of band limits where the polynomial system for the unknown molecule and distribution has only finitely many solutions. In Section~\ref{sec: implementation details}, we discuss some implementation details on how the system is solved and present numerical and visual results. Proofs and background material are provided in appendices.  For research reproducibility, MATLAB code is publicly available at \href{https://github.com/nirsharon/nonuniformMoM}{GitHub.com}.\footnote{The full address of the GitHub repository is https://github.com/nirsharon/nonuniformMoM.}

\section{Method of moments}\label{section: mom}
 
We begin by introducing the image formation model. Then, convenient basis for discretizing various continuous objects, namely the images and the volume (in the Fourier domain) as well as the distribution for orientations, are introduced. From these, relationships between the moments of the 2-D images and the 3-D molecular volume can be derived, enabling us to fit the molecular structure to the empirical moments of the images.
\subsection{Image formation model and the 3-D reconstruction problem}
In cryo-EM, data is acquired by projecting particles embedded in ice along the direction of the beaming electrons, resulting in tomographic images of the particles. The particles orient themselves randomly with respect to the projection direction. 
More formally, let $\phi \colon \mathbb{R}^3 \to \mathbb{R}$ be the Coulomb potential of the  3-D volume, and the projection operator be denoted by $\mathcal{P} \colon \mathbb{R}^3 \rightarrow \mathbb{R}^2$, where
\begin{equation}
\mathcal{P}\phi(x_1,x_2) := \int_{-\infty}^{\infty} \phi(x_1,x_2,x_3) \, dx_3.
\end{equation}
Assuming the $j$-th particle comes from the same volume $\phi$ but rotated by $R_j \in \textup{SO}(3)$, the image formation model is \cite{bhamre2016formation, frank2006three}
\begin{equation} \label{eqn:model_formation}
I_j = h_j * \mathcal{P}\left(R_j^{T} \cdot \phi \right) + \varepsilon_j , \quad R_j \in \so{3}, \quad j=1,\ldots,n \, ,
\end{equation}
where $\varepsilon_j$ is a random field modeling the noise term and $h_j$ is a point spread function, whose Fourier transform is known as the contrast transfer function (CTF) ~\cite{Mindell2003ctffind3, Rohou2015ctffind4, turovnova2017efficient}. Each image is assumed to lie within the box $[-1,1]\times [-1,1]$.  For size $N\times N$ discretized images, we assume the random field $\varepsilon_j \sim \mathcal{N}(0,\sigma^2 I_{N^2}),\ j=1,\ldots,n$. Here $R_j$ denotes an element in the group of $3 \times 3$ rotations $\so{3}$, and we define the group action by\footnote{Here we prefer to write the action of $R^T$ and correspondingly later we use Wigner $U$-matrices, instead of $R$ and Wigner $D$-matrices.  While simply notational, these conventions allow us to cite identities from \cite{chirikjian2016harmonic} verbatim, which are in terms of Wigner $U$-matrices and not Wigner $D$-matrices.}  
\begin{equation}
R_j^{T} \cdot \phi(x_1,x_2,x_3) := \phi(R_j \begin{bmatrix} x_1 & x_2 & x_3 \end{bmatrix}^T).
\end{equation}
The rotations $R_j$'s are not known since the molecules can take any orientation with respect to projection direction. For the purpose of simplifying the exposition, we shall henceforth disregard the CTF, by assuming
\begin{equation} \label{eqn:model_images}
I_j =  \mathcal{P}\left(R_j^T \cdot\phi \right) + \varepsilon_j , \quad j=1,\ldots,n.
\end{equation}
The presence of CTF is not expected to have a major impact on our main results, and we will incorporate the CTF in a future work. Typically, it is convenient to consider Fourier transform of the images, since by projection slice theorem, the Fourier transform $\widehat{I_j}$ of $I_j$ gives a slice of the Fourier coefficients $\widehat \phi$ of the volume $\phi$:
\begin{equation}
\label{eqn:FourierSlice}
\widehat {I_j}(x_1,x_2) = \widehat {\mathcal{P}(R_j^{T}\cdot \phi)}(x_1,x_2) + \widehat{\varepsilon_j}= (R_j^{T}\cdot \widehat{\phi})(x_1,x_2,x_3) \vert_{x_3=0} + \widehat{\varepsilon_j}.
\end{equation}
The goal of cryo-EM is to recover $\widehat \phi$ from the Fourier coefficients of the projections $\widehat {I_j}(x_1,x_2)$. While reconstructing $\widehat \phi$ given estimated $R_j$'s amounts to solving a standard computed tomography problem, we wish to reconstruct $\widehat \phi$ directly from the noisy images without estimating the rotations, for reasons detailed above. To this end, we assume the rotations are sampled from a distribution $\rho$ on $\so{3}$, where $\rho \colon \so 3\rightarrow \mathbb{R}$ is a smooth band-limited function. Then from the empirical moments of the images $\{\widehat {I_j}\}_{j=1}^n$, we jointly estimate the volume $\widehat \phi$ and the distribution $\rho$. 

\subsection{Representation of the volume, the distribution of rotations and the images} \label{subsec:representation}
As mentioned previously, the proposed method of moments consists of fitting the analytical moments
\begin{equation}
\label{eqn:analytical moments}
m_1 = \mathbb{E}_{R\sim \rho}[\widehat{\mathcal{P}\left(R^T \cdot \phi \right)}],\quad \text{  and  } \quad m_2 = \mathbb{E}_{R \sim \rho}[\widehat{\mathcal{P}\left(R^T \cdot \phi \right)}\otimes \widehat{\mathcal{P}\left(R^T \cdot \phi \right)}].
\end{equation}
to their empirical counterparts $\widetilde{m}_1$ and $\widetilde{m}_2$ as appears in~\eqref{eqn:empirical_moments} after debiasing.\footnote{By the law of large numbers, $\widetilde{m}_1 \rightarrow m_{1}$ and $\widetilde{m}_2 \rightarrow m_{2} + \sigma^2 I$ almost surely as $n \rightarrow \infty$, so $m_1$ is fitted to $\widetilde{m}_{1}$ and $m_{2}$ to $\widetilde{m}_{2} - \sigma^2 I$.  For notational convenience, we drop $\sigma^2 I$ in what follows, either assuming $\tilde{m}_{2}$ has been appropriately debiased already or $\sigma=0$.} Through fitting to the empirical moments, we seek to determine the Fourier volume $\hat \phi$ and also the distribution $\rho$. In this section, we present discretizations of $\hat \phi$ and $\rho$ by expanding them using convenient bases. 

\subsubsection{Basis for the Fourier volume $\widehat \phi$}
Since the image formation model involves rotations of the Fourier volume $\widehat \phi$, it is convenient to represent $\widehat \phi$ as an element of a function space closed under rotations;  in fact, this is the same as representing $\widehat \phi$ using spherical harmonics (see the Peter-Weyl theorem~\cite{chirikjian2016harmonic}):
\begin{equation} \label{eqn:phi_hat}
\widehat{\phi}\left(\kappa, \theta, \varphi  \right) = \sum_{\ell=0}^L  \sum_{m=-\ell}^\ell \sum_{s=1}^{S(\ell)} A_{\ell, m, s} F_{\ell,s}(\kappa) Y_\ell^m (\theta,\varphi). 
\end{equation}
Here $Y_\ell^m$ are the (complex) spherical harmonics:
\begin{equation} \label{eqn:sph_har}
Y_\ell^m(\theta,\varphi) = \sqrt{\frac{(2\ell+1)}{4 \pi}\frac{(\ell-m)!}{(\ell+m)!}} \,\, P_\ell^m(\cos \theta) \,\, e^{im\varphi}
\end{equation}
with associated Legendre polynomials $P_\ell^m$ defined by: 
\begin{equation}\label{eqn:assocLegen}
P_\ell^m(x) = \frac{(-1)^m}{2^{\ell}\ell!} (1-x^{2})^{m/2} \frac{d^{\ell+m}}{dx^{\ell+m}}(x^{2}-1)^{\ell}. 
\end{equation}
In Cartesian coordinates, spherical harmonics are polynomials of degree $\ell$. Without loss of generality, the radial frequency functions $F_{\ell,s}$ should form an orthonormal family (for each fixed $\ell$) with respect to $\kappa^2 d\kappa$, where $s = 1, \ldots, S(\ell)$ is referred to as the radial index.  Choices of radial functions suitable for molecular densities include spherical Bessel functions~\cite{andrews1992special}, which are eigenfunctions of the Laplacian on a closed ball with Dirichlet boundary condition, as well as the radial components of 3-D prolate spheroidal wave functions ~\cite{slepian1964prolate}.

We assume the volume is band-limited with Fourier coefficients supported within a radius of size $\pi N/2$, \ie, the Nyquist cutoff frequency for the images $I_j$'s discretized on a grid of size $N \times N$ (over the square $[-1,1]\times [-1,1])$.  Under this assumption, the maximum degree and radial indices $L$ and $S(\ell)$ in (\ref{eqn:phi_hat}) are essentially finite.  Further details on the particular basis functions $F_{\ell,s}$ and cutoffs $L$ and $S(\ell)$ that we choose to use are deferred to Section~\ref{app:pswf} in the appendix. Note that in practice, as we target low-resolution modeling, one can choose to decrease either the cutoff or the grid size to obtain more compact settings. The coefficients $A_{\ell,m,s} \in \mathbb{C}$ furnish our representation of $\widehat{\phi}$ using spherical harmonics. Note that since $\phi$ is real valued, its Fourier transform is conjugate-symmetric, which imposes restrictions on the coefficients $A_{\ell,m,s}$. The specific constraints are presented in Section~\ref{subse:constraints_in_optimization}.

The advantage of expanding $\widehat \phi$ in terms of spherical harmonics is that the space of degree $\ell$ spherical harmonics is closed under rotation; in group-theoretic language, this space forms a linear representation of $\textup{SO}(3)$.\footnote{\noindent In fact, this is an irreducible representation of $\textup{SO}(3)$ and varying $\ell$ these give all irreps.}    
Thus the
action of a rotation on $\widehat \phi$ amounts to a linear transformation on the expansion coefficients $A_{\ell,m,s}$ (with a block structure according to $\ell$ and $s$).  More precisely, fixing the vector space spanned by $\{Y^m_\ell(\theta,\varphi)\}_{m=-\ell}^\ell$ 
for a specific $\ell$, the action of a rotation $R$ on this vector space is represented by the \textit{Wigner matrix} $U^\ell(R)\in \mathbb{C}^{(2\ell+1)\times(2\ell+1)}$ (see \cite[p.~343]{chirikjian2016harmonic}) so that:
\begin{equation} \label{eqn:rotating_Y}
R^T \cdot Y_\ell^m(x) =  Y_\ell^m\left( R x \right) = \sum_{m^\prime = -\ell}^\ell U^\ell_{m,m^\prime}(R) Y_\ell^{m^\prime}(x), \quad x\in S^2.
\end{equation}
In particular, the matrix $U^\ell(R)$ is unitary, with entries degree $\ell$ polynomials in the entries of $R$  \cite{chirikjian2016harmonic}.  For all $R_1, R_2 \in \textup{SO}(3)$ and $\ell$, the group homomorphism property reads $U^{\ell}(R_{1}R_{2}) = U^{\ell}(R_{1})U^{\ell}(R_{2})$. In light of~\eqref{eqn:rotating_Y}, 3-D bases of the form $\{F_{\ell,s}(\kappa) Y_\ell^m(\theta,\phi)\}_{\ell,m,s}$ have been called \textit{steerable bases}.

\subsubsection{Basis for the probability distribution of rotations $\rho$}
 As we shall see, when expanding the volume in terms of spherical harmonics, the analytical moments \eqref{eqn:analytical moments} involve integrating different monomials of $\{U^l(R)\}_{\ell=0}^L$ with respect to the measure $\rho(R) dR$. To this end,  we assume the probability density $\rho$ over $\so{3}$  is a smooth band-limited function (and in a function space closed under rotation) by expanding
 \begin{equation} \label{eqn:distribution_extansion}
\rho(R) = \sum_{p=0}^P \sum_{u,v=-p}^p B_{p,u,v} U^p_{u,v}(R) , \quad R \in \so{3}.
\end{equation}
By Peter-Weyl, these form an orthonormal basis for $L^{2}(\textup{SO}(3))$, and for higher $p$ they are increasingly oscillatory functions on $\textup{SO}(3)$. Thus, expansion~\eqref{eqn:distribution_extansion} is analogous to using spherical harmonics to expand a smooth function on the sphere, or using Fourier modes for a function on the circle.  The cutoff $P \in \mathbb{N}$ is the band limit of the distribution $\rho$; we shall see in the next section that since we use only first and second moments it makes sense to assume $P \leq 2L$. Note that in the special case of a uniform distribution, the only nonzero coefficient is $B_{0,0,0} =1$. Also, $dR$ denotes the \textit{Haar measure}, which is the unique volume form on the group of total mass one that is invariant under left action.
Using the Euler angles parameterization of $\so{3}$, the Haar measure is of the form
\begin{equation} 
dR = \frac{1}{8\pi^2} \sin(\beta)  d\alpha d\beta d \gamma,
\end{equation}
where the normalizing constant ensures $ \int_{\textup{SO}(3)} dR = \int_{\alpha = 0}^{2\pi} \int_{\beta = 0}^\pi \int_{\gamma = 0}^{2\pi} dR  = 1 $.

\subsubsection{Basis for the 2-D images}
At this point, we discuss convenient representations for the images after Fourier transform, $\widehat I_j$.  Similarly to volumes, it is desirable to represent images using a function space closed under in-plane rotations, \ie, $\textup{SO}(2)$.  By the Peter-Weyl theorem, this is the same as expanding using Fourier modes, in a 2-D steerable basis:
\begin{equation}\label{eqn:image_expand}
\widehat I_j(\kappa, \varphi) = \sum_{q=-Q}^{Q} \sum_{t=1}^{T(q)} a^j_{q,t} \, f_{q,t}(\kappa) \, e^{i q \varphi}.
\end{equation}
Here the radial frequency functions $f_{q,t}$ (for fixed $q$) are taken to be an orthonormal basis with respect to $\kappa d\kappa$, with $\kappa$ referred to as the radial frequency. Comparing to expansion (\ref{eqn:phi_hat}) (see Section~\ref{subsec:representation}), it makes most sense to set $Q=L$.  Again, owing to the Nyquist frequency for the discretized images $I_j$, we may bound the cutoffs $T(q)$.  Typical choices for $f_{q,t}$ for representing tomographic images include Fourier-Bessel functions~\cite{zhao2013fourier} and the radial components of 2-D prolate spheroidal wave functions ~\cite{slepian1964prolate}.  Details on our specific choices are given in Section~\ref{sec:appendix_2D_PSWF} in the appendix.

\subsubsection{Choice of radial functions}
For the finite expansions in~\eqref{eqn:phi_hat} and~\eqref{eqn:image_expand} to accurately represent the Fourier transforms of the electric potential and its slices, one should carefully choose the radial functions $F_{\ell,s}$ and $f_{q,t}$, together with the truncation-related quantities $L$, $S(\ell)$, $Q$, and $T(q)$.
In this work, we consider $F_{\ell,s}$ and $f_{q,t}$ to be the radial parts of the three-dimensional and two-dimensional PSWFs~\cite{slepian1964prolate}, respectively.  In Appendix~\ref{app:pswf}, we describe some of the key properties of the PSWFs, and propose upper bounds for setting $L$, $S(\ell)$, $Q$, and $T(q)$.  In practice, band limits would be selected by balancing these expressivity considerations together with the well-posedness and conditioning considerations of Section~\ref{sec:algebraic geometry}.

\subsection{Low-order moments}
In this section, we derive the analytical relationship between the first two moments for the observed images $\{a^j_{q,t}\}_{j,q,t}$, and the coefficients $\{A_{\ell,m,s}\}_{\ell,m,s}$ and $\{B_{p,u,v}\}_{p,u,v}$ of the volume and distribution of rotations. These relationships will be used to determine $\{A_{\ell,m,s}\}_{\ell,m,s}$ and $\{B_{p,u,v}\}_{p,u,v}$ via solving a nonlinear least-squares problem. 

To this end, we first register a crucial relationship between the coefficients of the 2-D images and the 3-D volume. By indexing the images in terms of $R\in \so{3}$ (instead of $j$ in \eqref{eqn:image_expand}), we have:
\begin{equation} \label{eqn:image_expand_R}
\widehat{I_R}(\kappa,\varphi) = \sum_{q=-Q}^Q \sum_{t=1}^{T(q)} a^R_{q,t} f_{q,t}(\kappa) e^{i q\varphi}.
\end{equation}
On the other hand, using the Fourier slice theorem and \eqref{eqn:rotating_Y}:
\begin{align} \label{eqn:rotate_volume_SH}
\widehat{I_R}(\kappa,\varphi) &= R^T \cdot \widehat\phi(\kappa,\frac{\pi}{2},\varphi) \\
&=  \, \sum_{\ell =0}^L \sum_{s=1}^{S(\ell)} \sum_{m=-\ell}^\ell A_{\ell,m,s}  \, F_{\ell,s}(\kappa) \,\, R^T \! \cdot Y^{m}_{\ell}( \frac{\pi}{2}, \varphi) \\
&=  \sum_{\ell =0}^L \sum_{s=1}^{S(\ell)} \sum_{m=-\ell}^\ell \sum_{m^\prime=-\ell}^{\ell}  A_{\ell,m,s} \, F_{\ell,s}(\kappa) \, U^{\ell}_{m,m^\prime}(R) \, Y^{m^\prime}_{\ell}(\frac{\pi}{2}, \varphi).  
\end{align}
Multiplying \eqref{eqn:image_expand_R} and  \eqref{eqn:rotate_volume_SH} by $f_{q,t}(\kappa) e^{-i q\varphi}$ and integrating against $\frac{1}{2\pi} \kappa d\kappa d\varphi $, then combining the orthogonality relation
\begin{equation*}
\frac{1}{2\pi}\int_0^\infty \int_0^{2\pi} f_{q_{1},t_{1}}(\kappa) e^{iq_{1}\varphi}  f_{q_{2}, t_{2}}(\kappa) e^{-iq_{2}\varphi}d\varphi \kappa d\kappa \, = \, \mathbbm{1}_{q_{1}=q_{2}} \, \mathbbm{1}_{t_{1}=t_{2}}
\end{equation*}
with $Y^{m^\prime}_{\ell}(\frac{\pi}{2}, \varphi) \propto e^{i m^\prime \varphi}$, tells us
\begin{equation} \label{eqn:a_A_relation}
a^R_{q,t} =  \sum_{\ell =\vert q \vert}^L \sum_{s=1}^{S(\ell)} \sum_{m=-\ell}^\ell   A_{\ell,m,s}  \, U^{\ell}_{m,q}(R) \, \gamma_{\ell,s}^{q, t},
\end{equation}
where $\gamma_{\ell,s}^{q,t}$ are constants depending on the radial functions:
\begin{align} \label{eqn:gamma_def}
\gamma_{\ell,s}^{q,t} :\!&= \frac{1}{2\pi}\int_0^\infty \int_0^{2\pi} Y^{q}_{\ell}(\frac{\pi}{2}, \varphi) \, e^{-iq\varphi} \, F_{\ell,s}(\kappa) \, f_{q,t}(\kappa) \kappa d\kappa d\varphi  \\
 &=  \frac{1}{2\pi} \, \sqrt{\frac{(2\ell+1)}{4 \pi}\frac{(\ell-q)!}{(\ell+q)!}}  \,\, P^{q}_{\ell}(0) \int_0^\infty F_{\ell,s}(\kappa) \, f_{q,t}(\kappa) \kappa d\kappa.
\end{align}
From the term $P^{q}_{\ell}(0)$, we see $\gamma_{\ell,s}^{q,t} = 0$ if $q \not\equiv \ell$ (mod $2$) (and if $\vert q \vert > \ell$ then $\gamma_{\ell,s}^{q,t} := 0$).  Also one may check $\gamma_{\ell,s}^{-q,t} = (-1)^q \gamma_{\ell,s}^{q,t}$. Equation~\eqref{eqn:a_A_relation} connects 2-D image coefficients with 3-D volume coefficients.  We note we may as well choose $Q = L$ in \eqref{eqn:image_expand_R}, since if $\vert q \vert>L$ then $a^R_{q,t} = 0$. In practice, the coefficients $\gamma_{\ell,s}^{q,t}$ are calculated via numerical integration over a closed segment, according to the localization property of the PSWFs, see Appendix~\ref{app:pswf} and~\cite{lederman2017numerical}.

\subsubsection{The first moment}
In this section, from \eqref{eqn:a_A_relation} the relationship between the first moment of the images and the volume is derived. Taking the expectation over $R$, and using the distribution expansion~\eqref{eqn:distribution_extansion}, we get
\begin{align}
\mathbb{E}_R[a_{q,t}^R]  &=  \sum_{\ell=\vert q \vert}^{L} \sum_{s=1}^{S(\ell)} \sum_{m=-\ell}^{\ell}  A_{\ell,m,s}  \gamma_{\ell,s}^{q,t} \int U^{\ell}_{m,q}(R) \rho(R) dR \\
&=  \sum_{\ell=\vert q \vert}^{L} \sum_{s=1}^{S(\ell)} \sum_{m=-\ell}^{\ell} \sum_{p=0}^{P} \sum_{u,v= -p}^{p} A_{\ell,m,s} \, B_{p,u,v} \, \gamma_{\ell,s}^{q,t} \int U^{\ell}_{m,q}(R) U^p_{u,v}(R)  dR \\
&=  \sum_{\ell = \vert q \vert}^{\min(L,P)} \,\, \sum_{s=1}^{S(\ell)} \,\, \sum_{m=-\ell}^{\ell} A_{\ell,m,s} \,\, B_{\ell,-m,-q} \,\, \gamma_{\ell,s}^{q,t} \,\, \frac{(-1)^{m+q}}{2\ell+1}. \label{eqn:first_mom}
\end{align} 
The last equation follows from the orthogonality of the Wigner matrix entries~\cite[p. 68]{Angular_Momentum_book}
\begin{equation}\label{eqn:wigner_orthog}
\int_{\textup{SO}(3)} \overline{U^{\ell}_{m,n}(R)} U^{p}_{u,v}(R) \, dR  = \frac{1}{2\ell + 1}\,\, {\mathbbm{1}}_{\ell = p} \,\, {\mathbbm{1}}_{u = m} \,\, {\mathbbm{1}}_{v = n},
\end{equation}
and 
\begin{equation}\label{eqn:wigner_conj}
\overline{U^p_{u,v}(R)} = (-1)^{u+v} \, U^p_{-u,-v}(R).
\end{equation}
The first moment gives a set of bilinear forms in the unknowns $\{A_{\ell,m,s}\}_{\ell,m,s}$ and $\{B_{p,u,v}\}_{p,u,v}$,
as seen in~\eqref{eqn:first_mom} for each $(q,t)$ with $\vert q \vert \leq \min(L,P)$ and $1 \leq t \leq T(q)$.

It is convenient to provide compact notation for the first moment formula.  To this end, we introduce:
\begin{enumerate}\setlength{\itemsep}{1pt}
\item $\mathcal{A}_{\ell}$, a matrix of size $S(\ell) \times (2\ell + 1)$ given by $(\mathcal{A}_{\ell})_{s,m} = A_{\ell,m,s}$
\item $\beta^{q}_{\ell}$, a vector of size $2\ell + 1$ given by $(\beta^{q}_{\ell})_{m} = \tfrac{(-1)^{m}}{2\ell + 1}B_{\ell,-m,-q}$
\item $\Gamma^{q}_{\ell}$, a matrix of size $T(q) \times S(\ell)$ given by 
$(\Gamma^{q}_{\ell})_{t,s} = (-1)^{q} \gamma_{\ell,s}^{q,t}$. 
\end{enumerate}
Item 2 is zero if $\ell < \lvert q \rvert$ and item 3 is zero if either $ \ell < \lvert q \rvert$ or $\ell \not\equiv q \, (\textup{mod }2)$. In this notation, the first moment formula \eqref{eqn:first_mom} (with fixed $q$ and varying $t$) reads:
\begin{equation}\label{eqn:first_mom_compacy}
m_{1}(q) \,\, := \,\, \left( \mathbb{E}[a^{R}_{q,t}] \right)_{t=1, \ldots, T(q)} \hspace{1em} = \sum_{\substack{\ell \, : \, \lvert q \rvert  \leq \ell \leq L \\ \ell \equiv q \, (\textup{mod }2) }}\Gamma^{q}_{\ell} \, \mathcal{A}_{\ell} \, \beta^{q}_{\ell}.
\end{equation}
Here $m_{1}(q) \in \mathbb{C}^{T(q)}$ is nonzero only if $\lvert q \rvert \leq \min(L,P)$.

\subsubsection{The second moment}
Higher moments require higher powers of the image coefficients, and so in the case of the second moment and for $\vert q_{1} \vert, \vert q_{2} \vert \leq L$, we have
\begin{eqnarray}\label{eqn:second_momv1}
\mathbb{E}_R \left[ a^R_{q_1,t_1} a^R_{q_2,t_2} \right] = \sum_{\ell_1 = |q_1|}^{L} \sum_{s_1=1}^{S(\ell_1)}  \sum_{m_{1} = -\ell_{1}}^{\ell_{1}} \sum_{\ell_2 = |q_2|}^{L} \sum_{s_2=1}^{S(\ell_2)} \sum_{m_{2} = -\ell_{2}}^{\ell_{2}} A_{\ell_1,m_1,s_1} \gamma_{\ell_1,s_1}^{q_1,t_1} \\ \times \,\,\,\,  A_{\ell_2,m_2,s_2}   \gamma_{\ell_2,s_2}^{q_2,t_2} \int U^{\ell_1}_{m_1,q_1}(R) U^{\ell_2}_{m_2,q_2}(R) \rho(R) dR
\end{eqnarray}
where
\begin{equation}\label{eqn:triple_Uv1}
 \int U^{\ell_1}_{m_1,q_1}(R) U^{\ell_2}_{m_2,q_2}(R) \rho(R) dR = \sum_{p=0}^{P} \sum_{u, v=-p}^{p} B_{p,u,v} \int U^{\ell_1}_{m_1,q_1}(R) U^{\ell_2}_{m_2,q_2}(R) U^{p}_{u,v}(R) dR.
\end{equation}

The product of two Wigner matrix entries is expressed as a linear combination of Wigner matrix entries \cite[p. 351]{chirikjian2016harmonic},
\begin{gather} \label{eqn:CG_expand}
U^{\ell_1}_{m_1,q_1}(R) U^{\ell_2}_{m_2,q_2}(R) = \sum_{\ell_{3} =  \vert \ell_{2} - \ell_{1} \vert}^{\ell_{1} + \ell_{2}} \mathcal{C}_{\ell_3}(\ell_1, \ell_2, m_1, m_2, q_1, q_2) \, U^{\ell_3}_{m_1+m_2, n_1+n_2}(R),
\end{gather}
where 
\begin{equation} \label{eqn:C_asProductOfCGcoefs}
\mathcal{C}_{\ell_3}(\ell_1, \ell_2, m_1, m_2, q_1, q_2) = 
C(\ell_1, m_1; \ell_2, m_2 \vert \ell_3, m_1+m_2) C(\ell_1, q_1; \ell_2, q_2 \vert \ell_3, q_1+q_2) ,
\end{equation}	
is the product of two Clebsch-Gordan coefficients. This product is nonzero only if $(\ell_{1}, \ell_{2}, \ell_3)$ satisfy the triangle inequalities. Substituting \eqref{eqn:CG_expand} into \eqref{eqn:triple_Uv1}, and invoking \eqref{eqn:wigner_orthog} and \eqref{eqn:wigner_conj}, we obtain:
\begin{multline}\label{eqn:triple_Uv2}
\int U^{\ell_1}_{m_1,q_1}(R) U^{\ell_2}_{m_2,q_2}(R) \rho(R) dR = \sum_{p} \mathcal{C}_{p}(\ell_1, \ell_2, m_1, m_2, q_1, q_2) \\
\times \,\,\,\, B_{p,-m_1-m_2,-q_1-q_2} \,\, \frac{(-1)^{m_1+m_2+q_1+q_2}}{2p+1}
\end{multline}
where the sum is over $p$ satisfying $\max(\vert \ell_1 - \ell_2 \vert, \vert m_1 + m_2 \vert, \vert q_1 + q_2 \vert) \leq p \leq \min(\ell_1 + \ell_2, P)$.  Now substituting into \eqref{eqn:second_momv1} gives:
\begin{multline}\label{eqn:second_momv2}
\mathbb{E}_R \left[ a^R_{q_1,t_1} a^R_{q_2,t_2} \right] = \sum_{\ell_1, s_1, m_1, \ell_2, s_2, m_2} A_{\ell_1, m_1, s_1} \, A_{\ell_2, m_2, s_2} \, \gamma^{q_1,t_1}_{\ell_1,s_1} \, \gamma^{q_2,t_2}_{\ell_2,t_2} \, (-1)^{q_1+q_2} \,\, \times  \\ 
 \sum_{p} B_{p,-m_1-m_2,-q_1-q_2}  \mathcal{C}_{p}(\ell_1, \ell_2, m_1, m_2, q_1, q_2) \frac{(-1)^{m_1+m_2}}{2p+1}
\end{multline}
where the first sum has the range of \eqref{eqn:second_momv1} and the second sum has range of \eqref{eqn:triple_Uv2}.  The second moment thus gives a set of polynomials in unknowns $\{A_{\ell,m,s}\}_{\ell,m,s}$ and $\{B_{p,u,v}\}_{p,u,v}$, quadratic in the volume coefficients and linear in the distribution coefficients,
namely, the expression in \eqref{eqn:second_momv2} for each $(q_1,t_1,q_2,t_2)$ with $\vert q_1 \vert \leq L$, $\vert q_2 \vert \leq L$, $\lvert q_1 + q_2 \rvert \leq P$, $1 \leq t_1 \leq T(q_1)$ and $1 \leq t_2 \leq T(q_2)$.  
Also, it may be assumed that $P \leq 2L$, since $B_{p,u,v}$ with $p > 2L$ does not contribute in either~\eqref{eqn:second_momv2} or~\eqref{eqn:first_mom}.

As for the first moment, it will be convenient to rewrite the second moment in compact notation.  Let us further introduce:
\begin{enumerate}\addtocounter{enumi}{3}
\item $\mathcal{B}^{q_{1}, q_{2}}_{\ell_{1}, \ell_{2}}$, a matrix of size $(2 \ell_{1} + 1) \times (2 \ell_{2} + 1)$ given by 
\[ (\mathcal{B}^{q_{1}, q_{2}}_{\ell_{1}, \ell_{2}})_{m_{1}, m_{2}} = \sum_{p} B_{p,-m_{1}-m_{2},-q_{1}-q_{2}}  \mathcal{C}_{p}(\ell_1, \ell_2, m_1, m_2, q_1, q_2) \tfrac{(-1)^{m_{1} + m_{2}}}{2p+1} , \] 
where the sum is over $\max(\lvert \ell_{1} - \ell_{2} \rvert, \lvert m_{1} + m_{2} \rvert, \lvert q_{1} + q_{2} \rvert) \leq p \leq \min(\ell_{1}+\ell_{2}, P)$ and $C_{p}$ denotes the product Clebsch-Gordan coefficients in \eqref{eqn:C_asProductOfCGcoefs}.
\end{enumerate}
Item 4 is zero if either $\ell_{1} < \lvert q_{1} \rvert$ or $\ell_{2} < \lvert q_{2} \rvert$ or $\max(\lvert \ell_{1} - \ell_{2} \rvert, \lvert q_{1} + q_{2} \rvert) > P$.
Now for fixed $q_1, q_2$ and varying $t_1, t_2$, the second moment \eqref{eqn:second_momv2} neatly reads:
\begin{equation}\label{eqn:second_mom_compacy}
m_{2}(q_1,q_2) \,\, := \,\, \Big{(} \mathbb{E}[ a^{R}_{q_{1}, t_{1}} a^{R}_{q_{2},t_{2}} ]\Big{)}_{\substack{ t_{1} = 1, \ldots, T(q_{1}) \\ t_{2} = 1, \ldots, T(q_{2})}} \hspace{1em} =  
\sum_{\substack{\ell_{1}, \ell_{2}\, : \, \lvert q_1 \rvert  \leq \ell_1 \leq L \\ \lvert q_2 \rvert  \leq \ell_2 \leq L  \\ \ell_1 \equiv q_1 \, (\textup{mod }2) \\ \ell_2 \equiv q_2 \, (\textup{mod }2) \\ \lvert \ell_{1} - \ell_{2} \rvert \leq P}}
\Gamma^{q_1}_{\ell_1} \, \mathcal{A}_{\ell_1} \, \mathcal{B}^{q_{1}, q_{2}}_{\ell_{1}, \ell_{2}} \, \mathcal{A}^{T}_{\ell_{2}} \, (\Gamma^{q_{2}}_{\ell_{2}})^{T}.
\end{equation}
Here $m_{2}(q_{1}, q_{2}) \in \mathbb{C}^{T(q_{1}) \times T(q_{2})}$ is nonzero only if $\lvert q_1 \rvert, \lvert q_2 \rvert \leq L$ and $\lvert q_1 + q_2 \rvert \leq P$.

\section{Uniqueness Guarantees and Conditioning} \label{sec:algebraic geometry}

Here, we derive uniqueness guarantees and comment on intrinsic conditioning for the polynomial system defined by the first and second moments, \eqref{eqn:first_mom_compacy} and \eqref{eqn:second_mom_compacy}.  

Analysis comes in four cases, according to assumptions on the distribution $\rho$: whether $\rho$ is known or unknown; and if $\rho$ is invariant to in-plane rotations, i.e., $\rho$ depends only on the viewing directions up to rotations that retain the $z$-axis. This invariance restricts $\rho$ to be a non-uniform distribution function over $S^2$, see subsection~\ref{sec:in-planeCon}. If $\rho$ is not invariant to in-plane rotations, we say $\rho$ is totally non-uniform as a distribution on the entire $ \so{3}$. Throughout, our general finding is \textit{well-posedness}, \ie, the molecule is uniquely determined by first and second moments up to finitely many solutions, under genericity assumptions, for a range of band limits $L$ and $P$.  In the case of a known totally non-uniform distribution, we prove the number of solutions is $1$, and give an efficient, explicit algorithm to solve for $\{A_{\ell, m ,s}\}$.  For all cases, sensitivity of the solution to errors in the moments is quantified by condition number formulas. 

\subsection{Known, totally non-uniform \texorpdfstring{$\rho$}{distribution}}\label{sec:easy-case}
For this case, we have a provable algorithm that recovers $\{A_{\ell,m,s}\}$ from \eqref{eqn:first_mom_compacy} and \eqref{eqn:second_mom_compacy} (up to the satisfaction of technical genericity and band limit conditions). Remarkably, while the polynomial system is nonlinear (consisting of both quadratic and linear equations), our method is based only on linear algebra.  The main technical idea is \textit{simultaneous diagonalization} borrowed from Jennrich's well-known algorithm for third-order tensor decomposition \cite{harshman1970foundations}, that was also used recently for signal recovery in MRA~\cite{perry2017sample}.
\begin{theorem}\label{thm:known_rho}
The molecule $\{ A_{\ell, m, s} \}$ is uniquely determined by the analytical first and second moments, \eqref{eqn:first_mom_compacy} and \eqref{eqn:second_mom_compacy}, in the case the distribution $\{ B_{p,u,v} \}$ is \textup{totally non-uniform}, \textup{known} and $P \geq 2L$, provided it also holds:
\begin{enumerate}[(i)]
\item The matrices $B_{1} := \mathcal{B}^{L, L}_{L, L}$ and $B_{2} = \mathcal{B}^{L, -L}_{L,L}$ of size $(2L + 1) \times (2L+1)$ both have full rank, and $B_{1} B_{2}^{-1}$ has distinct eigenvalues.  Likewise $B_{3} := \mathcal{B}^{L-1,L-1}_{L-1,L-1}$ and $B_{4} = \mathcal{B}^{L-1,1-L}_{L-1,L-1}$ of size $(2L - 1) \times (2L - 1)$ both have full rank, and $B_{3} B_{4}^{-1}$ has distinct eigenvalues.
\item Writing $B_{1} B_{2}^{-1} =: Q_{12} D_{12} Q_{12}^{-1}$ and $B_{3} B_{4}^{-1} =: Q_{34} D_{34} Q_{34}^{-1}$ for eigendecompositions, the vectors $b_{12} := Q_{12}^{-1} \beta^{L}_{L}$ of size $2L+1$  and 
$b_{34} := Q_{34}^{-1} \beta^{L-1}_{L-1}$ of size $2L-1$ both have no zero entries. 
\item For $\ell \leq L-2$, the matrix $\mathcal{B}_{\ell,L}^{\ell,L}$ of size $(2\ell+1) \times (2L+1)$ has full row rank.
\item For all $\ell$, the matrix $\mathcal{A}_{\ell}$ of size $S(\ell) \times (2\ell + 1)$ has full column rank.  
\item For $\ell \geq \lvert q \rvert$ with $\ell \equiv q \,\, (\textup{mod }2)$, the matrix $\Gamma^{q}_{\ell}$ of size $T(q) \times S(\ell)$ has full column rank.
\end{enumerate}
Moreover in this case, there is a provable algorithm inverting \eqref{eqn:first_mom_compacy} and \eqref{eqn:second_mom_compacy} to get $\{A_{\ell, m, s}\}$ in time $\mathcal{O}\left(L^2 \cdot T^3\right)$, where $T := \max_{q} T(q)$.
\end{theorem}
\begin{proof}
	For this proof, we need some general properties of the \textit{Moore-Penrose pseudo-inverse}, denoted by $\dagger$, as in \cite{ben2003generalized}.  In particular, if $Y \in \mathbb{C}^{n_1 \times n_2}$ has full column rank and $Z \in \mathbb{C}^{n_2 \times n_3}$ has full row rank, then $Y^{\dagger} Y= I_{n_2}$, $Z Z^{\dagger} = I_{n_2}$, $(YZ)^{\dagger} = Z^{\dagger} Y^{\dagger}$, and also, pseudo-inversion and transposition commute.
	
	Proceeding, the second moment with $q_{1}=L, q_{2}=L$ tells us:
	\begin{equation}\label{eq:secMoM1}
	m_{2}(L,L) \,\,\, = \,\,\, \Gamma^{L}_{L} \mathcal{A}_{L} B_{1} (\mathcal{A}_{L})^{T}  (\Gamma^{L}_{L})^{T} \,\, \in \,\, \mathbb{C}^{T(L) \times T(L)},
	\end{equation}
	and with $q_{1}=L, q_{2}=-L$:
	\begin{equation}\label{eq:secMoM2}
	m_{2}(L,-L) \,\, = \,\,\, \Gamma^{L}_{L} \mathcal{A}_{L} B_{2} (\mathcal{A}_{L})^{T} (\Gamma^{-L}_{L})^{T} \,\, \in \,\, \mathbb{C}^{T(L) \times T(L)},
	\end{equation}
	
	\noindent where $\Gamma_{L}^{L} = (-1)^L \Gamma_{L}^{-L}$.  We compute $(-1)^L$ times the Moore-Penrose psuedo-inverse of \eqref{eq:secMoM2} and then multiply this on the right of \eqref{eq:secMoM1}.  Because $\Gamma_{L}^{L}$ and $\mathcal{A}_{L}$ are each tall with full column rank by assumptions ($v$) and ($iv$), respectively, and $B_2$ is invertible by ($i$), properties of the pseudo-inverse imply:
	\begin{align}\label{eq:like_jenrich}
	(-1)^L m_{2}(L,L) m_{2}(L,-L)^{\dagger} &= \left( \Gamma^{L}_{L} \mathcal{A}_{L} B_{1} (\mathcal{A}_{L})^{T} (\Gamma_{L}^{L})^{T}  \right)  \left( \Gamma^{L}_{L} \mathcal{A}_{L} B_{2} (\mathcal{A}_{L})^{T} (\Gamma_{L}^{L})^{T}  \right)^{\dagger} \nonumber \\ \nonumber 
	&=  (\Gamma^{L}_{L} \mathcal{A}_{L}) B_{1} (\mathcal{A}_{L})^{T} (\Gamma_{L}^{L})^{T}   (\Gamma_{L}^{L})^{T \dagger}  (\mathcal{A}_{L})^{T \dagger} B_{2}^{-1} (\Gamma^{L}_{L} \mathcal{A}_{L})^{\dagger} \\ \nonumber
	&=(\Gamma^{L}_{L} \mathcal{A}_{L}) B_{1} B_{2}^{-1} (\Gamma^{L}_{L} \mathcal{A}_{L})^{\dagger} \\ \nonumber
	&= (\Gamma^{L}_{L} \mathcal{A}_{L}) Q_{12} D_{12} Q_{12}^{-1} (\Gamma^{L}_{L} \mathcal{A}_{L})^{\dagger} \\ 
	&= \left(\Gamma^{L}_{L} \mathcal{A}_{L} Q_{12}\right) D_{12} \left(\Gamma^{L}_{L} \mathcal{A}_{L} Q_{12}\right)^{\dagger},
	\end{align}
	where we have substituted in an eigendecomposition  $B_{1} B_{2}^{-1} = Q_{12} D_{12} Q_{12}^{-1}$.
	As $B_{1}B_{2}^{-1}$ has distinct eigenvalues by condition ($i$), we see that the eigenvectors of $(-1)^L m_{2}(L,L) m_{2}(L,-L)^{\dagger}$ are unique up to scale and given as the columns of $\Gamma^{L}_{L} \mathcal{A}_{L} Q_{12}$. 
	Thus, $\Gamma^{L}_{L} \mathcal{A}_{L} Q_{12} = X\Lambda$, where $X$ consists of eigenvectors of \eqref{eq:like_jenrich} and $\Lambda$ is an unknown (as yet) diagonal matrix.  
	
	To disambiguate the scales $\Lambda$, we compare with the first moment for $q=L$:
	\begin{equation}
	m_{1}(L)=\Gamma_{L}^{L} \mathcal{A}_{L} \beta_L^{L} = X \Lambda Q_{12}^{-1} \beta_L^{L} = X \Lambda b_{12}.
	\end{equation}
	Multiplying on the left by $X^{\dagger}$ gives $X^{\dagger} m_{1}(L) = \Lambda b_{12}$, an equality of matrix-vector products in which the only unknown is the diagonal matrix $\Lambda$.  By the full support of $b_{12}$ (assumption ($ii$)), this determines $\Lambda$.  Substituting into $X \Lambda$, we now know $\Gamma^{L}_{L} \mathcal{A}_{L} Q_{12}$.  Multiplying on the left by $\Gamma^{L \dagger}_{L}$ and on the right by $Q_{12}^{-1}$ tells us $\mathcal{A}_{L}$.  
	
	Backward marching, the second moment with $q_1 = L-2$ and $q_2=L$ reads: 
	\begin{equation}
	m_{2}(L-2,L)= \Gamma^{L-2}_{L} \mathcal{A}_{L} B^{L-2,L}_{L,L} (\mathcal{A}_{L})^{T} (\Gamma^{L}_{L})^{T} \,\, + \,\, \Gamma^{L-2}_{L-2} \mathcal{A}_{L-2} B^{L-2,L}_{L-2,L} (\mathcal{A}_{L})^{T} (\Gamma^{L}_{L})^{T}.
	\end{equation} 
	At this point, we know the first term, and thus the second term gives us $ \mathcal{A}_{L-2}$ by appropriately multiplying by pseudo-inverses ($B^{L-2,L}_{L,L}$ is right-invertible by ($iii$)).  
	
	Then, we may look at the second moments with $q_1 = L-4$ and $q_2=L$ to similarly determine $\mathcal{A}_{L-4}$, and so on, to $\mathcal{A}_{0}$ or $\mathcal{A}_{1}$ (depending on the parity of $L$). Analogous reasoning and usage of the assumptions gives $\mathcal{A}_{L-1}, \mathcal{A}_{L-3}, \ldots$
	
	We have provided an algorithm to solve for each $\mathcal{A}_{\ell}$, which proves uniqueness of $\mathcal{A}_{\ell}$ as a byproduct.  The time complexity of the algorithm is $\mathcal{O}(L^{2} T^3)$ since it involves $\mathcal{O}(L^2)$ matrix operations  --matrix multiplications, pseudo-inversions or eigendecompositions -- of matrices whose dimensions are all bounded by $T$.  (Note that back-substituting to solve for $\mathcal{A}_{\ell}$ involves $\mathcal{O}(L-\ell)$ such matrix operations.)
\end{proof}

We remark that condition (\textit{iv}), which just involves the choice of radial bases, appears to always hold for PSWFs using the cutoffs proposed in Appendix \ref{app:pswf}.  Conditions (\textit{i}), ($ii$) and (\textit{iii}) just involve the distribution, and are full-rank, spectral and non-vanishing hypotheses.  Condition (\textit{iv}) just involves the molecule and in particular requires $S(L) \geq 2L+1$, which limits $L$ to be less than the Nyquist frequency where $S(L_{\textup{Nyquist}}) = 1$.  

Our algorithm goes by \textit{reverse}\footnote{Reverse frequency marching is natural given the sparsity structure of \eqref{eqn:second_mom_compacy}: only $\mathcal{A}_{\ell_1}$ and $\mathcal{A}_{\ell_2}$ with $\ell_{1} \geq |q_1|$, $\ell_{1} \equiv q_{1} \,\, (\textup{mod }2)$ and $\ell_{2} \geq |q_2|$, $\ell_{2} \equiv q_{2} \,\, (\textup{mod }2)$ appear in the moments $m_{2}(q_{1}, q_{2})$.} frequency marching, as we solve for top-frequency coefficients from the second moment \eqref{eqn:second_mom_compacy} where $q_1, q_2 = \pm L, \pm (L-1)$ via eigenvectors (similar to simultaneous diagonalization in Jennrich's algorithm), and then solving for lower-frequency coefficients via linear systems.   While our conditions in Theorem~\ref{thm:known_rho} are certainly not necessary, fortunately for generic\footnote{\noindent This means generic with respect to the \textit{Zariski topology} \cite{book-harris}.  Equivalently, there is a non-zero polynomial $p$ in $A,B$ such that $p(A,B) \neq 0$ implies the conditions in Theorem~\ref{thm:known_rho} are met.} $(A,B)$, those conditions are satisfied, so that the method applies:

\begin{lemma}\label{lem:hay}
Condition (\textit{ii}) in Theorem \ref{thm:known_rho} holds for Zariski-generic $\{B_{p,u,v}\}$.  If $S(L) \geq 2L+1$, then condition (\textit{iii}) holds for Zariski-generic $\{A_{\ell, m, s} \}$.  At least for $L \leq 100$, conditions (\textit{i}) and ($iii$) hold for Zariski-generic $\{B_{p,u,v}\}$.
\end{lemma}
\begin{proof}[Proof of Lemma \ref{lem:hay}]
	Conditions ($i$)-($iv$) are all Zariski-open, \ie, their failure implies $\{A_{\ell,m,s}\}$ or $\{B_{p,u,v}\}$ obey polynomial equations.  As such, to conclude genericity, it suffices to exhibit a single point $\{A_{\ell,m,s}\}$ or $\{B_{p,u,v}\}$, where the conditions are met.  For conditions ($i$), ($iii$), we verified the conditions hold at randomly selected points on computer up to $L \leq 100$.  Conditions ($ii$) and ($iv$) are obviously generic.
\end{proof}

By uniqueness, $A$ is a well-defined function of the first and second moments $m_1$ and $m_2$ almost everywhere.  It is useful to quantify the ``sensitivity'' of $A$ to errors in $m_1, m_2$, as, \eg, in practice one can access only empirical estimates $\widetilde{m}_{1}$ and $\widetilde{m}_{2}$.   An \textit{a posteriori} (absolute) condition number for $A$ is given by the reciprocal of the least singular value of the Jacobian matrix of the algebraic map:
\begin{equation}\label{eqn:poly-map}
{m}_{B}: \,\, \{A_{\ell, m, s} \} \mapsto \big{\{}m_1(q), \, m_2(q_1,q_2)\big{\}}.
\end{equation}
Throughout this section, all condition formulas are in the sense of  \cite[Section~14.3]{condition-book}, for which the domain and image of our moment maps are viewed as Riemannian manifolds.  To this end, when $\rho$ is unknown, dense open subsets of the orbit spaces $\{(A,B) \,\,  \textup{mod SO}(3) \}$, $\{A \,\,  \textup{mod SO}(3) \}$, $\{B \, \, \textup{mod SO}(3) \}$ naturally identify as Riemmannian manifolds (for the construction, see \cite{group-book}).

\subsection{Known, in-plane uniform \texorpdfstring{$\rho$}{distribution}} \label{sec:uni2}
For this case, given a particular image size (and other image parameters), together with band limits $L$ and $P$, we have code\footnote{Available in GitHub: https://github.com/nirsharon/nonuniformMoM/JacobianTest.} which decides if, for generic $A$ and $B$, the molecule $A$ is determined by \eqref{eqn:first_mom_compacy} and \eqref{eqn:second_mom_compacy}, up to finitely many solutions.  The basis for this code is the so-called \textit{Jacobian test} for algebraic maps, see Appendix~\ref{sec:appendix_Jacobian}.  Below is an illustrative computation.
\begin{comp}\label{comp:jacobian}
Consider $43 \times 43$ pixel images, and the following parameters for prolates (representative values): a bandlimit $c$ (see Appendix~\ref{app:pswf}) chosen as the Nyquist frequency, 2-D prescribed accuracy~\eqref{eqn:epsilon_eq} set to $\epsilon = 10^{-3}$ and 3-D truncation parameter~\eqref{eq:PSWFs truncation rule} to be $\delta = 0.99$~\footnote{The value of $\delta$ means we allow only $1\%$ of the energy to be outside the ball, and is chosen to best model a molecule structure which is assumed to be mostly supported inside a ball.}. We varied band limits $L$ in \eqref{eqn:phi_hat} and $P$ in \eqref{eqn:distribution_extansion}, and randomly fixed \eqref{eqn:distribution_extansion} to give a known in-plane uniform distribution. For each $(L,P)$, we computed the numerical rank of the Jacobian matrix of the polynomial map $m_B$ of \eqref{eqn:poly-map} at a randomly chosen $A$, with random $B$. The Jacobian was convincingly of full numerical rank for a variety of band limits, as seen in Table~\ref{table:jac-kn-inp}. Cases where the gap between the two least singular values of the Jacobian matrix exceeds a threshold of $10^6$ are set as indecisive numerics, and appears in the table as \qmark. Note that if the rank was calculated in exact arithmetic, this gives a proof that for generic $(A,B)$ generic fibers of the map $m_{B}$ consist of finitely many $A$; \ie, first and second moments (with known in-plane uniform distribution) determine the molecule up to finitely many solutions. For fibers and related definitions, see Appendix~\ref{sec:appendix_Jacobian}.
\end{comp}

Again, the sensitivity of $A$ as a locally defined function of~\eqref{eqn:first_mom_compacy} and~\eqref{eqn:second_mom_compacy} is quantified by the reciprocal of the least singular value of the Jacobian matrix of $m_{B}$.

\subsection{Unknown, totally non-uniform \texorpdfstring{$\rho$}{distribution}} \label{sec:uni3}
In this case, it is important to note that solutions come in symmetry classes.  If $(A,B)$ have specified moments, then so too for $(R \cdot A, R \cdot B)$ for all $R \in \textup{SO}(3)$, that is, we may jointly rotate the molecule and probability distribution and the moments are left invariant.  So, solutions come in $3$-dimensional equivalence classes, and we are interested in solutions modulo $\textup{SO}(3)$.

That said, we have code which accepts a particular image size (and other image parameters), together with band limits $L$ and $P$.  The code then numerically decides which of the following situations occur: \textit{i)} for generic $(A,B)$, both $A$ and $B$ are determined by \eqref{eqn:first_mom_compacy} and \eqref{eqn:second_mom_compacy} up to finitely many solutions modulo $\textup{SO}(3)$; \textit{ii)} for generic $(A,B)$, the molecule $A$ is determined by \eqref{eqn:first_mom_compacy} and \eqref{eqn:second_mom_compacy} up to finitely many solutions modulo $\textup{SO}(3)$, whereas the distribution $B$ has infinitely many solutions; \textit{iii)} for generic $(A,B)$, both $A$ and $B$ have infinitely many solutions modulo $\textup{SO}(3)$.  Note these cases are (essentially) exhaustive, since if $B$ is determined so is $A$ in the regime of Theorem \ref{thm:known_rho}.  Moreover, we noticed the case \textit{ii)} really does arise, \eg, this seems to happen when $P=2L$.
\begin{comp}\label{comp:jacobian2}
We keep the running example of $43 \times 43$ pixel images, and the prolates parameters of a bandlimit $c$ chosen as the Nyquist frequency, 2-D prescribed accuracy~\eqref{eqn:epsilon_eq} set to $\epsilon = 10^{-3}$ and 3-D truncation parameter~\eqref{eq:PSWFs truncation rule} of $\delta = 0.99$. We varied band limits $L$ in \eqref{eqn:phi_hat} 
and $P$ in \eqref{eqn:distribution_extansion}. For each $(L,P)$, we computed the numerical rank of the Jacobian matrix of the polynomial map
\begin{equation}\label{eqn:moment-map2}
{m}: \,\, \{A_{\ell, m, s}, B_{p, u, v} \} \mapsto \big{\{}m_1(q), \, m_2(q_1,q_2)\big{\}}.
\end{equation}
at a randomly chosen point in the domain.  The numerical rank of the Jacobian convincingly equaled three less (that is $d_1 = 3$, see Appendix~\ref{sec:appendix_Jacobian}) than full column rank for a variety of band limits, see Table \ref{table:jac-unk-tot}. Cases where the gap between the third and fourth least singular values of the Jacobian matrix exceeds a threshold of $10^6$ are set as indecisive numerics, and appears in the table as \qmark. If the rank were calculated in exact arithmetic, this furnishes a proof that generic fibers of the map $m$ consist of finitely many $\textup{SO}(3)$-orbits; that is, first and second moments determine both the molecule and the totally non-uniform distribution up to finitely many solutions (modulo global rotation).
\end{comp}

For band limits $L$ and $P$ such that generically there are only finitely many solutions for $(A,B)$ mod $\textup{SO}(3)$, the sensitivity of $(A,B) \textup{ mod SO}(3)$ as a (locally defined) function of \eqref{eqn:first_mom_compacy} and \eqref{eqn:second_mom_compacy}  is quantified by the reciprocal of the fourth least singular of $m$.
For band limits such that generically there are only finitely many solutions for $A$ mod $\textup{SO}(3)$, the sensitivity of $A \textup{ mod SO}(3)$ as a locally defined of  \eqref{eqn:first_mom_compacy} and \eqref{eqn:second_mom_compacy} is quantified by the reciprocal of the fourth least singular value of
\begin{equation} \label{eqn:Jacobian}
\mathcal{P}_{A}\,  \textup{Jac}(m \rvert_{(A,B)})^{\dagger}
\end{equation}
where $\dagger$ denotes pseudo-inverse and $\mathcal{P}_{A}$ is the differential of $(A,B) \mapsto A \textup{ mod SO}(3)$. We compute~\eqref{eqn:Jacobian} by analytically differentiating~\eqref{eqn:first_mom_compacy} and~\eqref{eqn:second_mom_compacy}, evaluating at $(A,B)$ and place as diagonal blocks of a matrix, and finally applying pseudo-inverse which is SVD-based.

\subsection{Unknown, in-plane uniform \texorpdfstring{$\rho$}{distribution}} \label{sec:uni4}
Again in this case, solutions come in $3$-symmetry classes, orbits under the action of global rotation, so we are interested in solutions modulo $\textup{SO}(3)$.
We have code which accepts a particular image size (and other image parameters), together with band limits $L$ and $P$, and numerically decides if for generic $(A,B)$, both $A$ and $B$ are determined by \eqref{eqn:first_mom_compacy} and \eqref{eqn:second_mom_compacy} up to finitely many solutions modulo $\textup{SO}(3)$, or if there are infinitely many solutions.  We did not find parameters giving a ``mixed'' result as in case \textit{ii)} above.  
\begin{comp}\label{comp:jacobian3}
For $43 \times 43$ pixel images, and the parameters for prolates (representative values): a bandlimit $c$ chosen as the Nyquist frequency, 2-D prescribed accuracy~\eqref{eqn:epsilon_eq} set to $\epsilon = 10^{-3}$ and 3-D truncation parameter~\eqref{eq:PSWFs truncation rule} of $\delta = 0.99$. We varied band limits $L$ in \eqref{eqn:phi_hat} 
and $P$ in \eqref{eqn:distribution_extansion}, restricting \eqref{eqn:distribution_extansion} to an in-plane uniform distribution. For each $(L,P)$, we computed the numerical rank of the Jacobian matrix of the polynomial map:
\begin{equation}\label{eqn:moment-map3}
{m}: \,\, \{A_{\ell, m, s}, \, B_{p, u, 0} \} \,  \mapsto \big{\{}m_1(q), \, m_2(q_1,q_2)\big{\}}.
\end{equation}
at a randomly chosen point in the domain.  The numerical rank of the Jacobian convincingly equaled three less than full column rank for a variety of band limits, see Table \ref{table:jac-ukn-inp}. Cases where the gap between the third and fourth least singular values of the Jacobian matrix exceeds a threshold of $10^6$ are set as indecisive numerics, and appears in the table as \qmark. If the rank was calculated in exact arithmetic, this furnishes a proof that generic fibers of the map $m$ consist of finitely many $\textup{SO}(3)$-orbits; that is, first and second moments determine both the molecule and the in-plane uniform distribution up to finitely many solutions (modulo global rotation).
\end{comp}

For band limits $L$ and $P$ such that generically there are only finitely many solutions for $(A,B)$ mod $\textup{SO}(3)$, the sensitivity of $(A,B) \textup{ mod SO}(3)$ as a  function of moments  is quantified by the reciprocal of the fourth least singular of $m$. For example, in the $P=2$ row of Table \ref{table:jac-ukn-inp}, when evaluating at random $(A,B)$, this worked out to: 
\[ 1.98 \times 10^{15}, \quad 47.1,   \quad  209,    \quad  2700, \quad  4.66 \times 10^4, \quad 1.17 \times 10^6,  \quad 6.02 \times 10^7,  \quad 9.10 \times 10^8 . \]
Further, in the $L=4$ column of Table 1, evaluating at random $(A,B)$ gave:
\[ 1.44 \times 10^{16}, \quad 2.15 \times 10^{15}, \quad 209,  \quad 154, \quad 1360. \]
In practice, we run this refined Jacobian test (takes $< 1$ minute on a standard laptop) to identify well-conditioned band limits $L$ and $P$ before we attempt non-convex optimization.

\begin{table}[ht]
    \centering
    \caption{\textit{Uniqueness for inverting the first two moments in the case of a known, in-plane uniform $\rho$, according to band limits.}  Generically finitely many solutions for $A$ is denoted by \cmark, infinitely many solutions for $A$ is denoted by \xmark, and indecisive numerics is denoted by \qmark.} 
    \vspace{0.5em}
    \renewcommand{\arraystretch}{1.2}
    \begin{tabular}{l c c c c c c c c c}
    \toprule
         & $L = 2$  & $L = 3$ & $L = 4$ & $L = 5$ & $L = 6$ & $L=7$ & $L=8$ & $L=9$ & $L=10$\\
    \midrule
    $ P  = 0$ & \xmark  & \xmark & \xmark & \xmark & \xmark & \xmark & \xmark & \xmark & \xmark\\
    $P  = 1$ & \xmark  & \xmark & \qmark & \cmark & \cmark & \cmark & \cmark & \cmark & \cmark\\
    $ P = 2$ & \xmark  & \cmark & \cmark & \cmark & \cmark & \cmark & \cmark & \cmark & \cmark \\
    $ P  = 3$ & \xmark  & \cmark & \cmark & \cmark & \cmark & \cmark & \cmark & \cmark & \cmark\\
    $ P  = 4$ & \xmark  & \cmark & \cmark & \cmark & \cmark & \cmark & \cmark & \cmark & \cmark\\
    \bottomrule
    \end{tabular}
    \label{table:jac-kn-inp}
\end{table}

\begin{table}[ht]
    \centering
     \caption{\textit{Uniqueness for inverting the first two moments in the case of an unknown, totally non-uniform $\rho$, according to band limits.}  Generically finitely many solutions for $(A,B) \textup{ mod SO(3)}$ is denoted by \cmark, finitely many solutions for $A \textup{ mod SO(3)}$ but infinitely many solutions for $B \textup{ mod SO(3)}$ is denoted by \tmark, infinitely many solutions for $A \textup{ mod SO(3)}$ is denoted by \xmark, and indecisive numerics is denoted by \qmark.}
    \vspace{0.5em}
    \renewcommand{\arraystretch}{1.2}
    \begin{tabular}{l c c c c c c c c c}
    \toprule
         & $L = 2$  & $L = 3$ & $L = 4$ & $L = 5$ & $L = 6$ & $L=7$ & $L=8$ & $L=9$ & $L=10$\\
    \midrule
    $ P  = 0$ & \xmark  & \xmark & \xmark & \xmark & \xmark & \xmark & \xmark & \xmark & \xmark\\
    $P  = 1$ & \cmark  & \cmark & \cmark & \cmark & \cmark & \cmark & \cmark & \cmark & \qmark\\
    $ P = 2$ & \cmark  & \cmark & \cmark & \cmark & \cmark & \cmark & \cmark & \cmark & \cmark \\
    $ P  = 3$ & \cmark  & \cmark & \cmark & \cmark & \cmark & \cmark & \cmark & \cmark & \cmark\\
    $ P  = 4$ & \tmark  & \cmark & \cmark & \cmark & \cmark & \cmark & \cmark & \cmark & \cmark\\
    \bottomrule
    \end{tabular}
    \label{table:jac-unk-tot}
\end{table}

\begin{table}[ht]
    \centering
     \caption{\textit{Uniqueness for inverting the first two moments in the case of an unknown, in-plane uniform $\rho$, according to band limits.}  Generically finitely many solutions for $(A,B) \textup{ mod SO(3)}$ is denoted by \cmark, infinitely many solutions for $A \textup{ mod SO(3)}$ and $B \textup{ mod SO(3)}$ is denoted by \xmark, and indecisive numerics is denoted by \qmark.}
    \vspace{0.5em}
    \renewcommand{\arraystretch}{1.2}
    \begin{tabular}{l c c c c c c c c c}
    \toprule
         & $L = 2$  & $L = 3$ & $L = 4$ & $L = 5$ & $L = 6$ & $L=7$ & $L=8$ & $L=9$ & $L=10$\\
    \midrule
    $ P  = 0$ & \xmark  & \xmark & \xmark & \xmark & \xmark & \xmark & \xmark & \xmark & \xmark\\
    $P  = 1$ & \xmark  & \xmark & \xmark & \xmark & \xmark & \xmark & \xmark & \xmark & \xmark\\
    $ P = 2$ & \xmark  & \cmark & \cmark & \cmark & \cmark & \cmark & \qmark & \qmark & \qmark \\
    $ P  = 3$ & \xmark  & \cmark & \cmark & \cmark & \cmark & \cmark & \cmark & \qmark & \qmark\\
    $ P  = 4$ & \xmark  & \cmark & \cmark & \cmark & \cmark & \cmark & \cmark & \qmark & \qmark\\
    \bottomrule
    \end{tabular}
    \label{table:jac-ukn-inp}
\end{table}

\section{Numerical Optimization and First Visual Examples} \label{sec: implementation details}

After studying the theoretical properties of the polynomial system which is defined by the first two moments, we discuss in this section aspects of numerically inverting the polynomial map via optimization.  

\subsection{Incorporating natural constraints in optimization} \label{subse:constraints_in_optimization}

When determining the coefficients $A = \{A_{\ell,m,s}\}_{\ell,m,s}$  and $B = \{B_{p,u,v}\}_{p,u,v}$, the search space has to be restricted in order to ensure the coefficients stem from some physical volume and density. 
\subsubsection{Constraints on the volume}
To ensure the volume $\phi\colon \mathbb{R}^3\rightarrow \mathbb{R}$ is a real-valued function, one has to ensure its Fourier transformation $\widehat \phi: \mathbb{R}^3 \rightarrow \mathbb{C}$ satisfies conjugate symmetry $\hat{\phi}(\kappa,\theta,\varphi) = \bar{\hat{\phi}}(\kappa,\pi-\theta,\pi+\varphi)$. That is, in spherical coordinates, 
\begin{equation*}
\sum_{\ell=0}^L \sum_{m=-\ell}^\ell \sum_{s=1}^{S(\ell)}  \overline{A_{\ell,m,s} Y^m_\ell(\theta,\varphi)F_{\ell,s}(\kappa)} =
\sum_{\ell=0}^L \sum_{m=-\ell}^\ell \sum_{s=1}^{S(\ell)} A_{\ell,m,s} Y^m_\ell(\pi-\theta,\pi+\varphi)F_{\ell,s}(\kappa) .
\end{equation*}
Assuming the basis $\{F_{\ell,s}\}$ is a set of real-valued functions, along with the facts that
$ \overline{Y^m_\ell(\theta,\varphi)} = (-1)^m Y^{-m}_\ell(\theta,\varphi) $ and 
$ Y^m_\ell(\pi-\theta,\pi+\phi) = (-1)^\ell Y^{m}_\ell(\theta,\phi)$, 
we get
\begin{equation*}
\sum_{\ell,m,s} \overline{A_{\ell,-m,s}} (-1)^{-m} Y_\ell^{-m}(\theta,\varphi)F_{\ell,s}= \sum_{\ell,m,s} A_{\ell,m,s}  (-1)^\ell Y^m_\ell(\theta,\varphi) F_{\ell,s}
\end{equation*}
This further implies
\begin{equation} \label{eqn:reality_cons}
\overline{A_{\ell,m,s}} (-1)^{-m} = A_{\ell,-m,s} (-1)^\ell.
\end{equation}
Having such relationships, $\{A_{\ell,m,s}\}_{\ell,m,s}$ can thus be written in terms of some real coefficients $\{\alpha_{\ell,m,s}\}_{\ell,m,s}$ as:
\begin{equation} \label{eqn:real_A_coefs}
A_{\ell,m,s} = 
\begin{cases}
\alpha_{\ell,m,s}  - i(-1)^{l+m}\alpha_{\ell,m,s}, & m>0,\\
i^l \alpha_{\ell,m,s} & m=0,\\
(-1)^{l+m} \alpha_{\ell,m,s}+ i\alpha_{\ell,m,s}, & m<0.\\
\end{cases}
\end{equation}
The latter means that instead of solving a complex optimization problem in terms of the coefficients $A_{\ell,m,s}$, one can work with the real coefficients $\alpha_{\ell,m,s}$ of ~\eqref{eqn:real_A_coefs}. Otherwise, the equality constraints~\eqref{eqn:reality_cons} are required.
\subsubsection{Constraints on the density}
Similarly, to ensure the density $\rho$ being a real-valued function, we need to ensure 
\begin{equation}
\sum_{p=0}^P \sum_{u=-p}^p\sum_{v=-p}^p   B_{p,u,v} U^p_{u,v}(R) = \sum_{p=0}^P \sum_{u=-p}^p\sum_{v=-p}^p  \overline{B_{p,u,v} U^p_{u,v}(R)}.
\end{equation}
The fact that $ \overline{U^p_{u,v}(R)} = (-1)^{v-u}U^p_{-u,-v}(R)$ leads to
\begin{equation} \label{eqn:reality_cons_rho}
B_{p,u,v} = (-1)^{u-v} \overline{B_{p,-u,-v}}.
\end{equation}
Again, from such relationships, it can be shown that an alternative to~\eqref{eqn:reality_cons_rho} can be written in terms of real coefficients $\beta_{p,u,v}$:
\begin{equation}
B_{p,u,v} = 
\begin{cases}
\beta_{p,u,v} + (-1)^{u-v} i \beta_{p,-u,-v}, &(u,v) \succ_{\textup{lex}} (0,0), \\
\beta_{p,0,0}, &(u,v) = 0, \\
\beta_{p,u,v} - (-1)^{u-v} i \beta_{p,-u,-v}, &(u,v) \prec_{\textup{lex}} (0,0).
\end{cases}
\end{equation}
Here, $\prec_{\textup{lex}}$ is the lexicographical order, that is 
$(u_1,v_1) \prec_{\textup{lex}} (u_2,v_2)$ iff $u_1<u_2$ or both $u_1=u_2$ and $v_1<v_2$.

Two additional constraints are required. First, the integral of any density function is one. To ensure such a correct normalization, we simply let
\begin{equation}\label{eq:B_norm}
B_{0,0,0} = \int \sum_{p=0}^P \sum_{u=-p}^p \sum_{v=-p}^p B_{p,u,v} U^p_{u,v}(R) dR = 1,
\end{equation}
which means it is no longer considered as unknown. Finally, the nonnegativity of the density is ensured via a collocation method, that is requiring
\begin{equation} \label{eqn:positivity_cons}
\rho(R_i) = \sum_{p,u,v} B_{p,u,v} U^p_{u,v}(R_i)\geq 0 ,
\end{equation}
for $R_i$'s on a near uniform, refined grid on $\so{3}$. While~\eqref{eqn:positivity_cons} does not prevent the density from becoming negative off the $\so{3}$ grid, requiring the density to be non-negative entirely on $\so{3}$ leads to an optimization problem that is much more costly to solve in practice.  Note that we do not enforce positivity of $\rho$ by requiring it to be a sum-of-squares, as, \eg, already in the case of an in-plane uniform distribution on the sphere $S^2 \subset \mathbb{R}^3$, not all nonnegative polynomials may be written as a sum-of-squares, see Motzkin's example when $P=6$ \cite{Motzkin}.

\subsection{Accommodating invariance to in-plane rotations}\label{sec:in-planeCon}
While molecules typically exhibit preferred orientations, there is no physical reason why molecules should have preferred \textit{in-plane} orientations. In this section, we focus on the case of non-uniform rotational distributions invariant to in-plane rotations since these distributions better model real cryo-EM data sets.

For simplicity, we fix the image plane as perpendicular to the $z$-axis. We add the prior that the density for drawing $R$ equals the density for drawing $R z(\alpha)$, for all $R \in \so{3}$ and all rotations $z(\alpha)$ of $\alpha \in \mathbb{R}$ radians about the $z$-axis.  This assumption reads
\begin{equation}
\rho(R) = \rho\left( Rz(\alpha)\right) \quad  R \in \so{3}, \quad \alpha \in \mathbb{R}.
\end{equation}
Therefore,
\begin{align}
\sum_{p,u,v} B_{p,u,v} \, U^{p}_{uv}(R) \,  &= \, \sum_{p,u,v} B_{p,u,v} \, U^{p}_{uv}(Rz(\alpha)) \\
&= \, \sum_{p,u,v} B_{p,u,v} \, \Big{(}U^{p}(R)U^{p}\big{(}z(\alpha)\big{)}\Big{)}_{uv} 
\end{align}
Here we used the group representation property of $U^{p}$.  Checking explicitly the action of $z(\alpha)$ on degree $p$ spherical harmonics,
\begin{equation}
U^{p}\big{(}z(\alpha)\big{)} = \textup{diag}(e^{-ip\alpha}, e^{-i(p-1)\alpha}, \ldots, e^{ip\alpha}).
\end{equation}
So continuing the above,
\begin{equation}
\sum_{p,u,v} B_{p,u,v} \, U^{p}_{uv}(R) =  \sum_{p,u,v} B_{p,u,v} \, U^{p}_{uv}(R) e^{iv\alpha}, \quad R \in \so{3}, \quad \alpha \in \mathbb{R}.
\end{equation}
This is equivalent to $B_{p,u,v} = 0$ for $v \neq 0$ where $v$ ranges over $-p,-p+1, \ldots, p$.  To sum, we have found that in-plane invariance is captured by:
\begin{equation}
d\rho(R) = \sum_{p, u} B_{p,u,0} \, U^{p}_{u0}(R) \, dR 
\end{equation}
For a sanity check, a distribution with in-plane invariance should sample a rotation with density only depending on which point maps to the north pole. Namely, $\rho(R)$ should only depend on the last column of $R$, that is, $R(:,3) = R_{\bullet3}$. Indeed, this holds as $U^{p}_{u0}(R) = (-1)^{u} \sqrt{\frac{4\pi}{2l+1}} \overline{Y^{u}_{p}}(R_{\bullet3})$ \cite[Eqn. 9.44, Pg. 342]{chirikjian2016harmonic}.

Restricting the expansion of $\rho$ as above, we easily see the first moment is independent of $\varphi$.  It is now merely a linear combination of basis functions $F_{\ell,s}(\kappa)$. Likewise, for the second moment, angular dependency is only on the difference $\varphi_{1} - \varphi_{2}$, meaning it is a linear combination of basis functions $e^{i m (\varphi_{1} - \varphi_{2})} F_{\ell_1,s_{1}}(\kappa_{1}) F_{\ell_2,s_{2}}(\kappa_{2})$. Thus, in subsection~\eqref{sec:uni4}, we have the following polynomial map, now with fewer $B$ variables and fewer invariants than in subsection~\eqref{sec:uni3}
\begin{equation}
    m \colon \{A_{\ell, m, s}, B_{p,u,0} \} \mapsto \{ \, \mathbb{E}_R[a^R_{0,t}] , \quad \mathbb{E}_R[a^R_{q_1,t_1}a^R_{-q_1,t_2}]  \, \}    . 
\end{equation}

\subsection{Direct method -- known totally non-uniform distribution}
 
 For the ``easy'' case of a known, totally non-uniform distribution, we have implemented the provable algorithm in Theorem~\ref{thm:known_rho}.  The method's performance is illustrated by way of an example.  
 As the ground truth volume, we use EMD-0409, that is, the catalytic subunit of protein kinase A bound to ATP and IP20~\cite{herzik2019high}, as presented at the online cryo-EM data-bank~\cite{lawson2015emdatabank}. 
 The volumetric array's original dimension is $128$ voxels in each direction, which we downsampled by a factor of three to $43$.  
 The volume was expanded using PSWFs with a band limit $c$ chosen to be the Nyquist frequency and 3-D truncation parameter~\eqref{eq:PSWFs truncation rule} of $\delta = 0.99$. Before downsampling, the full expansion consists of degree $L = 40$; with downsampling and proper truncation, we aim to recover the terms up to degree $L = 7$.  For the known totally non-uniform distribution, we took $P=14$ (per Theorem~\ref{thm:known_rho}), and then formed a particular distribution using a sums-of-squares.  Precisely, we formed a random linear combination of Wigner entries up to degree $7$, multiplied this by its complex conjugate, invoked~\eqref{eqn:wigner_conj} and~\eqref{eqn:C_asProductOfCGcoefs} to rewrite the result as a linear combination of Wigner entries up to degree $14$, repeated for a second square, added, and finally normalized to satisfy~\eqref{eq:B_norm}.  Then, with the distribution known as such, the volume contributes $1080$ unknowns (without discounting for \eqref{eqn:reality_cons}).  Providing the algorithm with $m_1$ and $m_2$, our method took $0.24$ seconds on a standard laptop, and recovered the unknowns $A$ up to a relative error in $\mathcal{L}^{2}$ norm of $5.4 \times 10^{-11}$.  Visual results are in Figure~\ref{fig:known_dist}.   
 
 \begin{figure}[ht]
 	\begin{center}
 		\includegraphics[width=.4\linewidth]{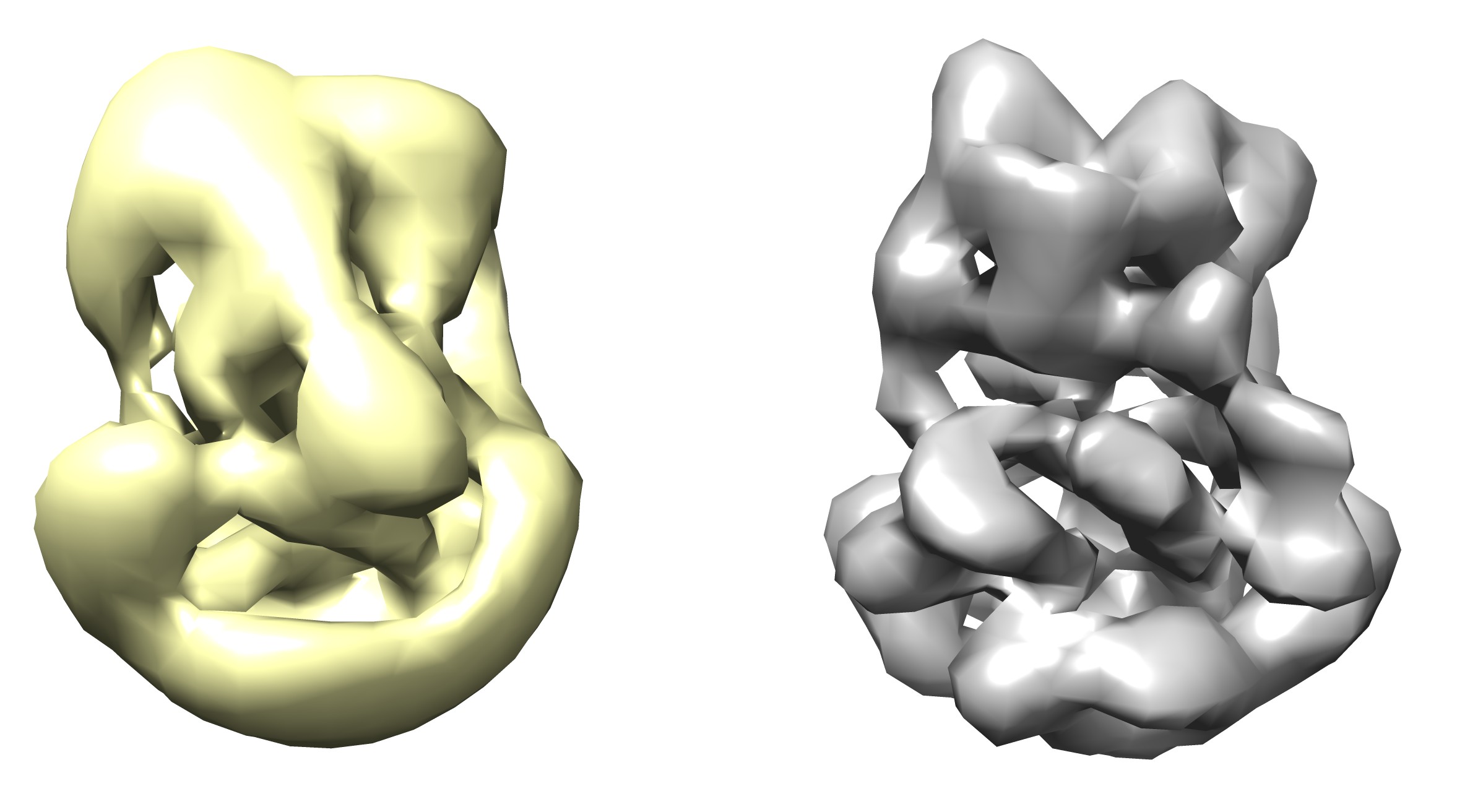} \qquad
 		\includegraphics[width=.4\linewidth]{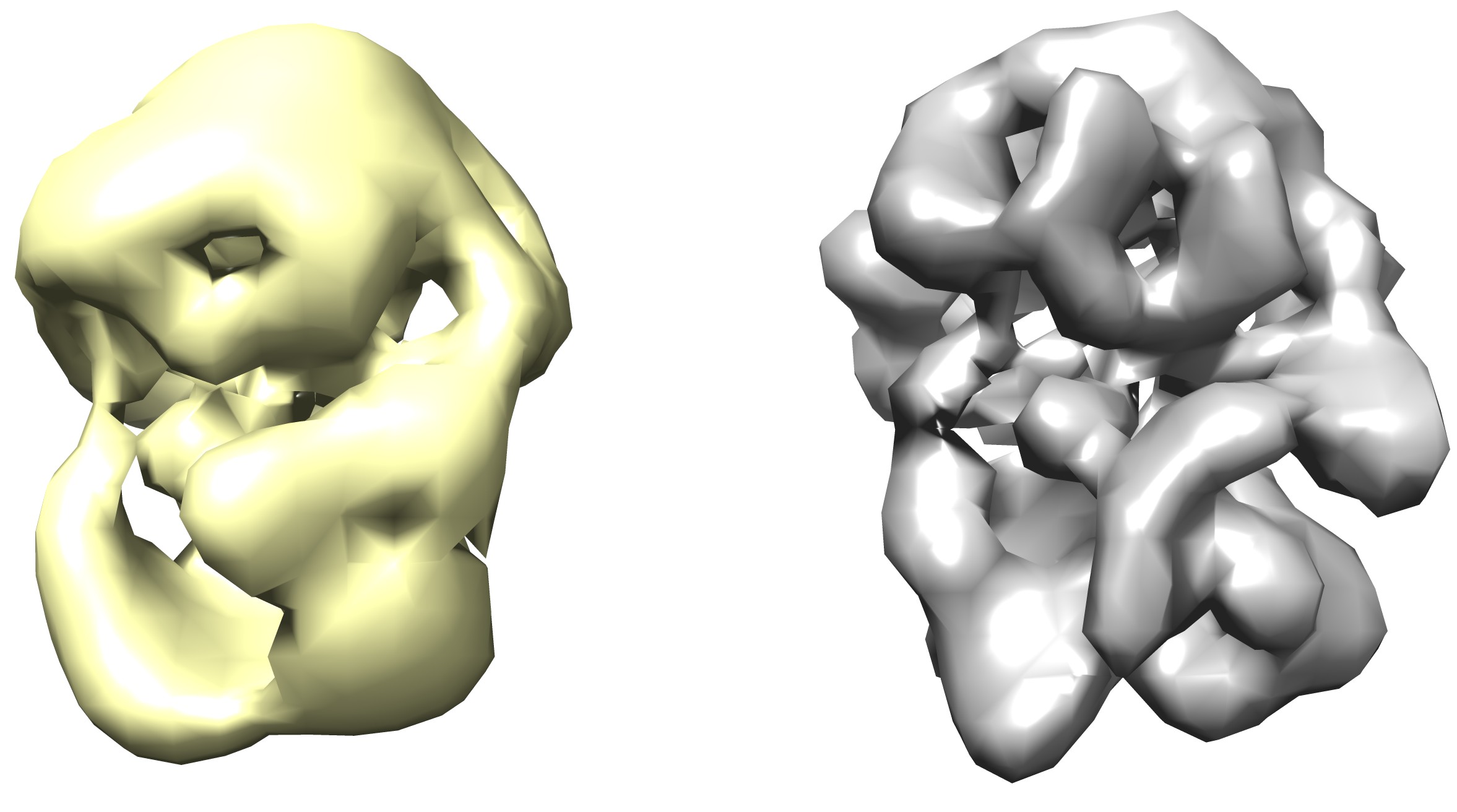} \\
 		View 1 \qquad  \qquad \qquad \qquad \qquad  \qquad \qquad View 2 
 		\caption{Two views of the reconstruction as provided by the algorithm of Theorem~\ref{thm:known_rho} to the case of known, totally non-uniform distribution. The ground truth volume appears on the right of each pair (in gray), whereas the lower degree estimation resulting from the downsampled volume appears on the left (in yellow). Note that the estimation is visually identical to the truncated volume, and it thus illustrates the effect of truncation.}   
 		\label{fig:known_dist}
 	\end{center}
 \end{figure}


\subsection{Setting up a least-squares formulation}

For the cases where we lack a direct method, we formulate the problem in terms of minimizing a least-squares cost function. First, we define the unknowns of our optimization process to be the coefficients of the volume $A = \{A_{l,m,s}\}$ and distribution $B = \{B_{p,u,v}\}$. The explicit formulas~\eqref{eqn:first_mom_compacy} and \eqref{eqn:second_mom_compacy} provide means to write the low-order moments~\eqref{eqn:analytical moments} as functions of our unknown coefficients, that is $m_1 = m_1(A,B)$ and $m_2 = m_2(A,B)$.

In practice, given data images, one estimates the low-order statistics using the empirical moments $\widetilde{m}_1$ and $\widetilde{m}_2$ of~\eqref{eqn:empirical_moments}, but now given in PSWFs coordinates
\begin{equation} \label{eqn:empirical_mom}
\left(\widetilde{m}_1 \right)_{q,t} = \frac{1}{n}\sum_{j=1}^n a^j_{q,t} \quad \text{  and  } \quad \left(\widetilde{m}_2 \right)_{q_1,t_1,q_2,t_2} = \frac{1}{n} \sum_{j=1}^n a^j_{q_1,t_1} a^j_{q_2,t_2},
\end{equation}
The connection between the empirical moments and their analytical formulas as functions of our unknowns gives rise to a nonlinear least-squares 
\begin{multline} \label{eqn:LS}
\min_{A,B} \sum_{q=-Q}^Q \sum_{t=0}^{T(q)} \, \left(m_1(A,B)_{q,t} - \left(\widetilde{m}_1 \right)_{q,t} \right)^2 \\
+ \lambda \sum_{q_1,q_2=-Q}^Q \, \sum_{t_1,t_2=0}^{T(q)} \, \left(m_2(A,B)_{q_1,t_1,q_2,t_2} - \left(\widetilde{m}_2 \right)_{q_1,t_1,q_2,t_2} \right)^2 \, ,
\end{multline}
where $\lambda$ is a parameter chosen to balance the errors from both terms. In particular, two main considerations determine the value of $\lambda$. First is the number of elements in each summand. Namely, the second moment includes many more entries than the first moment. Therefore, without the effect of noise, $\lambda$ is set to be the ratio between the number of entries in first moment and the second moment.  The second factor to balance is the different convergence rates of the empirical moments, see also \cite{abbe2018multireference}. The nonlinear least-squares~\eqref{eqn:LS} may be adjusted to incorporate the constraints on $\{A_{l,m,s}\}$ and $\{B_{p,u,v}\}$ that ensure $\phi$ is a real-valued volume and $\rho$ a probability density.

We remark that it is interesting to consider pre-conditioners, or more intricate weighings, in the formation of the nonlinear least-squares cost \eqref{eqn:LS}.
Such might alleviate high condition numbers observed in Section~\ref{sec:algebraic geometry}, and potentially accelerate optimization algorithms.
While we have not tested a pre-conditioner in optimization experiments yet, one possibility would be to consider the following normalized cost:

\begin{multline} \label{eqn:LS1}
\min_{A,B} \sum_{q=-Q}^Q \sum_{t=0}^{T(q)} \, \left(m_1(A,B)_{q,t} - \left(\widetilde{m}_1 \right)_{q,t} \right)^2 \Big{/} \, \left( \widetilde{m}_1 \right)_{q,t}^2  \\
+ \lambda \sum_{q_1,q_2=-Q}^Q \, \sum_{t_1,t_2=0}^{T(q)} \, \left(m_2(A,B)_{q_1,t_1,q_2,t_2} - \left(\widetilde{m}_2 \right)_{q_1,t_1,q_2,t_2} \right)^2  \Big{/} \, \left( \widetilde{m}_2 \right)_{q_1, t_1, q_2, t_2}^2.
\end{multline}
Effectively, \eqref{eqn:LS1} scales each polynomial in $(A,B)$ given by $m_1$ and $m_2$ to take value \nolinebreak $1$.

\subsection{Complexity analysis of inverting the moments via gradient-based optimization} \label{subsec:complexity}

Before moving forward to further numerical examples, we state the computational load of minimizing the least-squares cost function~\eqref{eqn:LS}. It is worth noting that in many modern \textit{ab initio} algorithms, like SGD~\cite{punjani2017cryosparc} and EM~\cite{scheres2012relion}, the runtime of each iteration is measured with respect to the size of the set of data images, which can be huge. In our approach, we only carry out one pass over the data to collect the low-order statistics. In here, we assume the empirical moments are already given, and so the complexity of each iteration is merely a function of the size of the moments or equivalently depends on the size and resolution of the data images, as reflected by their PSWF representations.

Many possible algorithms exist to minimize the least squares problem~\eqref{eqn:LS}, for example direct gradient descent methods, such as trust-region~\cite{powell1990trust}, or alternating approaches, including alternating stochastic gradient descent. Here, we present the complexity of evaluating the cost function and its gradient, regardless of the specific algorithm or implementation one wishes to exploit.

For simplicity, denote by $S$ and $T$ two bounds for the radial indices $S(\ell)$ and $T(q)$ of the 3-D and 2-D PSWF expansions, respectively. Typically, it is sufficient to take $S=S(0)$ and $T=T(0)$, as radial degree decreases as overall degree ($\ell$) increases. 

Starting from the first moment~\eqref{eqn:first_mom_compacy}: with a fixed $\ell$ we have to apply two matrix-vector products in a row which requires an order of $\mathcal{O}\left( S \ell + TS \right)$ arithmetic operations. The variable $\ell$ increases up to $L$, which sums up to a total of $L \cdot \mathcal{O}\left( S \ell + TS \right) = \mathcal{O}\left( LS (L + T) \right)$. The gradient uses the precomputed remainder $m_1(A,B) - \left(\widetilde{m}_1 \right)$ and is calculated by two terms with similar complexity as the above. Namely, the cost of both evaluation and gradient calculations is again $\mathcal{O}\left( LS (L + T) \right)$.

For the second moment, we follow~\eqref{eqn:second_mom_compacy}: establishing $\Gamma^{q}_{\ell} \, \mathcal{A}_{\ell}$ is done in $\mathcal{O}\left(TSL\right)$ and applying the product in $\mathcal{O}\left(TL^2\right)$. Overall, the evaluation is bounded by 
 \begin{equation}
 \mathcal{O}\left(L^2(TSL + TL^2)\right) = \mathcal{O}\left(TL^3 (S+L)\right). 
 \end{equation}
The gradient is a bit more complicated, in short, there are two terms for the volume derivatives and one term for the distribution part, with the precomputed remainder $m_2(A,B) - \left(\widetilde{m}_2 \right)$ we get an overall complexity of $\mathcal{O}\left(L^2 S (L^2+T^2+TL)\right)$. In summary, the first moment requires third-order complexity with respect to the different parameters where the second moment requires a total power of five.

Finally, the parameters $T$, $S$, and $L$ can be described by the PSWF representation: the length $L$ of the 3-D PSWF expansion and the bound on the radial indices $S$ are related to the parameter $c$ of sampling rate, and are bounded according to~\eqref{eq:Asymptotic number of terms}. Additional bound, now on the radial 2-D expansion $T$, uses the accuracy parameter $\epsilon$ of the 2-D images and the above $L$ as given in~\eqref{eqn:epsilon_eq}. For more details on those parameters, see Appendix~\ref{app:pswf}.
 
\subsection{Remark on using semidefinite programming (SDP) relaxation} 
Solving the nonlinear least-squares problem in Eq. \eqref{eqn:LS} could suffer from slow convergence because the cost function is a polynomial of degree 6. We remark that in principle, it is possible to apply a semidefinite programming relaxation to facilitate the optimization. For convenience, let the second moments $m_2(A,B)_{q_1,t_1,q_2,t_2}$ be summarized as
\begin{equation}
m_2(A,B) _{q_1,t_1,q_2,t_2}:= G_{q_1,t_1,q_2,t_2}(A A^T \otimes B)
\end{equation}
where $G_{q_1,t_1,q_2,t_2}(\cdot)$ is a linear operator that captures the RHS of Eq. \eqref{eqn:second_momv2}. If we define
\begin{equation*}
\bar A = A A^T,
\end{equation*}
the optimization problem can be written as 
\begin{multline*} 
\min_{\substack{A,\bar A, B\\ \bar A = A A^T}} \sum_{q=-Q}^Q \sum_{t=0}^{T(q)} \, \left(m_1(A,B)_{q,t} - \left(\widetilde{m}_1 \right)_{q,t} \right)^2 \\
+ \lambda \sum_{q_1,q_2=-Q}^Q \, \sum_{t_1,t_2=0}^{T(q)} \, \left(G_{q_1,t_1,q_2,t_2}(\bar A \otimes B) - \left(\widetilde{m}_2 \right)_{q_1,t_1,q_2,t_2} \right)^2 \, .
\end{multline*}
To deal with the non-convex constraint $\bar A = A A^T$, we propose the following relaxed constraint 
\begin{equation}\label{eqn:relaxation}
\bar A \succeq AA^T ,
\end{equation}
which gives the following non-linear least squares problem 
\begin{multline} \label{eqn:LS3}
\min_{\substack{A,\bar A, B\\ \bar A \succeq A A^T}} \sum_{q=-Q}^Q \sum_{t=0}^{T(q)} \, \left(m_1(A,B)_{q,t} - \left(\widetilde{m}_1 \right)_{q,t} \right)^2 \\
+ \lambda \sum_{q_1,q_2=-Q}^Q \, \sum_{t_1,t_2=0}^{T(q)} \, \left(G_{q_1,t_1,q_2,t_2}(\bar A \otimes B) - \left(\widetilde{m}_2 \right)_{q_1,t_1,q_2,t_2} \right)^2 \, .
\end{multline}
Comparing with \eqref{eqn:LS}, although \eqref{eqn:LS3} is still a non-convex problem, the degree of the polynomial in the cost function of \eqref{eqn:LS3} is 4 (instead of 6).  Furthermore, one can solve \eqref{eqn:LS3} efficiently by minimizing $(A, \bar A)$ and $B$ in an alternating fashion. Therefore if at the optimum $\bar A \approx AA^T$ in spite of the relaxation \eqref{eqn:relaxation}, solving \eqref{eqn:LS3} can be advantageous. 

We remark on the special case when the density coefficient $B$ is given. In this situation, one can consider an SDP relaxation 
\begin{equation} \label{eqn:nuc_norm}
     \begin{alignat}{2}
	&\!\min_{\substack{A,\bar A,\\ \bar A \succeq A A^T}}        &\qquad& \text{Tr}(\bar A)     \\
	&\text{subject to} &      & \abs{ m_1(A,B)_{q,t} - \left(\widetilde{m}_1 \right)_{q,t} } \leq \epsilon_{q,t}, \quad 0 \leq t \leq T(q),\quad -Q\leq q \leq Q,  \nonumber  \\
	&                  &      & \abs{G_{q_1,t_1,q_2,t_2}(\bar A \otimes B) - \left(\widetilde{m}_2 \right)_{q_1,t_1,q_2,t_2} } \leq \epsilon_{q_1,t_1,q_2,t_2},  \nonumber \\
	 &                  &      &0 \leq t_1 \leq T(q_1),\quad 0 \leq t_2 \leq T(q_2),\quad -Q\leq q_1, q_2 \leq Q. \nonumber
	\end{alignat}
\end{equation}

The nuclear norm minimization strategy as in matrix completion~\cite{candes2010matrix} is used to promote $\bar A$ to be of rank-1. We test the SDP in \eqref{eqn:nuc_norm} when given a fixed $B_0$. We generate $B_0$ for a non-uniform distribution from a 6-th degree nonnegative polynomial over the rotation group, i.e. letting $P=6$. We generate a volume with random coefficients $A_0$ with $L=3$. Noise is added to the moments in the following manner:
\begin{align*} 
 \left(\widetilde{m}_1 \right)_{q,t} & = m_1(A_0,B_0)_{q,t} + \abs{ m_1(A_0,B_0)_{q,t}} z_{q,t}, \\
 \left(\widetilde{m}_2 \right)_{q_1,t_1,q_2,t_2} & = G_{q_1,t_1,q_2,t_2}(A_0 A_0^* \otimes B_0) + \abs{G_{q_1,t_1,q_2,t_2}(A_0 A_0^* \otimes B_0)} z_{q_1,t_1,q_2,t_2} .
 \end{align*}
 Where 
 \[z_{q,t}, z_{q_1,t_1,q_2,t_2} \sim \text{Uniform}[-\epsilon,\epsilon] , \] 
 and
 \[
  0 \leq t \leq T(q),\ -Q\leq q \leq Q,\ 0 \leq t_1 \leq T(q_1),\ 0 \leq t_2 \leq T(q_2),\ -Q\leq q_1,q_2 \leq Q. \]
In this case, we set in~\eqref{eqn:nuc_norm},
\[
\epsilon_{q,t} = \epsilon \abs{ m_1(A_0,B_0)_{q,t}} \quad  \text{  and  } \quad  \epsilon_{q_1,t_1,q_2,t_2} = \epsilon \abs{ G_{q_1,t_1,q_2,t_2}(A_0 A_0^*\otimes B_0)} .
\]
The stability results in recovering $A_0$ are shown in Figure~\ref{fig:sdp_stab}. We ran five simulations for every $\epsilon$ and average the relative error
\begin{equation*}
\mathtt{RE} = \frac{\| \bar A-A_0 A_0^* \|_F}{\|A_0 A_0^*\|_F}.
\end{equation*}
Results show an exact recovery in the noiseless case and slowly increasing in relative error as $\epsilon$ grows. 

 \begin{figure}[ht]
	\begin{center}
		\includegraphics[width=.5\linewidth]{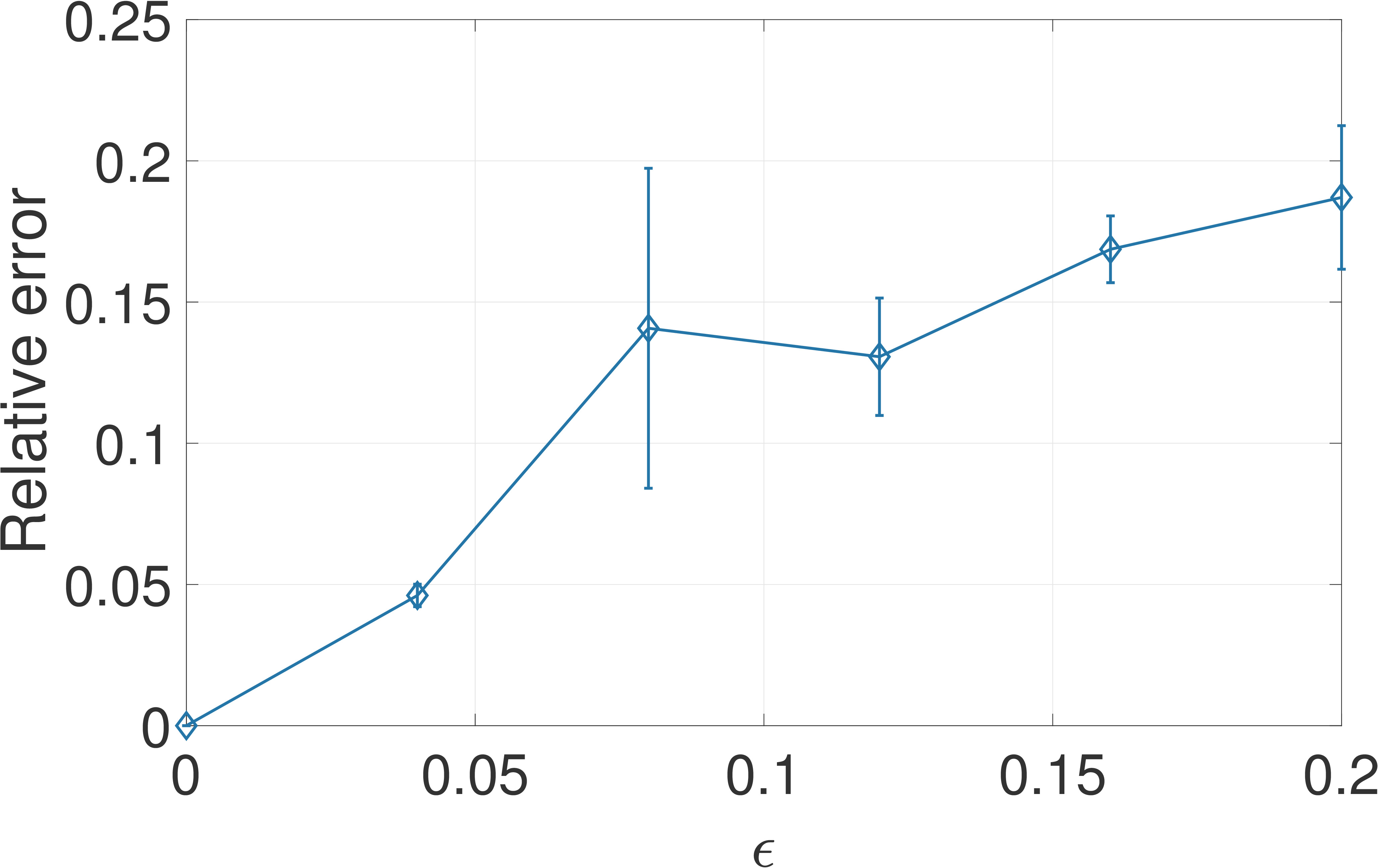} 
		\caption{Stability of the SDP in \eqref{eqn:nuc_norm} when fixing the density to be a non-uniform density. }   
		\label{fig:sdp_stab}
	\end{center}
\end{figure}

\subsection{Volume from moments -- non-uniform vs. uniform} \label{subsec:nonVSuniform}

As a first numerical example, we present a recovery comparison between the cases of uniform and non-uniform distributions of rotations. In this example, we use as a ground truth a low degree approximation of a mixture of six Gaussians, given in a non-symmetric conformation. The approximation, which we ultimately use as our reference, is attained by discretizing the initial volume to $23 \times 23 \times 23$ and truncating the PSWFs expansion to $L=4$. This expansion consists of $118$ coefficients in total. The other PSWFs parameters that we use are a band limit $c$ that corresponds to the Nyquist frequency and 3-D truncation parameter~\eqref{eq:PSWFs truncation rule} of $\delta = 0.99$. The original volume and its approximation appear in Figure~\ref{fig:vols}.

We divide the example into two scenarios of different distributions, uniform and non-uniform. In each case, we start from the analytic moments~\eqref{eqn:analytical moments}, calculated with respect to 2-D prescribed accuracy~\eqref{eqn:epsilon_eq} of $\epsilon = 10^{-3}$, and obtain an estimation based on minimizing the least squares cost function~\eqref{eqn:LS}. The optimization is carried with a gradient-based method, specifically we use an implementation of the \textit{trust-region} algorithm, see \eg,~\cite{powell1990trust}. In the first case, we use as the distribution of rotations a quadratic expansion $P=2$ which is in-plane uniform. Based on the in-plane invariance, we present this distribution as a function on the sphere in Figure~\ref{fig:dist1}. For the second case, we use a uniform distribution of rotations. 

In both cases, we let the optimization reach numerical convergence, where the progress in minimization is minor. In this example, it is usually at about $100-150$ iterations. In the case of non-uniform distribution, we observe that choosing a random initial guess can have an effect on the speed of convergence but has almost no influence on the resulted volume. In other words, we gain numerical evidence for uniqueness. The estimated volume, in this case, is depicted on the left side of Figure~\ref{fig:recovery1}. 

On the other hand, in the case of a uniform distribution, while convergence was typically quicker than in the non-uniform case, the results vary between different initial guesses, indicating the richness of the space of possible solutions. One such solution appears on the right side of Figure~\ref{fig:recovery1}. This behavior of the optimization solver agrees with our previous knowledge on the ill-posedness of Kam's method and also with the Jacobian test which shows degree deficiency of the polynomial system defined by the first and second moment under the uniform distribution.

\begin{figure}[ht]
	\centering
	\begin{subfigure}[t]{0.45\textwidth}
		\includegraphics[width=.8\linewidth]{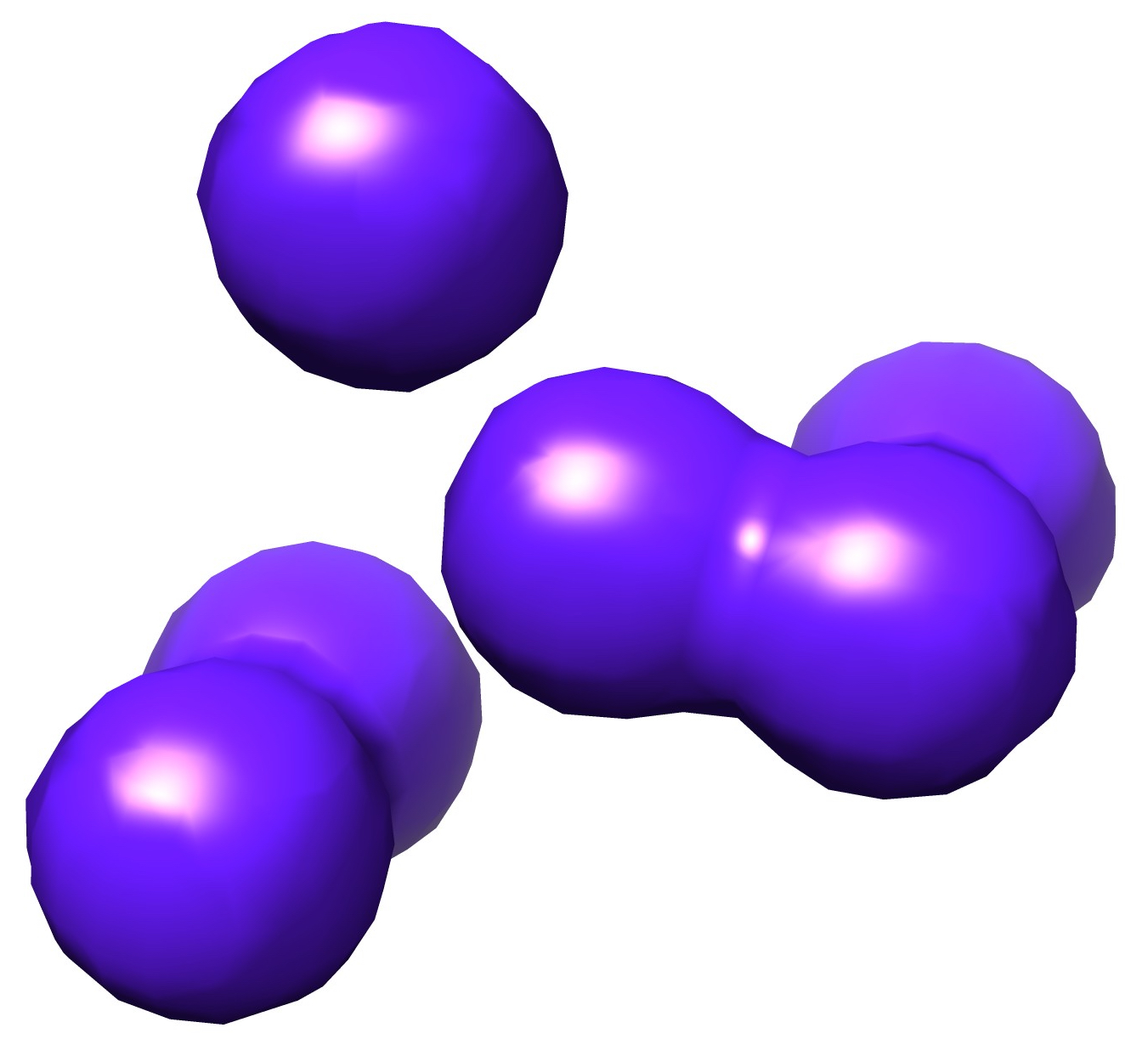}
		\caption{Mixture of Gaussians}   
	\end{subfigure}   \quad
	\begin{subfigure}[t]{0.45\textwidth}
		\includegraphics[width=.8\linewidth]{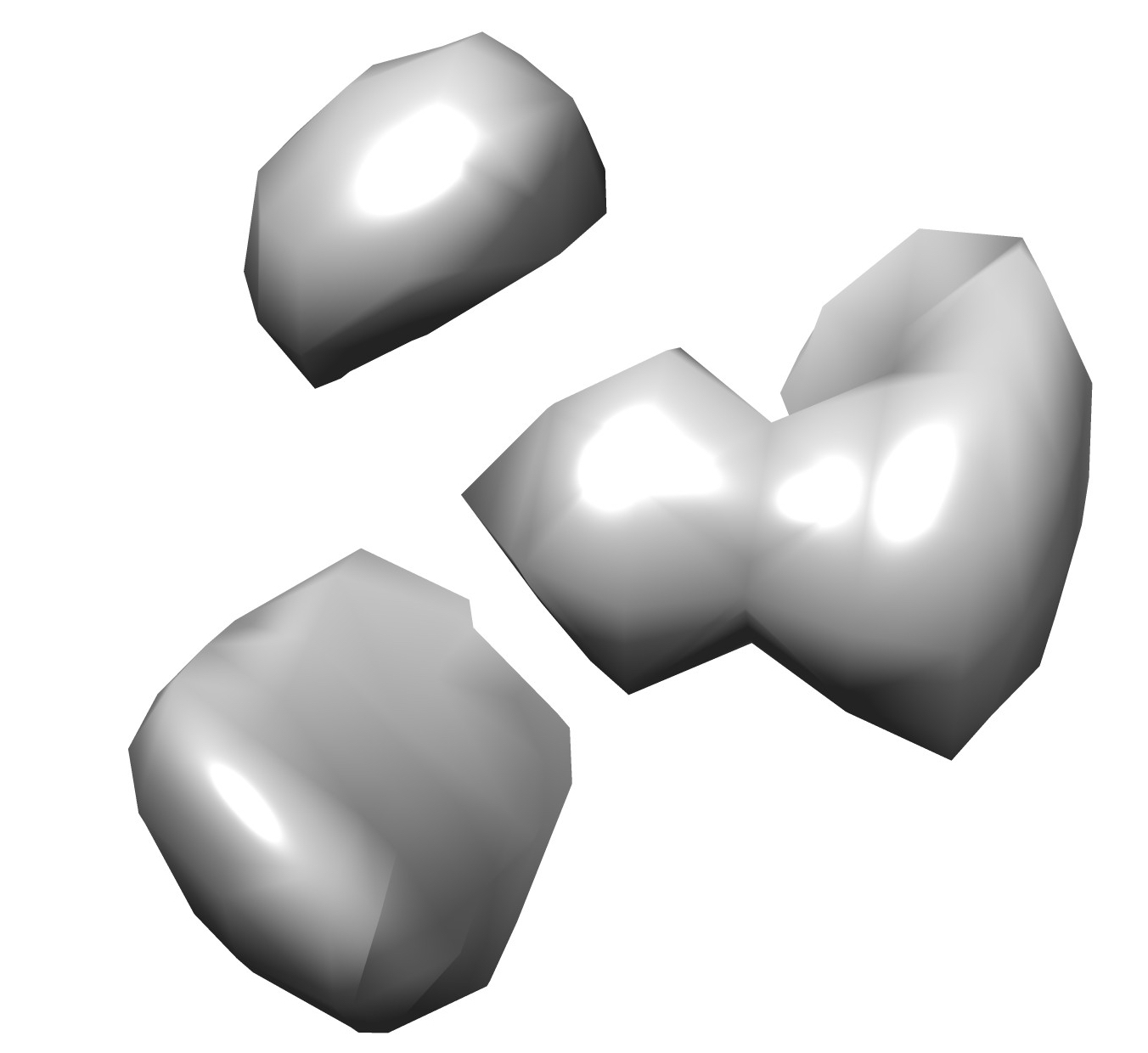}
		\caption{A low degree approximation using PSWF expansion with $L=4$}  \label{fig:vols_trunc} 
	\end{subfigure} 
\caption{Ground truth volumes}    
	\label{fig:vols}
\end{figure}

\begin{figure}[ht]
	\centering
			\includegraphics[width=.42\linewidth]{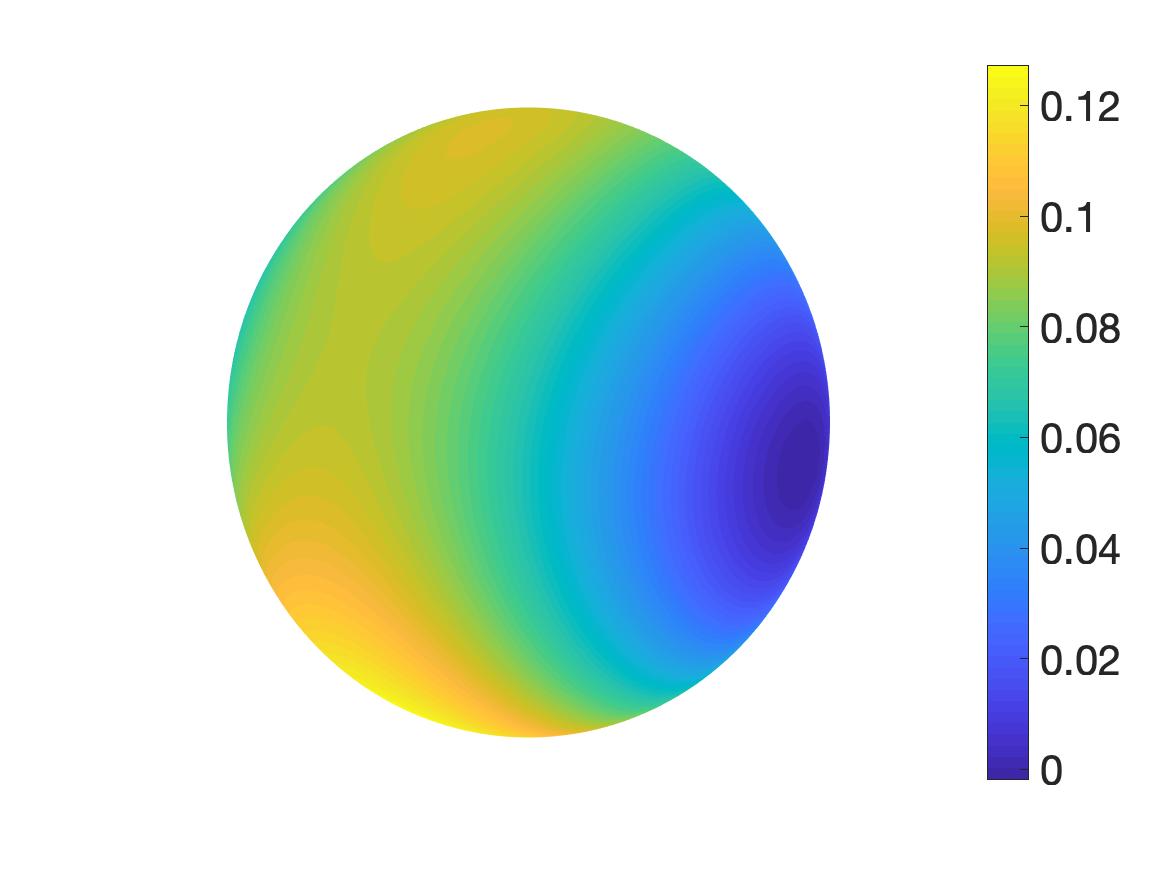}
			\caption{The non-uniform distribution of viewing angles which we use for Section~\ref{subsec:nonVSuniform}. This distribution satisfies in-plane invariance and depicted as a function on the sphere}  
				\label{fig:dist1} 
\end{figure}

\begin{figure}[ht]
	\centering
	\begin{subfigure}[t]{0.48\textwidth}
		\includegraphics[width=\textwidth]{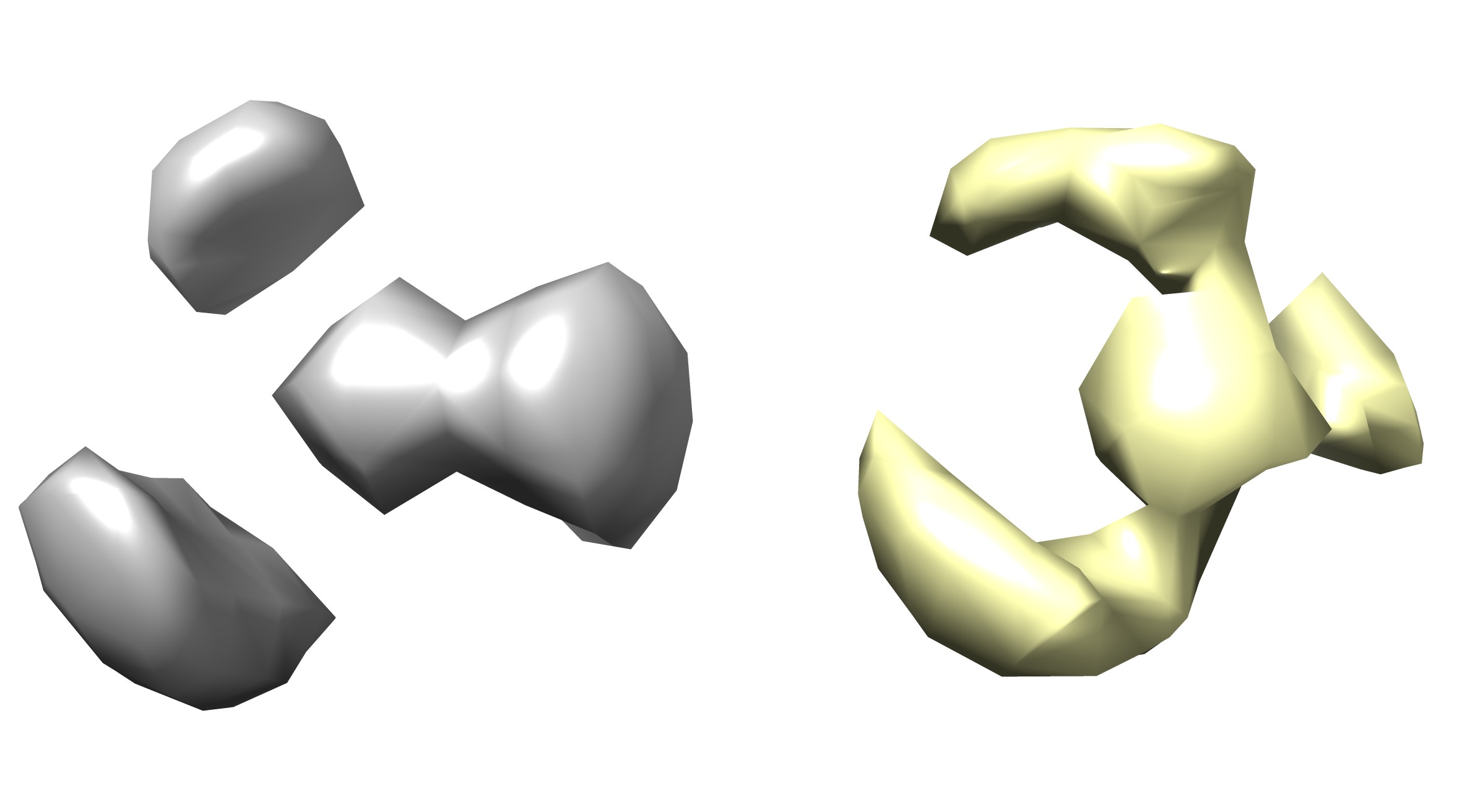}
		\caption{Recovery under non-uniform distribution} 
	\end{subfigure}      \quad
	\begin{subfigure}[t]{0.48\textwidth}
		\includegraphics[width=\textwidth]{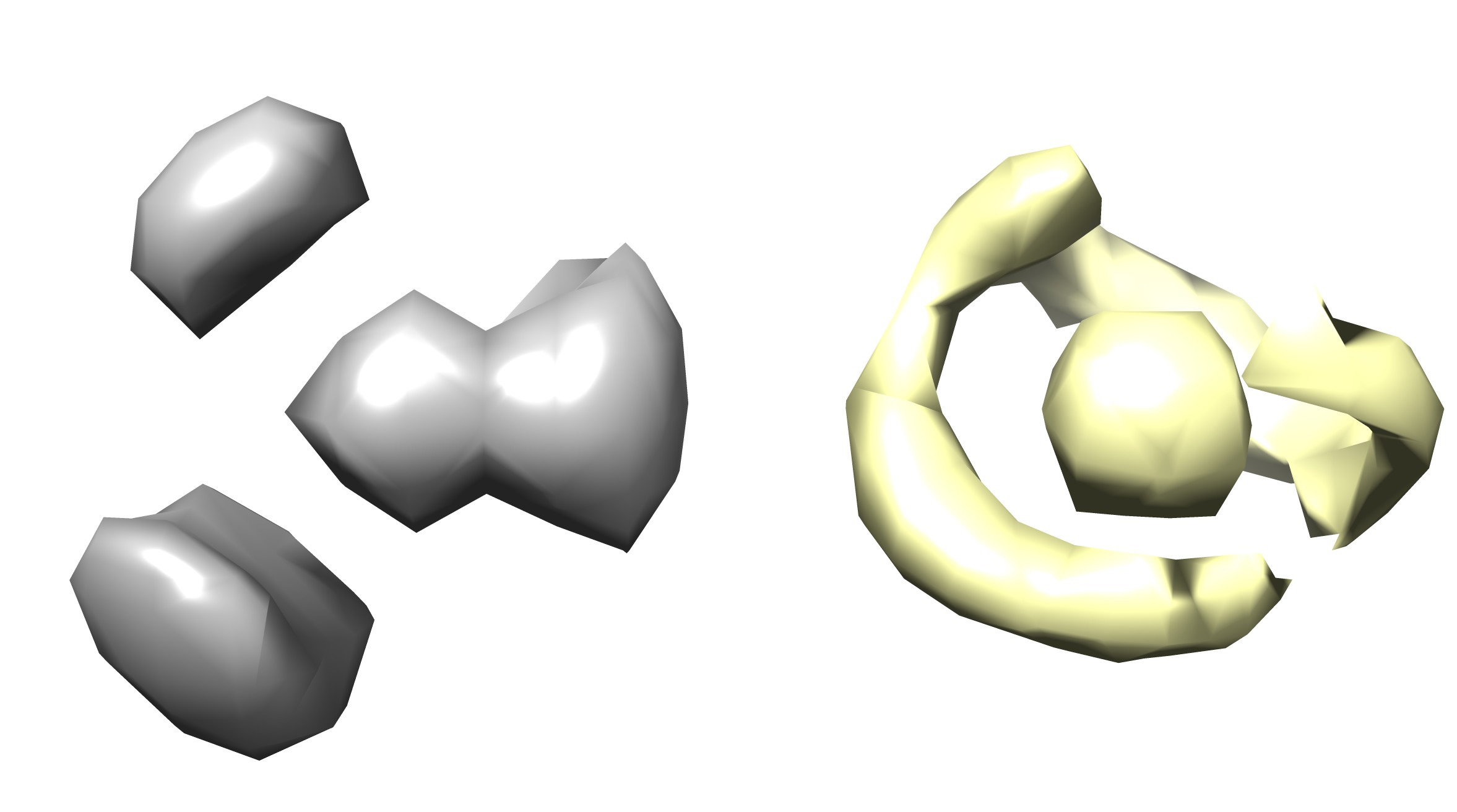}
		\caption{Recovery under uniform distribution}   
	\end{subfigure}
	\caption{Comparison of reconstructions for two cases of non-uniform and uniform distribution of rotations: ground truth volume, as also seen in Figure~\ref{fig:vols_trunc}, appears on the left of each pair (in gray), where the estimation is on the right (in yellow)} 
	\label{fig:recovery1}
\end{figure}

\subsection{Comparing volumes using FSC}

A commonly used cryo-EM resolution measure is the Fourier shell correlation (FSC)~\cite{harauz1986exact}. The FSC measures cross-correlation coefficient
between two 3-D volumes over each corresponding shell. That is, given two volumes $\phi_1$ and $\phi_2$, the FSC in a shell $\kappa$ is calculated using all voxels $\boldsymbol{\kappa}$ on this $\kappa$-th shell:
\begin{equation} \label{eqn:fsc}
\operatorname{FSC}(\kappa) = \frac{\sum_{\norm{\boldsymbol{\kappa}}=\kappa} \phi_1(\boldsymbol{\kappa})    \overline{\phi_2(\boldsymbol{\kappa})}  }{\sqrt{\sum_{\norm{\boldsymbol{\kappa}}=\kappa}\abs{ \phi_1(\boldsymbol{\kappa})}^2 \sum_{\norm{\boldsymbol{\kappa}}=\kappa}\abs{ \phi_2(\boldsymbol{\kappa})}^2 }}
\end{equation}
Customary, the resolution is determined by a cutoff value. The threshold question is discussed in~\cite{van2005fourier}, where in our case since we wish to compare a reconstructed volume against its ground truth, we use the $0.5$ threshold. Since we focus on \textit{ab initio} modeling, we aim to estimate a low-resolution version of the molecule from the first two moments. Thus, we expect the cutoff to reach a value which ensures a good starting point for a refinement procedure.

\subsection{Visual example and the effect of non-uniformity}

We next introduce an example for the most realistic scenario of an unknown, in-plane uniform distribution, by inverting the moment map of a real-world structure through minimization of a least-squares cost function~\eqref{eqn:LS}. In this example, we once again illustrate the feasibility of numerically approaching the solution, without any prior assumption on the volume.

The example is constructed as follows. As the ground truth volume, we once again use EMD-0409, the catalytic subunit of protein kinase A bound to ATP and IP20~\cite{herzik2019high}, as presented at the online cryo-EM data-bank~\cite{lawson2015emdatabank}. The map original dimension is $128\times 128 \times 128$ voxels. Since we aim to recover a low-resolution model, we reduce complexity and downsample it by a factor of three to $43\time 43 \time 43$. We firstly expand this volume using PSWFs with a band limit $c$ chosen as the Nyquist frequency and 3-D truncation parameter~\eqref{eq:PSWFs truncation rule} of $\delta = 0.99$. The full expansion consists of degree $L=40$, and after truncation to maximize conditioning, as done in Section~\ref{sec:uni3}, we aim to recover the low degree counterpart up to degree $L=6$. The moments were calculated with respect to 2-D prescribed accuracy~\eqref{eqn:epsilon_eq} of $\epsilon = 10^{-3}$ and in the absence of noise. The volume contributes $657$ unknowns to be optimized.

As the ground truth distribution, we choose three different functions: uniform, highly non-uniform and a non-uniform case in-between. The two non-uniform cases are cubic spherical harmonics expansions ($P=3$) and satisfy in-plane invariance and so we present them in Figure~\ref{fig:nonuniform_dists} as functions on the sphere, together with a histogram to compare and illustrate their ``non-uniformness''. The non-uniform distributions add extra $15$ unknowns which means that, in total, we optimize $672$ unknowns in the cases of non-uniform distribution and only $657$ unknowns in the case of uniform distribution.

\begin{figure}[ht]
	\centering
	\begin{subfigure}[t]{0.31\textwidth}
		\includegraphics[width=.99\linewidth]{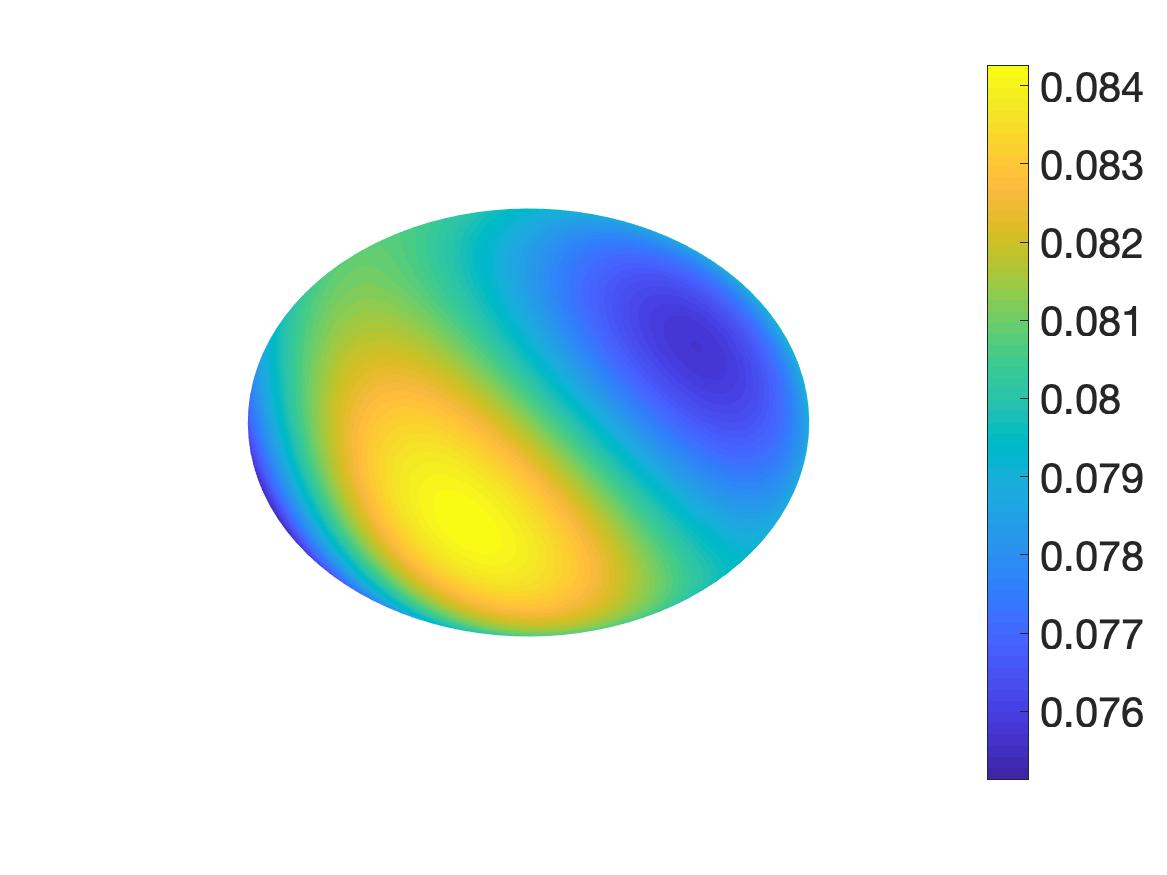}
		\caption{The less non-uniform distribution function on the sphere}   
	\end{subfigure}   \quad
	\begin{subfigure}[t]{0.31\textwidth}
		\includegraphics[width=.98\linewidth]{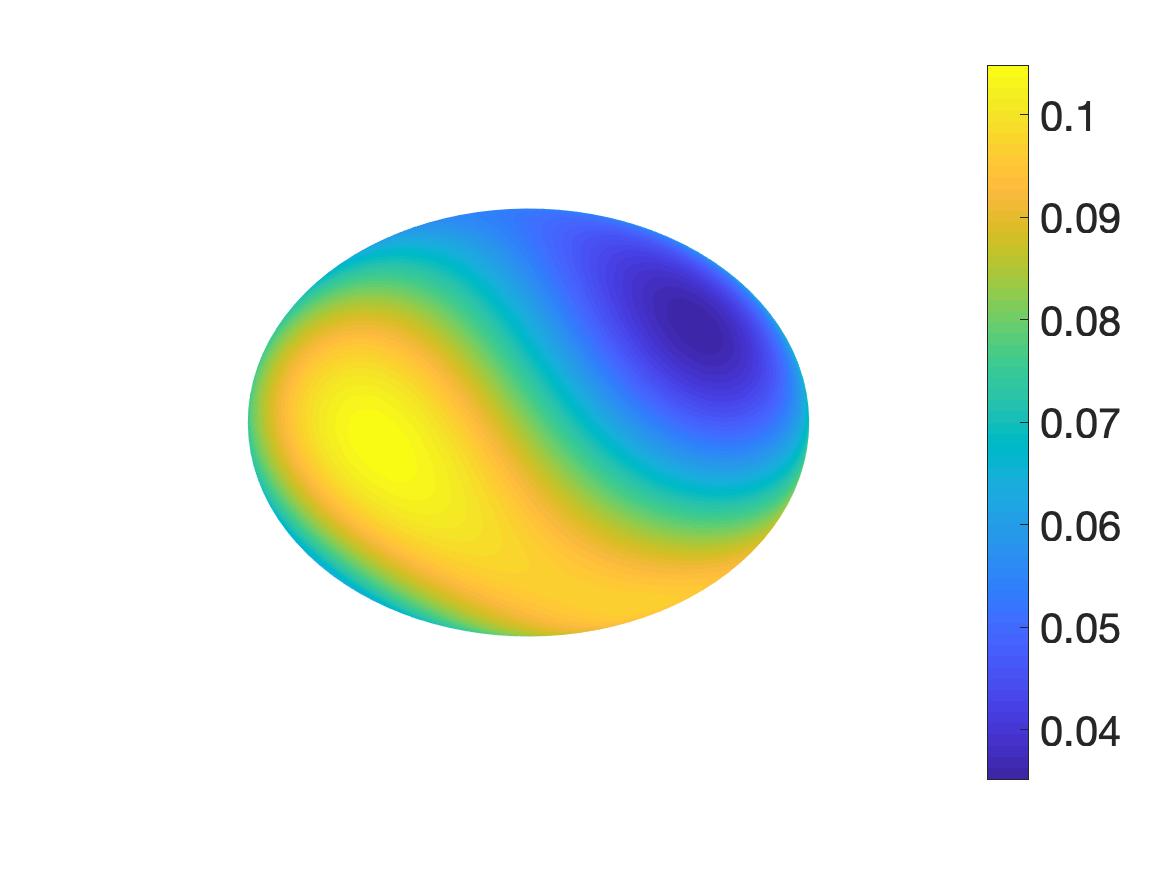}
		\caption{The more non-uniform distribution function on the sphere} 
	\end{subfigure}       \quad
	\begin{subfigure}[t]{0.3\textwidth}
		\includegraphics[width=.99\linewidth]{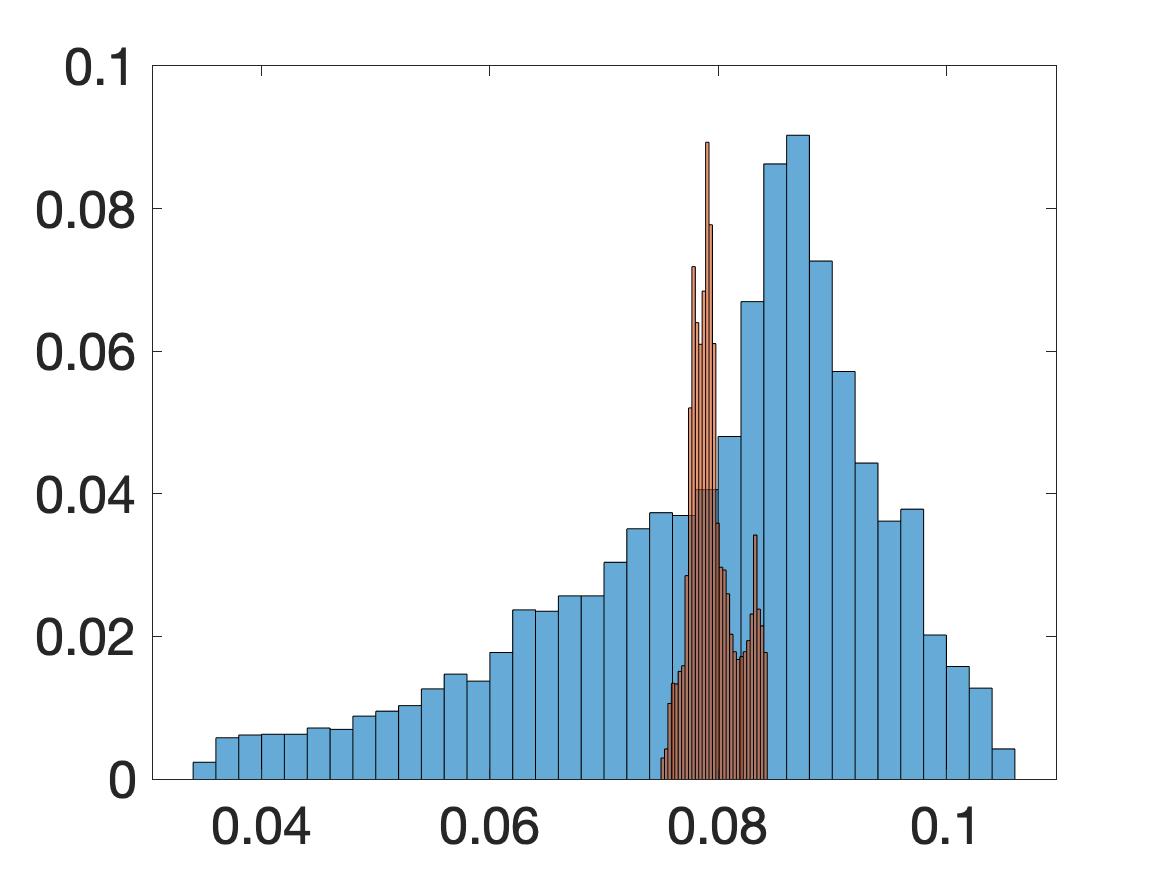}
		\caption{The probability of each value to appear in the distribution: a comparison to illustrate the different non-uniformity levels of the two distributions}   
	\end{subfigure}  
	\caption{The two non-uniform distributions in use} 
	\label{fig:nonuniform_dists}
\end{figure}

In the optimization process, we use the limit of the empirical moments~\eqref{eqn:empirical_mom} ($n \to \infty$) as our input moments. As before, we use a \textit{trust-region} algorithm, see \eg,~\cite{powell1990trust}, which is a gradient-based method. To fix the initialization between the different cases, we start the search with the zero volume. In cases of non-uniform distribution, we provide a random non-uniform distribution to start with. Our method is implemented in MATLAB R2017b, and we calculated the example on a laptop with a 2.9 GHz Intel Core i5 processor and 16 GB 2133 MHz memory. 

The result we present next is obtained after $60$ iterations of trust-region, each iteration usually uses up to $30$ inner iterations to estimate the most accurate step size. The runtime of this example is about $55$ minutes for each model, where at this point, our naive implementation does not support any parallelization which potentially can lead to a significant improvement in the total runtime. For example, the evaluation of the second moment and its associated gradient part are related to the leading complexity term as described in Section~\ref{subsec:complexity}. Their implementation is based upon matrix product as seen in the form~\eqref{eqn:second_mom_compacy}. This part can remarkably benefit from parallel execution. Note that evaluating the PSWF functions, as well as the product Clebsch-Gordan coefficients (which appears in the moments), are all calculated offline as a preprocessing step.

We present a comparison between the different FSC curves for the three cases. As implied by Figure~\ref{fig:fsc}, the resolution increases (lower FSC cut) as the non-uniformity becomes more significant. Specifically, with the uniform distribution we obtain merely $39.1 \textup{\AA} $, where for the two other non-uniform cases we get $22.5 \textup{\AA} $ and $19.0 \textup{\AA} $ as the non-uniformity increases in the examples of Figure~\ref{fig:fsc}.

\begin{figure}[ht]
	\begin{center}
		\includegraphics[width=.8\textwidth]{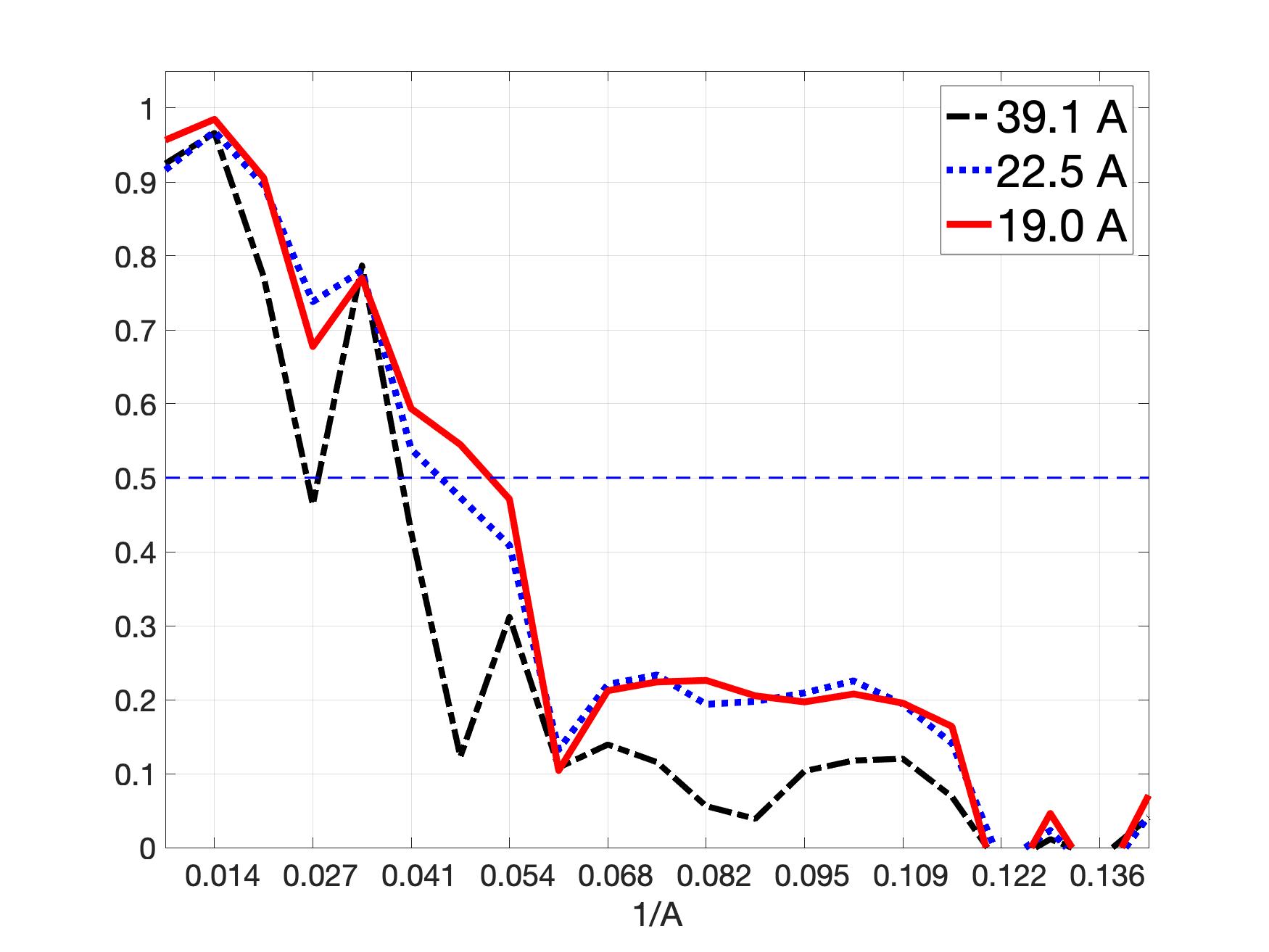}
	\end{center}
		\caption{The FSC curves of the three test cases. The dashed curve (in black) is of the uniform distribution, the dot line (blue) is of the less radical non-uniform case, and the solid curve (red) is of the most non-uniform distribution case. As customary, we use the conventional FSC cutoff value of $0.5$.} 
		\label{fig:fsc}
\end{figure}

A visual demonstration of the output of the optimization is presented in Figure~\ref{fig:recon}, where we plot side by side the ground truth and three models, from the uniform to the most non-uniform one.

\begin{figure}[ht]
	\centering
\begin{center}
		\includegraphics[width=.48\linewidth]{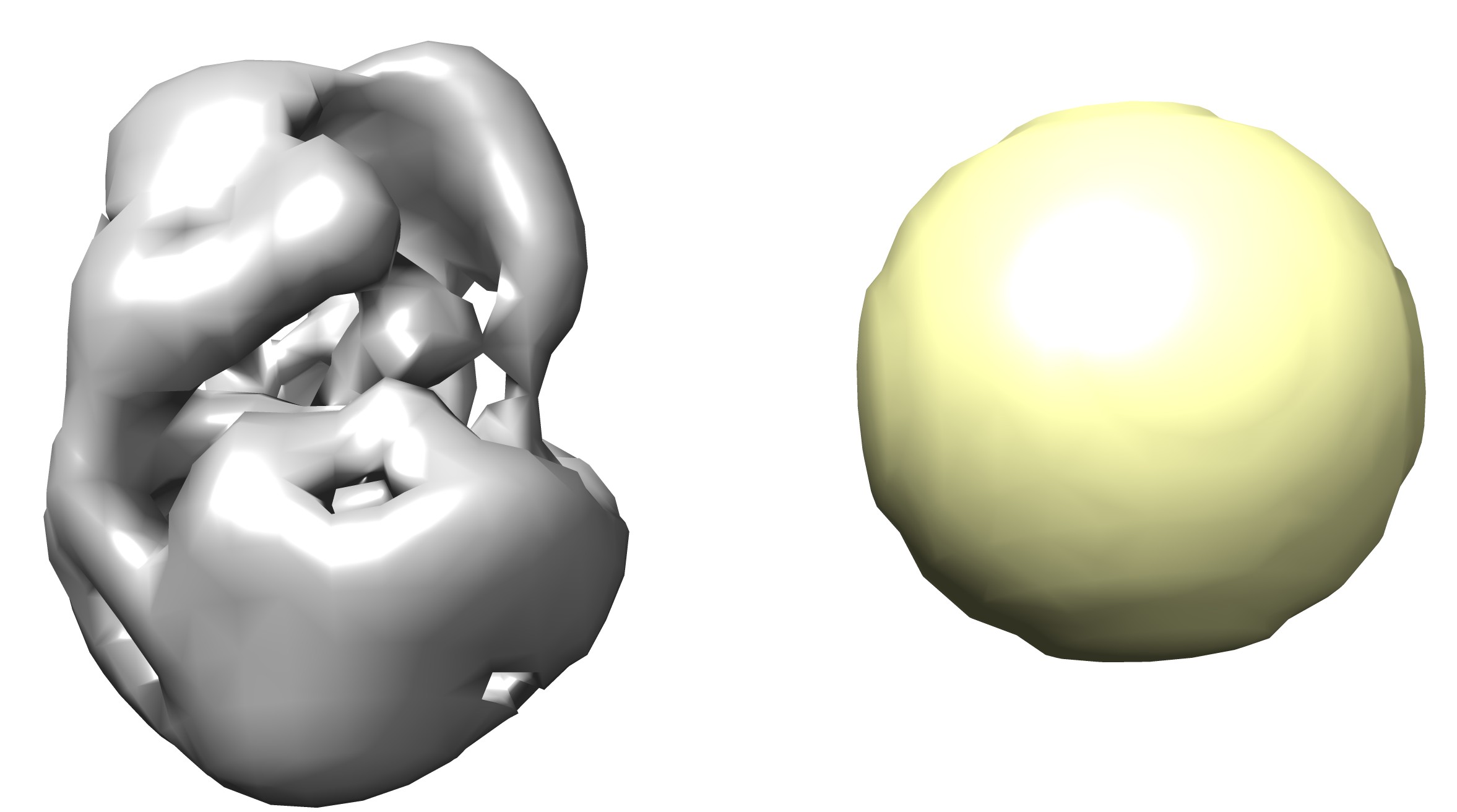} \,
		\includegraphics[width=.48\linewidth]{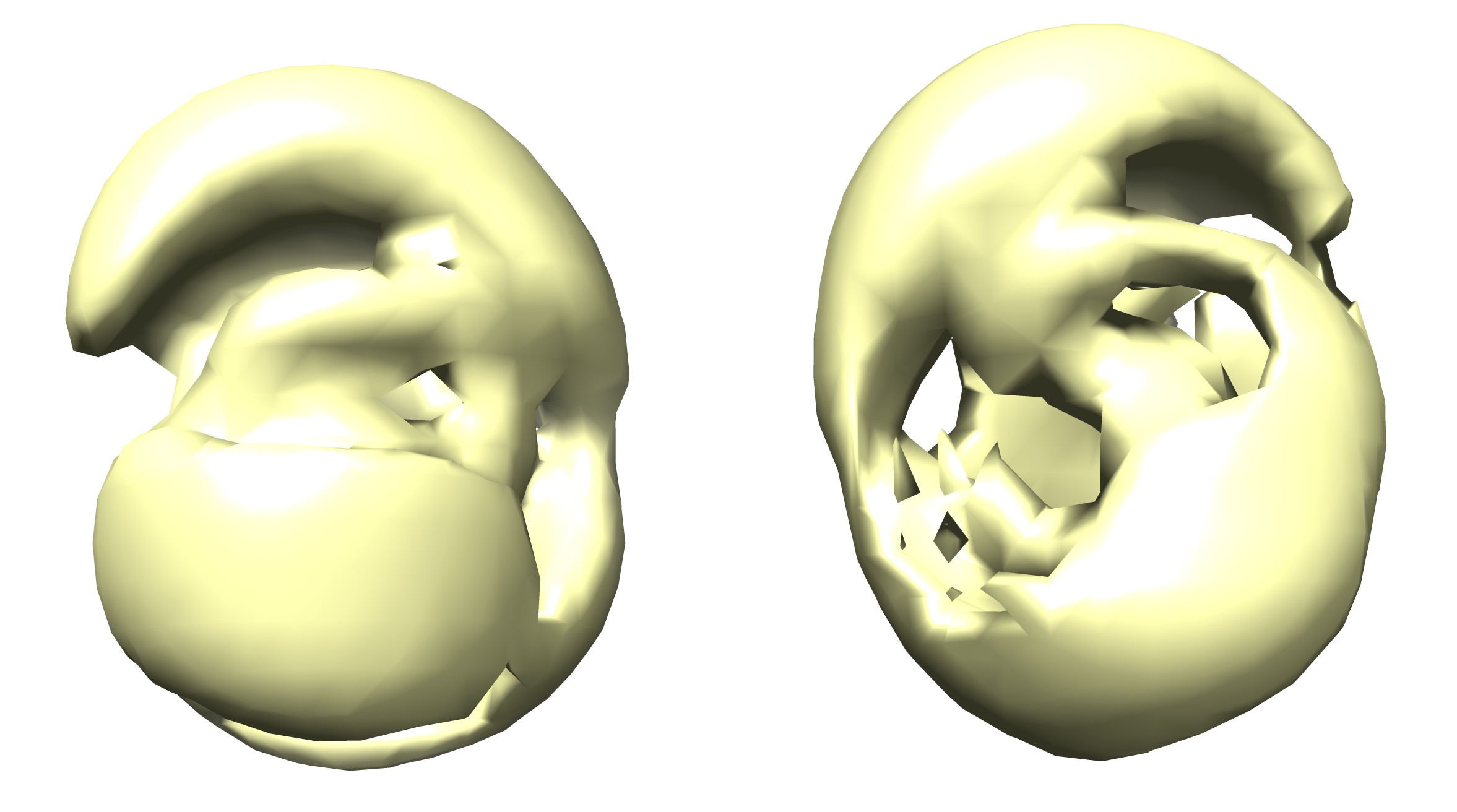}
\end{center}
	\caption{The estimations which were obtained by inverting the moments via optimization. The ground truth volume appears on the left (in gray), where the models are on the right (in yellow), ordered as associated with the different distributions, from uniform on the left to the most non-uniform on the right.} 
	\label{fig:recon}
\end{figure}

\subsection{Recovery from noisy images}

We conclude this section with an example of recovering a volume from its noisy projection images. The volume is a mixture of six Gaussians, synthetically designed to have no spatial symmetry. The volume's size is $15 \times 15 \times 15$ and its full PSWF expansion is of length $L=13$, with band limit $c$ chosen as the Nyquist frequency and 3-D truncation parameter~\eqref{eq:PSWFs truncation rule} of $\delta = 0.99$. 
We use an in-plane uniform distribution of rotations, very localized on a $45$ degree cone, represented with an expansion length of $P=3$. The distribution function is shown on Figure~\ref{fig:dist_noisy_im} and can model a realistic scenario of highly anisotropic viewing directions (see, e.g.,~\cite{baldwin2019non}). Using the distribution, we generated projection images to obtain $200,000$ observations. These images were then contaminated with noise. The SNR of an image $I_j$ with the noise term $\varepsilon_j$ is $\SNR_j = \norm{I_j-\varepsilon_j}^2/\norm{\varepsilon_j}^2$, using the Frobenius  norm. The noise was chosen to achieve an average SNR value of $1/3$. Three examples of clean images and their noisy versions are depicted in Figure~\ref{fig:noisy_im}. As seen in Figure~\ref{fig:noisy_im}, the projections are hardly noticed in the noisy images for the naked eye.

\begin{figure}[ht]
	\centering
	\includegraphics[width=0.35\textwidth]{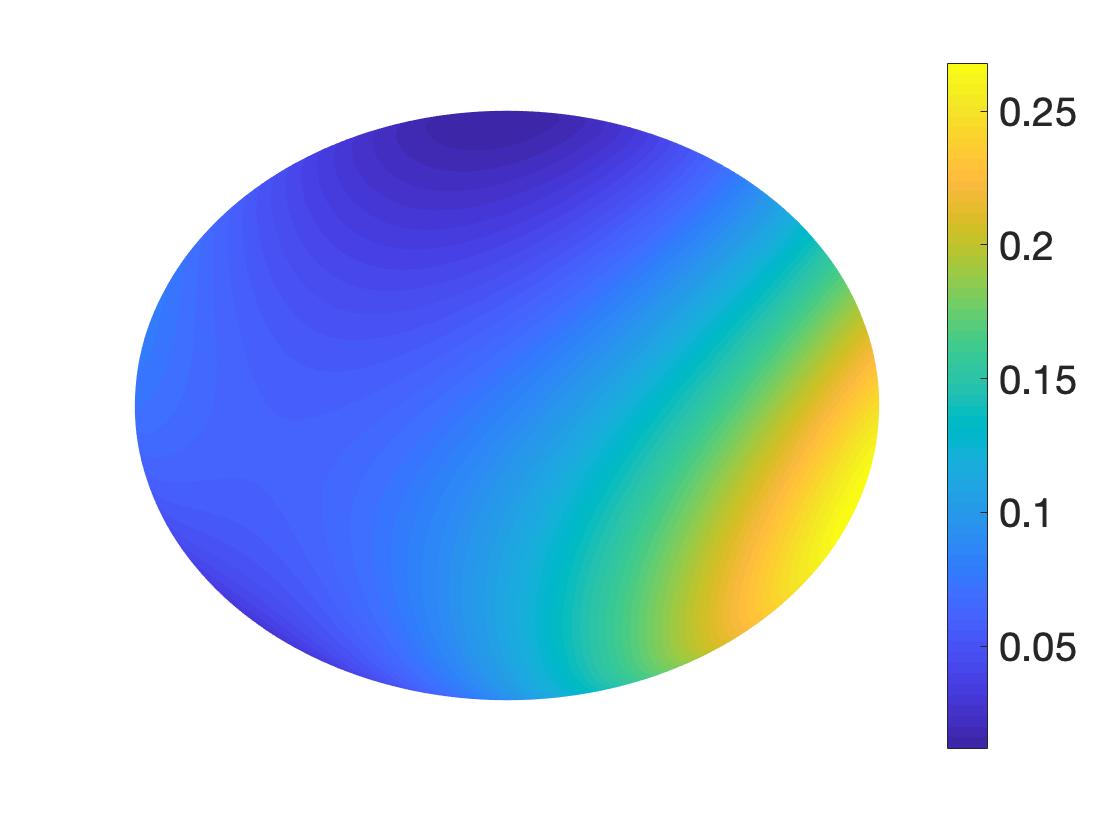}
	\caption{The distribution function on the sphere.}
	\label{fig:dist_noisy_im}
\end{figure}

\begin{figure}[ht]
	\begin{minipage}{.95\textwidth}
	\centering
	\includegraphics[width=.3\textwidth]{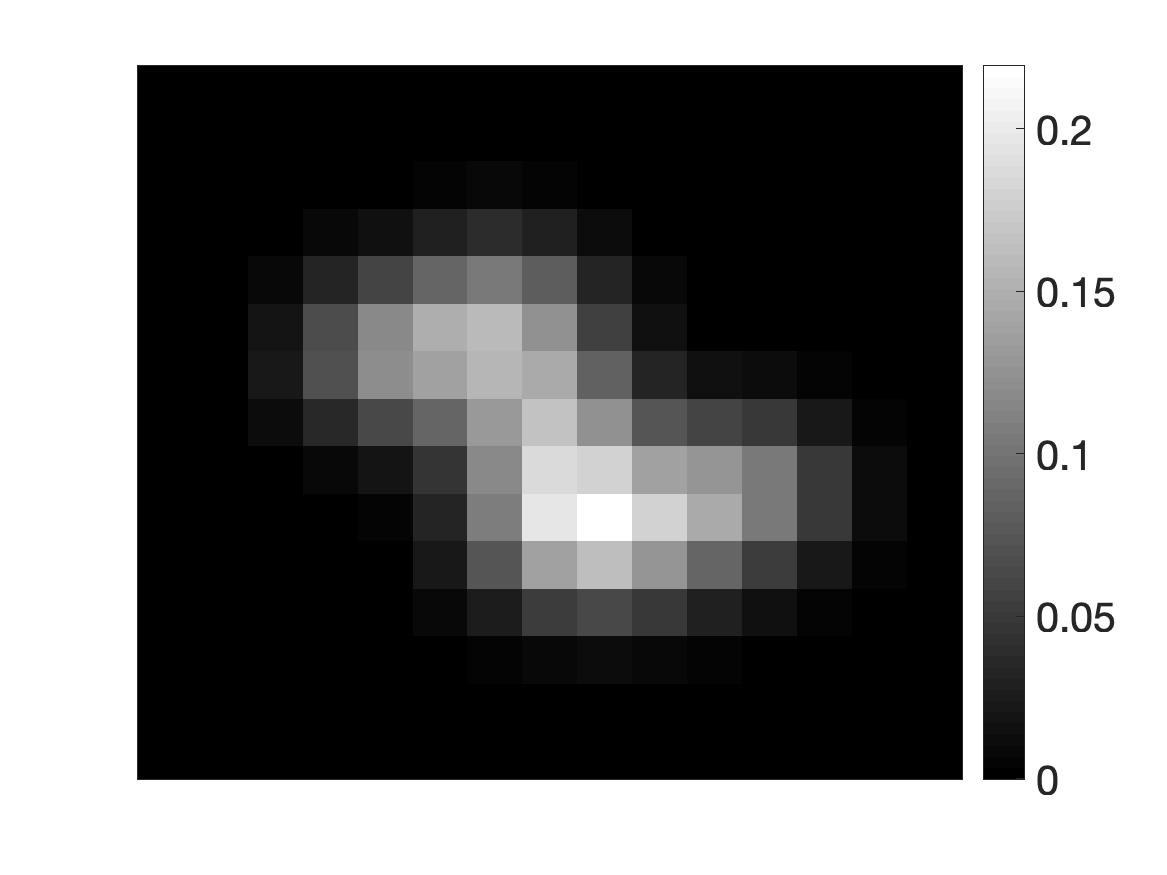}  \hspace{-9pt}
	\includegraphics[width=.3\textwidth]{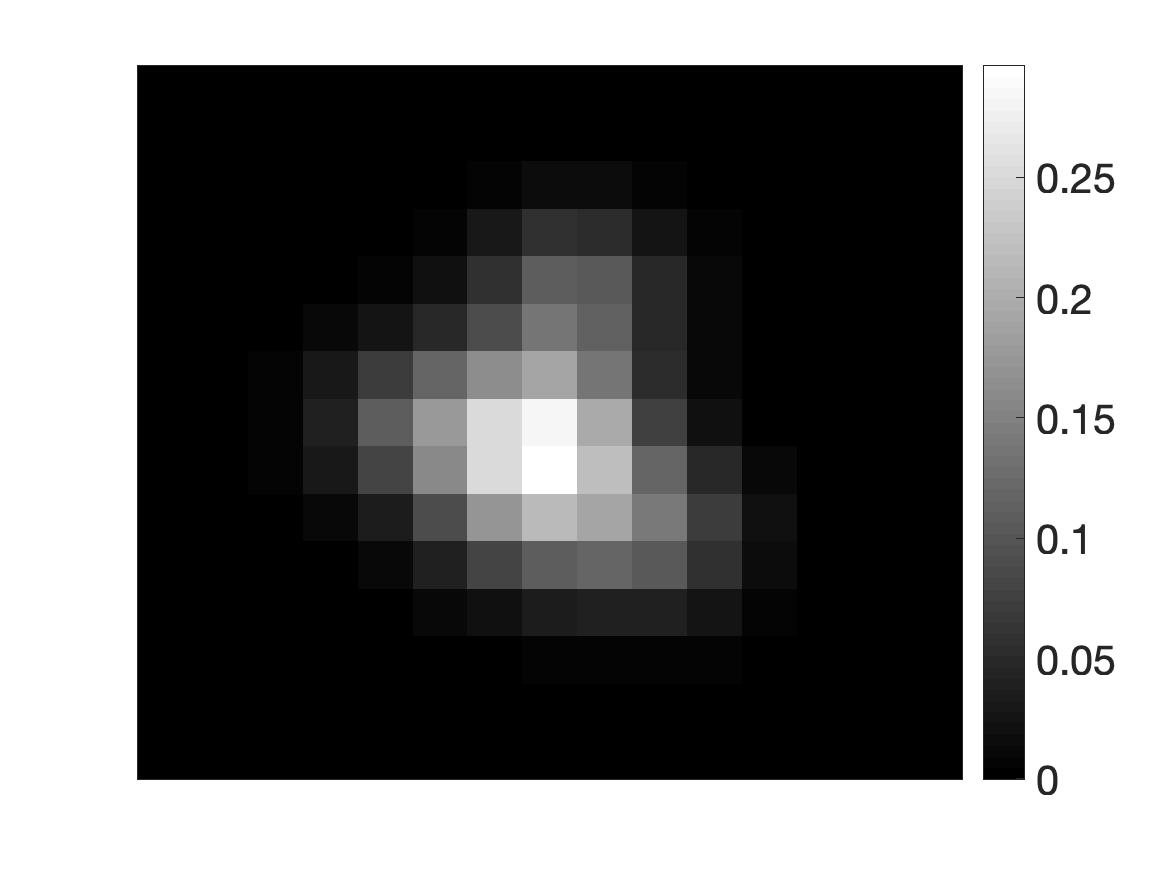}  \hspace{-9pt}
	\includegraphics[width=.3\textwidth]{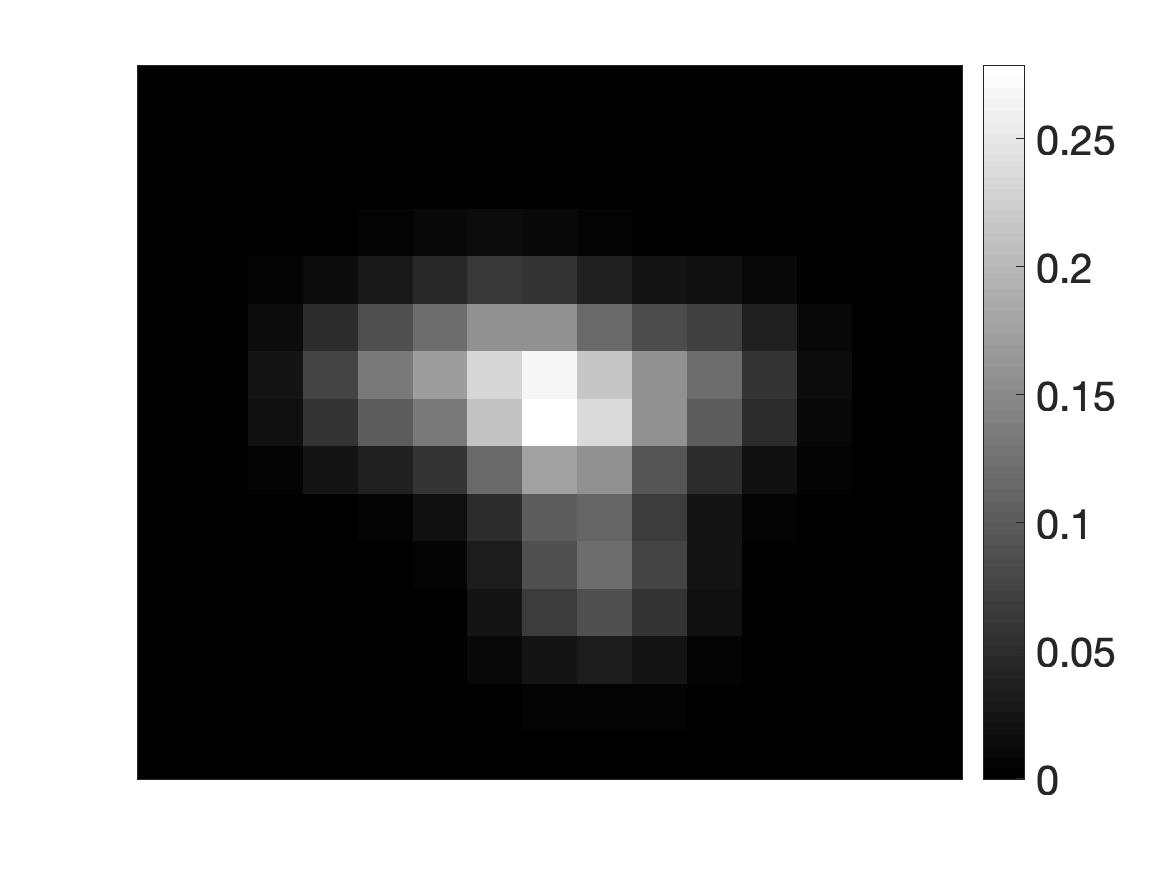} \\
	\includegraphics[width=.3\textwidth]{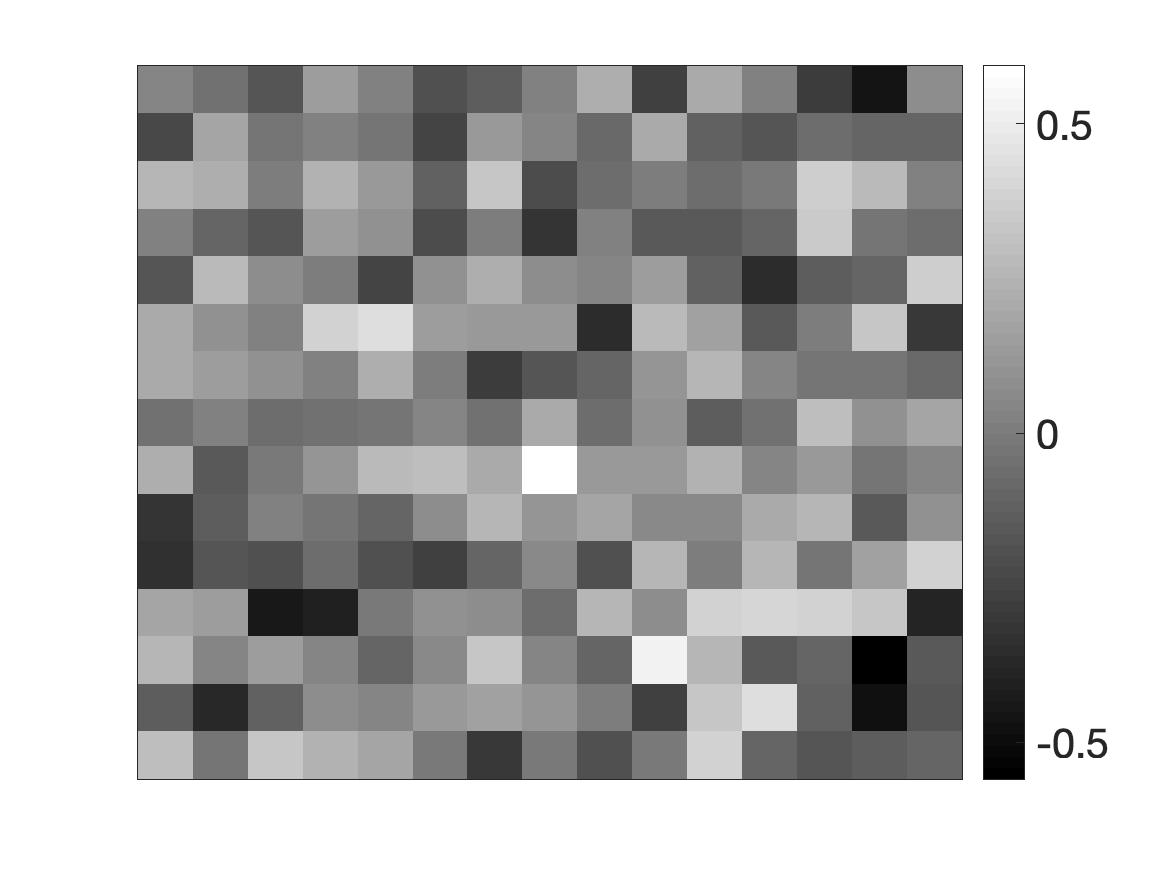}   \hspace{-9pt}
	\includegraphics[width=.3\textwidth]{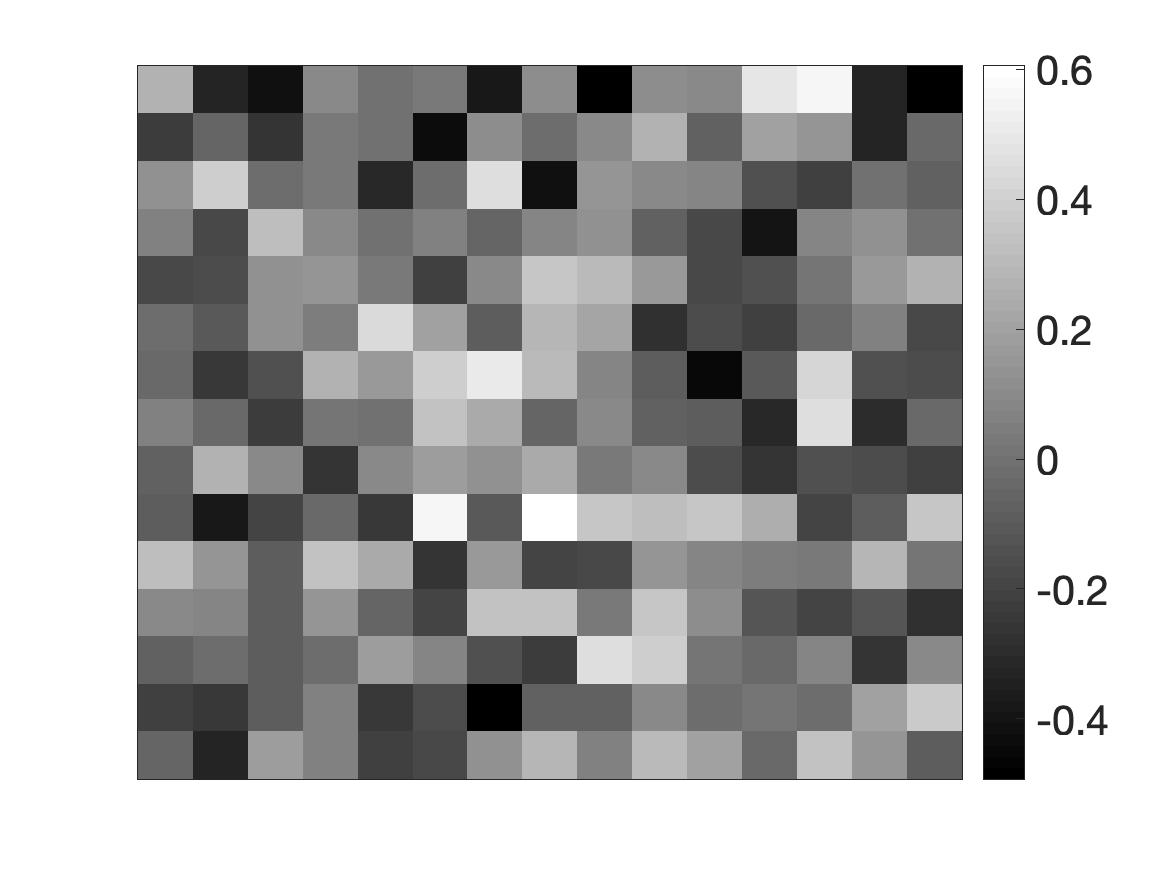}   \hspace{-9pt}
	\includegraphics[width=.3\textwidth]{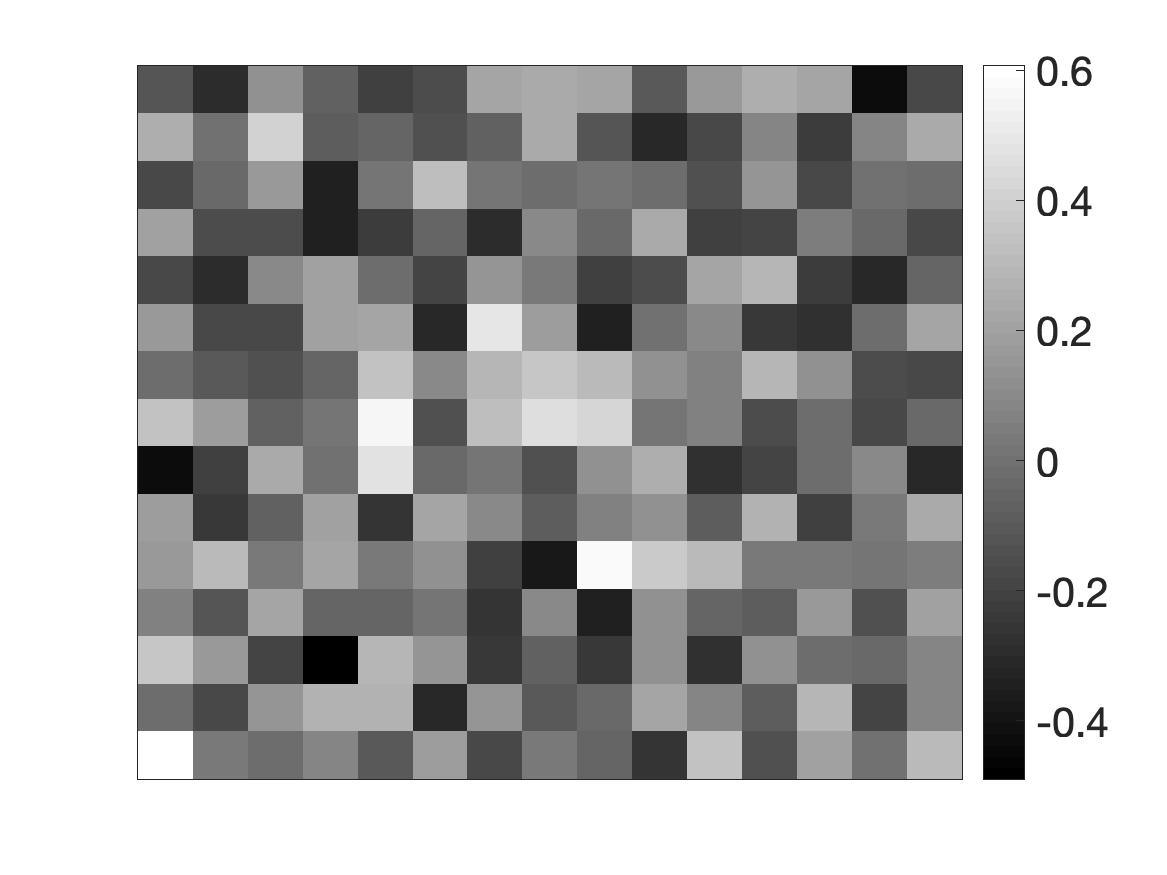} 
	\caption{Three projection images: in upper row  as clean and noisy images. The resulted SNR is about $1/3$.}
	\label{fig:noisy_im}
\end{minipage} 
\end{figure}

We expand the noisy images using a 2-D PSWF basis, as appears in~\eqref{eqn:image_expand}. Then, the coefficients and their double-products are averaged to estimate the first and second moments as in~\eqref{eqn:empirical_mom}. The reconstruction uses the empirical moments to estimate the volume and distribution. For the volume, our gradient-based least-squares algorithm targets its full expansion, which consists of $192$ unknowns. The unknown distribution includes $8$ unknowns spherical harmonics coefficients. We reached the result we present next very quickly, starting from a random initial guess. It took about $15$ iterations of trust-region; each iteration could use up to $30$ inner iterations to estimate the most accurate step size. The runtime of this example is less than $10$ minutes.

A visual demonstration of the estimated volume is provided in Figure~\ref{fig:recon_from_ims}. We present the estimation, side by side, to the original volume. As seen in the various pictures, the reconstruction, while not perfect, captures most features and the general shape of the structure. This encouraging result indicates that inverting the moments is possible also from noisy moments and that the mapping has some robustness to small perturbations.

\begin{figure}[ht]
	\centering
	\begin{center}
		\includegraphics[width=.4\linewidth]{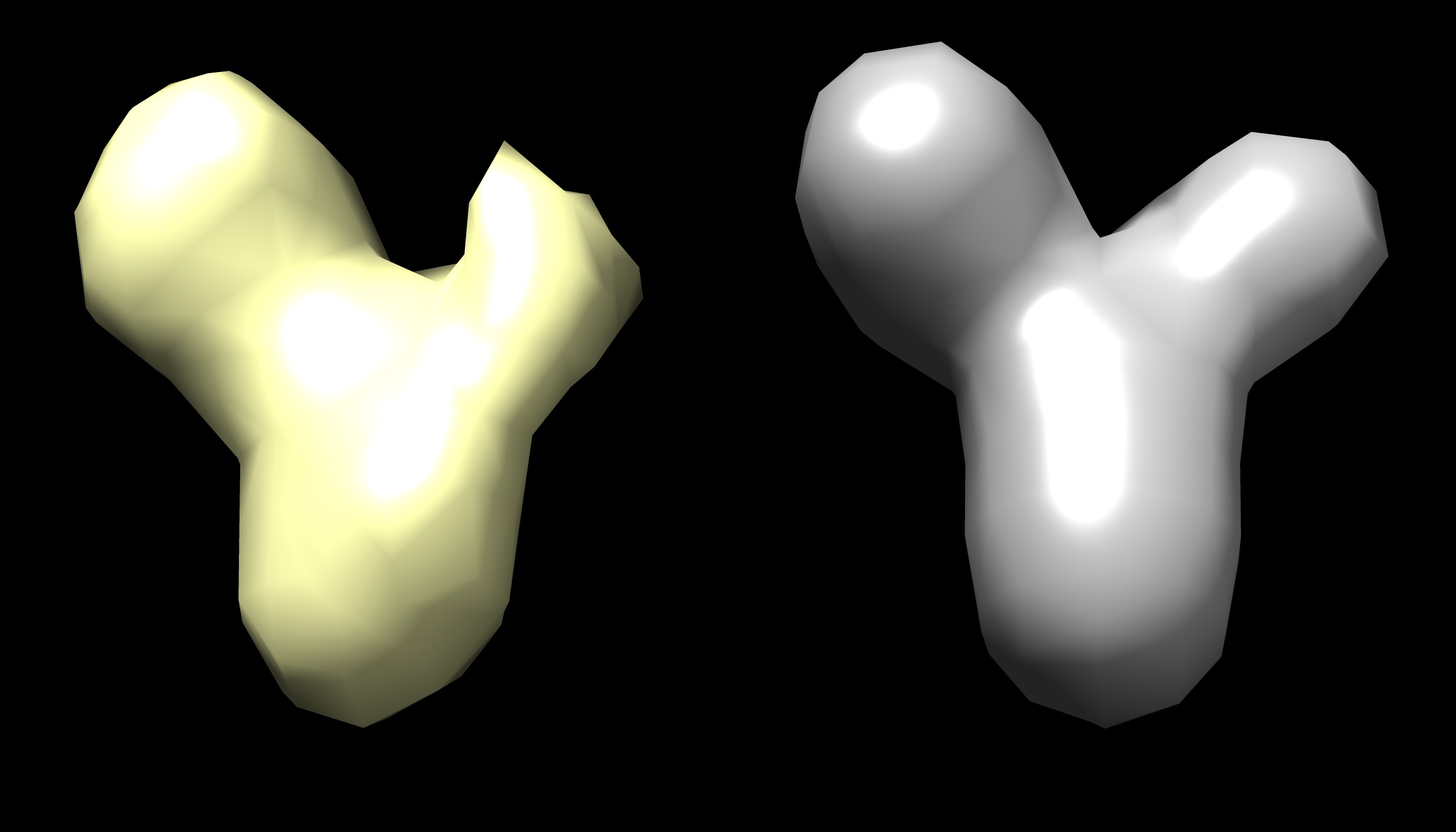}  \quad
		\includegraphics[width=.4\linewidth]{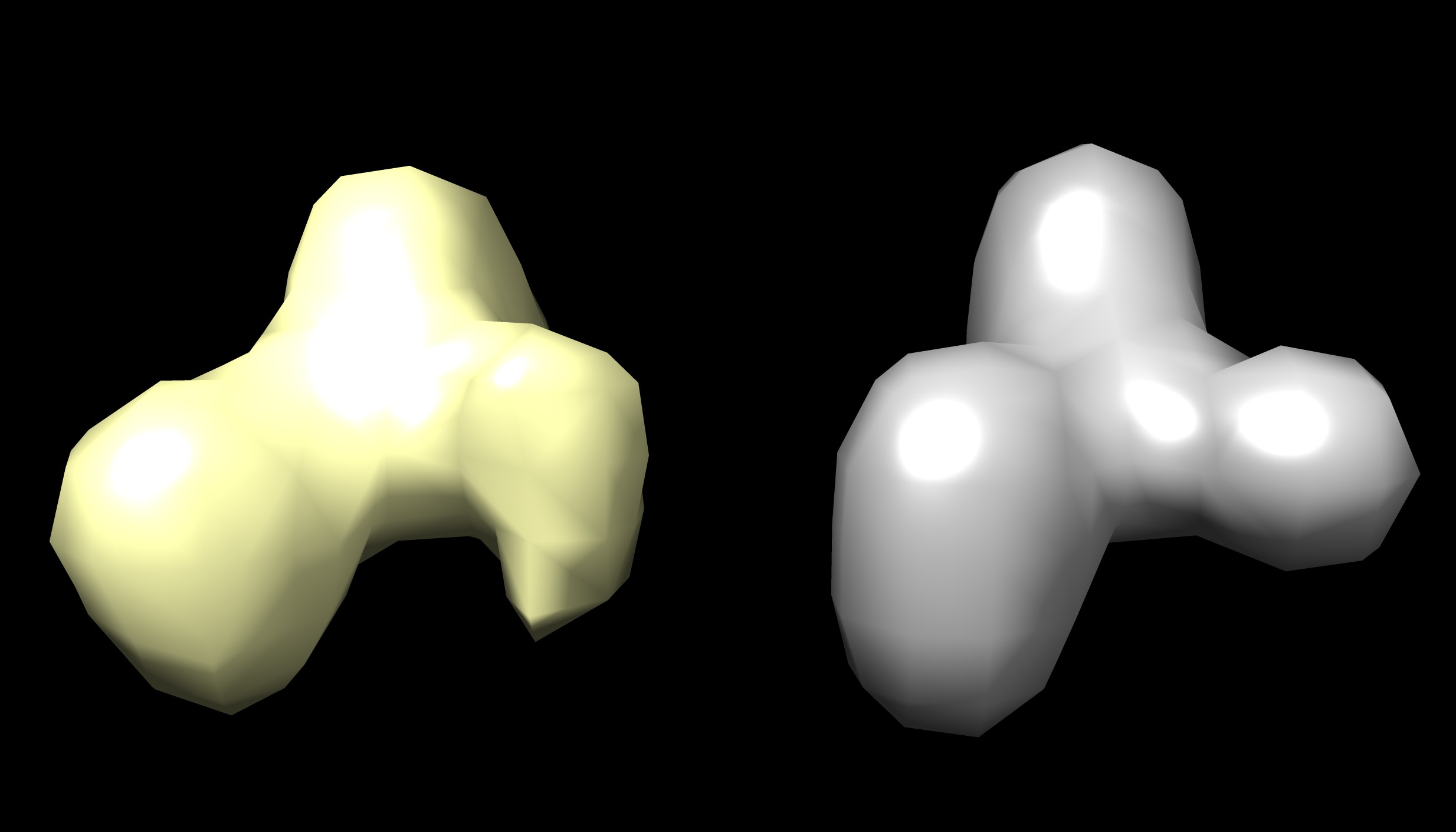}  \\ \vspace{5pt}
		\includegraphics[width=.4\linewidth]{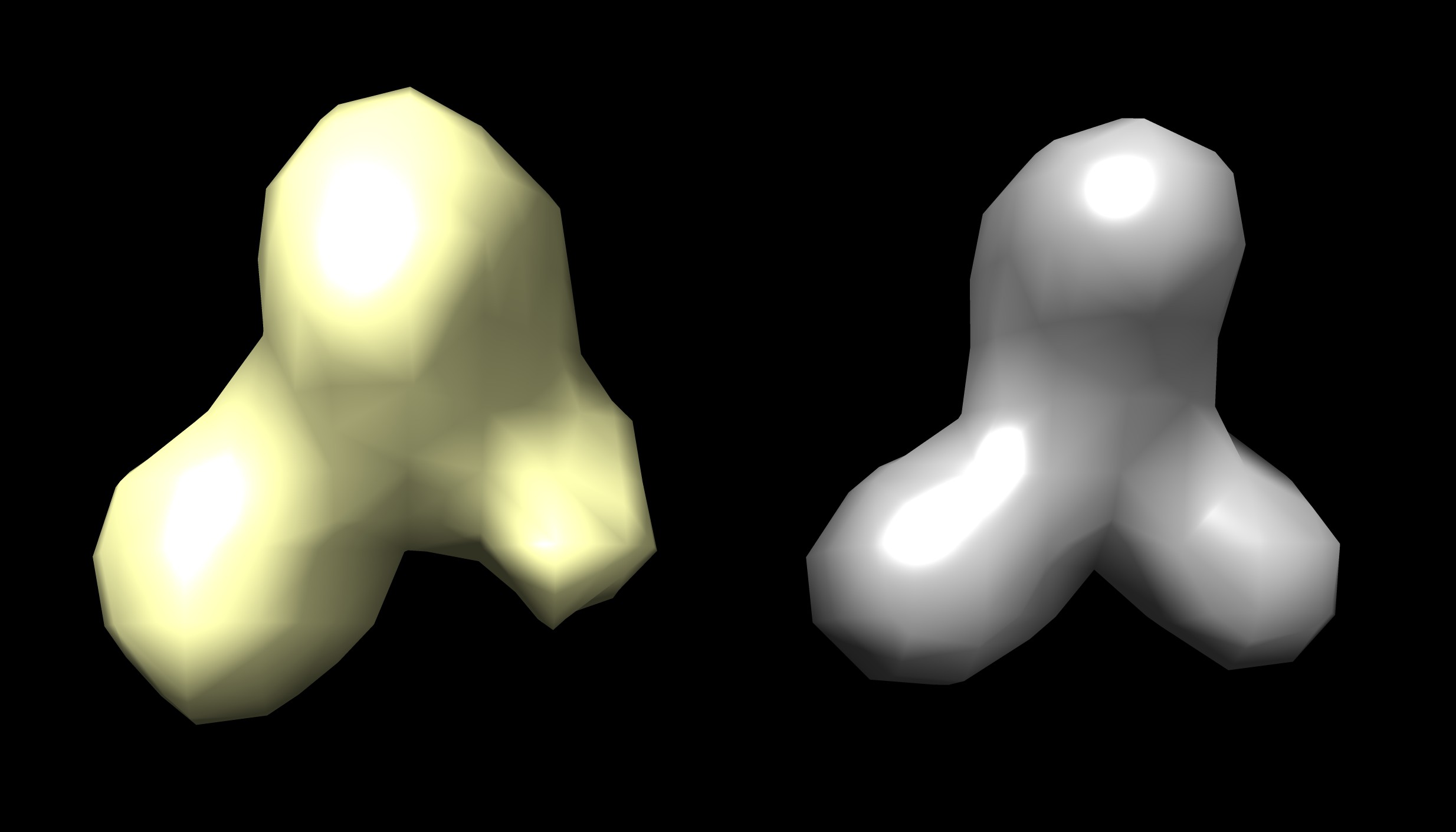} \quad
		\includegraphics[width=.4\linewidth]{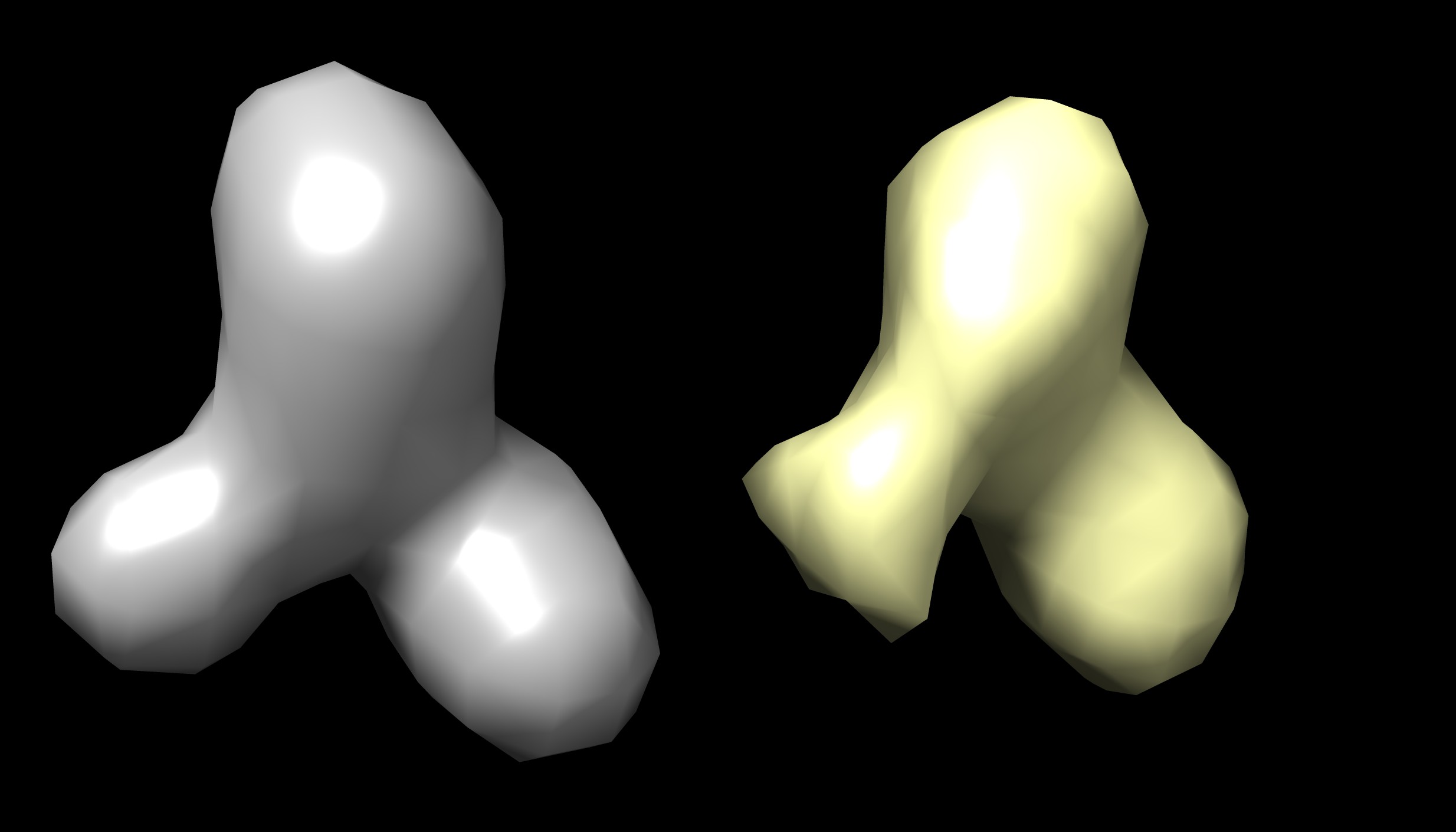} 
	\end{center}
	\caption{Reconstruction from moments of noisy images: an illustration taken from four different viewing angles. The estimation appears in yellow (left volume on the top left corner picture) and the original volume is in gray.} 
	\label{fig:recon_from_ims}
\end{figure}

\section{Discussion and Conclusion}

The method of moments offers an attractive approach for modeling volumes in cryo-EM. This statistical method completely bypasses the estimation of viewing directions by treating them directly as nuisance parameters. The assumption of a non-uniform distribution of viewing angles enables in many cases volume estimation using only the first and second moments of the data. This phenomenon opens the door for fast, single-pass reconstruction algorithms, based on inverting the map from the volume and distribution to the low-order statistics of the projection images.  

This paper extended Zvi Kam's original method of moments for cryo-EM to the setting of a non-uniform distribution of viewing directions.  We formulated the reconstruction problem using appropriate discretizations for the images, the volume, and the distribution. Then, we derived moment formulas using properties of the spherical harmonic functions and Wigner matrix entries.  Computational algebra was employed to analyze the resulting large-scale system of polynomial equations.  The analysis shows the seeming complication of an unknown, non-uniform distribution renders 3-D reconstruction \textit{easier} than in the uniform case, as now only first and second moments are required to determine a low-resolution expansion of the molecule, up to finitely many solutions. Intermediate cases were treated; remarkably, when the distribution is known and totally non-uniform over $\so{3}$, there is an efficient, provable algorithm to invert the first and second moments non-linearly.  Additionally, our work addressed several numerical and computational aspects of the method of moments.  An implementation of a trust-region method was presented and used to illustrate the advantages of our approach over Kam's classical approach by numerical experiments involving synthetic volumes. 

We regard our work as a definite, albeit initial step toward developing the method of moments for \textit{ab initio} modeling from experimental datasets.  Firstly, even in the synthetic cases considered here, further work on the optimization side is warranted.  Variations on our nonlinear cost function that incorporate a pre-conditioner, e.g., \eqref{eqn:LS1}, could be considered.  Secondly, other techniques for large-scale nonlinear least squares optimization should be tried, such as Levenberg-Marquardt \cite{lourakis2005levenberg} or Variable Projection \cite{chen2018regularized}, where in the latter one can exploit the linearity in the moments with respect to the distribution, by eliminating out the distribution. Thirdly, to get our method working on images, further effects, such as the CTF and imperfect centering of picked particles, should be incorporated into the moment formulas.  Fourthly, accurate covariance estimation in high dimensions requires eigenvalue shrinkage \cite{donoho2018optimal}, the theory for which may call for a modification in the non-uniform setting. 

To simplify our exposition, we have stuck to the asymmetric and homogeneous cases here, although both of these can be relaxed in the method of moments.  Specifically, as already noted in Kam's original paper~\cite{kam1980reconstruction}, point symmetries of molecules are reflected in the vanishing of certain expansion coefficients, see also~\cite{van1999pointgroup}. Therefore, MoM can be reformulated using fewer coefficients for symmetric molecules. This fact, alongside with further improvement of the representation of the distribution, may pave the way for recovery under practical cases of very restricted viewing angles, as reported in literature~\cite{baldwin2019non, glaeser2017opinion, naydenova2017measuring, tan2017addressing}. At the same time, heterogeneity, at least if it is finite and discrete, can be expressed using a mixture of volumes and a corresponding mixture of moments, see \cite{boumal2018heterogeneous, bandeira2017estimation}. In future work, computational algebra should be applied to these cases to check whether the first and second moments remain sufficient for unique recovery.

To conclude, we raise one further possibility, in some sense at odds with the message of this paper. In the non-uniform case, we have determined that the first and second moments are sufficient information-theoretically for volume recovery. Nonetheless, the resulting optimization landscape is potentially challenging, due to non-convexity or ill-conditioning. Thus, despite the increased \textit{statistical} cost of estimating the third moment, it seems worthwhile to ask what can be gained \textit{computationally} by reprising the third moment in MoM (or at least, using a carefully chosen slice of the third moment). Specifically, we would like to answer this question: can the third moment facilitate more efficient modeling at higher resolution?

\section*{Acknowledgments} The authors thank Nicolas Boumal, Peter B\"urgisser, Eitan Levin, Dilano Saldin and Yoel Shkolnisky for stimulating conversations, and the anonymous referees for their valuable comments.

This research was supported in parts by Award Number R01GM090200 from the NIGMS, FA9550-17-1-0291 from AFOSR, Simons Foundation Math$+$X Investigator Award, the Simons Collaboration on Algorithms and Geometry, the Moore Foundation Data-Driven Discovery Investigator Award, and NSF BIGDATA Award IIS-1837992. BL's research was supported by the European Research Council (ERC) under the European Union's Horizon 2020 research and innovation programme (grant agreement 723991 - CRYOMATH).

\bibliographystyle{plain}
\bibliography{NUcryoEM}

\appendix

\section{Prolate Spheroidal Wave Functions}   \label{app:pswf}
Here we describe key properties of the PSWFs, and propose a method for setting the expansion parameters $L$, $S(\ell)$, $Q$, and $T(q)$. We begin with the three-dimensional PSWFs, where we describe important properties established in the literature~\cite{slepian1964prolate,katz2019sampling,serkh2015prolates}, and outline our choice for setting $L$ and $S(\ell)$, accordingly. Then, we proceed with a short analogous description for the two-dimensional PSWFs (summarizing results of~\cite{slepian1964prolate}), and derive a method for choosing $Q$ and $T(q)$ by directly exploiting the fact that the images to be expanded are tomographic projections of a bandlimited and localized volume function (employing our previous representation for the volume function).

\subsection{Volume function representation with three-dimensional PSWFs}
Let $\Phi:\mathbb{R}^3 \rightarrow \mathbb{R}$ be a square integrable (volume) function on $\mathbb{R}^3$, representing the true underlying electric potential of the molecule, and denote by $\hat{\Phi}$ its three-dimensional Fourier transform. 
It is common practice to assume that $\Phi(x)$ is bandlimited (\ie, $\hat{\Phi}$ is restricted to a ball) while being localized in space. Functions satisfying this property are naturally represented by three-dimensional PSWFs, as detailed next.

We say that the function $\Phi(x)$ as \textit{$c$-bandlimited} if $\hat{\Phi}(\omega)$ vanishes outside a ball of radius $c$. That is, $\Phi$ is {$c$-bandlimited} if
\begin{equation} 
\Phi(x)= \left(\frac{1}{2\pi}\right)^{3} \int_{c\mathbf{B}} \hat{\Phi}(\omega)e^{\imath \omega x}d\omega, \quad x \in\mathbb{R}^{3},
\end{equation}
where $\mathbf{B}$ is the unit ball.
Among all $c$-bandlimited functions, the three-dimensional PSWFs on $\mathbf{B}$~\cite{slepian1964prolate} are the most energy concentrated in $\mathbf{B}$, while constituting an orthonormal system over ${\mathcal{L}^2(\mathbf{B})}$. Namely, they satisfy
\begin{align}
\begin{aligned}
\Psi_i = &\operatorname{argmin}_{\psi} {\left\Vert \psi \right\Vert_{\mathcal{L}^2(\mathbb{R}^3)}} \\
&\text{ subject to} \;\; {\left\Vert \psi \right\Vert_{\mathcal{L}^2(\mathbf{B})}} = 1, \;\; \left\langle \psi, \Psi_j\right\rangle_{\mathcal{L}^2(\mathbf{B})} = 0, \;\; \forall j<i,
\end{aligned}
\label{eqn:PSWFs 3-D optimal conc}
\end{align}
for $i=1,2,\ldots$, \ie, $\Psi_1$ is the most energy concentrated $c$-bandlimited function, $\Psi_2$ is the most energy concentrated $c$-bandlimited function orthogonal to $\Psi_1$, and so on. Three-dimensional PSWFs can be obtained as the solutions to the integral equation
\begin{equation}
\alpha\Psi(x)=\int_{\mathbf{B}}\Psi(\omega)e^{\imath c \omega x}d\omega ,\quad x\in\mathbf{B},\label{eq:Basic PSWF eq}
\end{equation}
where we denote the solutions (the PSWFs with bandlimit $c$) as $\Psi_{\ell,m,s}^c(x)$ and their corresponding eigenvalues as $\alpha_{\ell,m,s}^c$, where the enumeration over $i$ in~\eqref{eqn:PSWFs 3-D optimal conc} is replaced with an enumeration over the triplet $\ell,m,s$ described below, and the eigenvalues appear in non-increasing ordering with respect to the original enumerate $i$.
$\Psi_{\ell,m,s}^c(x)$ and $\alpha_{\ell,m,s}^c$ together form the eigenfunctions and eigenvalues of~\eqref{eq:Basic PSWF eq}, with $m\in\mathbb{Z}$, $\ell\in\mathbb{N}\cup\left\lbrace 0 \right\rbrace$, and $s \in\mathbb{N}$. Furthermore, the functions $\Psi_{\ell,m,s}^c(x)$ are orthogonal on
both $\mathbf{B}$ and $\mathbb{R}^{3}$ using the standard $\mathcal{L}^2$ inner products on $\mathbf{B}$ and $\mathbb{R}^{3}$, respectively, and are dense in both the
class of $\mathcal{L}^{2}(\mathbf{B)}$ functions and in the class of $c$-bandlimited
functions on $\mathbb{R}^{3}$.
In spherical coordinates, the functions $\Psi_{\ell,m,s}^c(x)$ agree with the form in the right-hand side of~\eqref{eqn:phi_hat}, and can be expressed as
\begin{equation}
\Psi_{\ell,m,s}^c(r,\theta,\varphi)=F_{\ell,s}^c(r)Y^m_\ell(\theta,\varphi), 
\end{equation}
where $Y^m_\ell(\theta,\varphi)$ are the spherical harmonics (see~\eqref{eqn:sph_har}).
Numerical evaluation of the three-dimensional PSWFs (in particular of the radial part $F_{\ell,s}^c$) was considered in~\cite{lederman2017numerical}.

From the properties of the three-dimensional PSWFs mentioned above, any volume function $\Phi(x)\in\mathcal{L}^2(\mathbb{R}^3)$ can expanded in $\mathbf{B}$ as
\begin{equation} \label{eq: Infinite PSWF expansion}
\Phi(x) = \sum_{\ell= 0}^{\infty} \sum_{m=-\ell}^{\ell} \sum_{s=1}^{\infty} \widetilde{A}_{\ell,m,s} \Psi_{\ell,m,s}^c(x), \quad x \in \mathbf{B},  \qquad \widetilde{A}_{\ell,m,s} = \int_{\mathcal{\mathbf{B}}} \Phi(x) \overline{\Psi_{\ell,m,s}^c(x)} dx,
\end{equation}
where $\overline{(\cdot)}$ denotes complex conjugation.
Next, we consider the truncation of the expansion in~\eqref{eq: Infinite PSWF expansion}, where it is convenient to bound the resulting truncation error in terms of the assumed spatial localization of $\Phi(x)$.
Towards this end, we say that the function $\Phi(x)$ is \textit{$\varepsilon$-concentrated} if 
\begin{equation}
\sqrt{\int_{x \notin \mathbf{B} } {\left| \Phi(x) \right|}^2 dx} \leq \varepsilon.
\end{equation}
Additionally, we define the normalized eigenvalues
\begin{equation}
\lambda^c_{\ell,m,s}=\left(\frac{c}{2\pi}\right)^3 \left\vert \alpha_{\ell,m,s}\right\vert^2,
\end{equation}
where we mention that $0\leq\lambda^c_{\ell,m,s}\leq 1$, $\lambda^c_{\ell,m,s} = \lambda^c_{\ell,0,s}$ for all triplets $(\ell,m,s)$, and $\lambda^c_{\ell,m,s}\underset{s\rightarrow \infty}{\longrightarrow} 0$ for every $\ell$. Now, we propose to set $S(\ell)$ according to
\begin{equation} \label{eq:PSWFs truncation rule}
S(\ell) = \max_{s\in\mathbb{N}}\left\{s:\;\lambda_{\ell,0,s}^{c} \geq \delta \right\},
\end{equation}
where $\delta\in (0,1)$ is some constant, and set $L$ to be the largest $\ell$ for which $S(\ell)$ is defined (\ie, such that the set $\left\{s:\;\lambda_{\ell,0,s}^{c} \geq \delta \right\}$ is non-empty). Correspondingly, the volume function resulting from the truncating the expansion in~\eqref{eq: Infinite PSWF expansion}, according to the chosen $S(\ell)$ and $L$, is
\begin{equation}
\phi(x) = \sum_{\ell= 0}^{L} \sum_{m=-\ell}^{\ell} \sum_{s=1}^{S(\ell)} \widetilde{A}_{\ell,m,s} \Psi_{\ell,m,s}^c(x). \label{eqn:PSWF truncated exp space}
\end{equation}
The following proposition bounds the error of approximating $\Phi(x)$ by $\phi(x)$.
\begin{proposition}
	Let $\Phi(x)$ be $c$-bandlimited with a unit $\mathcal{L}^2\left( \mathbf{B}\right)$ norm and assume it is $\varepsilon$-concentrated. Then,
	\begin{equation}
	{\left\Vert \Phi - \phi \right\Vert}_{\mathcal{L}^2\left( \mathbf{B}\right)}
	\leq \varepsilon\sqrt{\frac{\delta}{1-\delta}}. \label{eq:Total approx err 1}
	\end{equation}
\end{proposition}
The proof follows immediately from Theorem 5 in~\cite{katz2019sampling} and from our choices of $S(\ell)$ and $L$.
It is evident that the approximation error in the right-hand side of~\eqref{eq:Total approx err 1} can be made arbitrarily small by taking $\delta$ sufficiently small. Furthermore, in the case that $\Phi(x)$ is localized in space, \ie, $\varepsilon \ll 1$, we can take $\delta$ to be large, possibly even close to $1$, and still get  approximation errors sufficiently small for our purposes. 

\subsubsection{Length of the expansion}
Clearly, the number of basis functions taking part in the expansion~\eqref{eqn:PSWF truncated exp space}, which is given explicitly by $\sum_{\ell=0}^L \sum_{m=-\ell}^\ell S(\ell)$, depends on the number of normalized eigenvalues $\lambda^c_{\ell,m,s}$ exceeding $\delta$. In this respect, the normalized eigenvalues $\lambda^c_{\ell,m,s}$ are known to admit the following three distinct regions of behavior (when sorted in descending order). The first is called the ``flat region'', where $\lambda^c_{\ell,m,s}$ take values very close to $1$, the second is called the ``transitional region'', where $\lambda^c_{\ell,m,s}$ shift rapidly from values close to $1$ to values close to $0$, and the third is called the ``decay region'', where $\lambda^c_{\ell,m,s}$ are very close to $0$ and exhibit a super-exponential decay rate. 
As for the number of basis functions chosen according to~\eqref{eq:PSWFs truncation rule}, the following holds~\cite{serkh2015prolates}:
\begin{align} \label{eq:Asymptotic number of terms}
\sum_{\ell=0}^L \sum_{m=-\ell}^\ell S(\ell) & = \left\vert \left\{(\ell,m,s):\;\lambda_{\ell,m,s}^{c} \geq \delta \right\} \right|  \nonumber \\
&=  \frac{c^3}{4.5\pi} + \frac{c^2}{2\pi^2} \log{(c)}\log{(\frac{1-\delta}{\delta})} + o(c^2\log{(c)}), 
\end{align}
where the first, second, and third terms on the right-hand side of~\eqref{eq:Asymptotic number of terms} correspond to the number of normalized eigenvalues $\lambda^c_{\ell,m,s}$ exceeding $\delta$ from the flat region, the transitional region, and the decay region of the eigenvalues, respectively. Clearly, the asymptotically dominant term is $\mathcal{O}(c^3)$, which corresponds to the number of terms in the expansion chosen from the flat region. Additionally, we need an extra $\mathcal{O}(c^2\log{(c)})$ terms if we take $\delta$ to be small (note that the second term in the right hand-side of~\eqref{eq:Asymptotic number of terms} is negative for $\delta>0.5$, meaning that asymptotically we need less than ${c^3}/{4.5\pi}$ terms for values of $\delta$ close to $1$). The remaining $o(c^2\log{(c)})$ terms from the decay region are negligible compared to the leading asymptotic terms.

\subsubsection{Fourier domain representation}
Up to this point, we have shown that three-dimensional PSWFs are naturally adapted for expanding a volume function $\Phi(x)$ which is bandlimited and localized in space, where we provided an appropriate error bound~\eqref{eq:Total approx err 1}. However, note that in~\eqref{eqn:phi_hat} we actually expand the Fourier transform of the molecular potential. We now connect our previous expansion of $\Phi(x)$ with the expansion of its Fourier transform, and show that in fact (and uniquely for PSWFs) the two coincide, in the sense that expanding a function in three-dimensional PSWFs is equivalent to expanding its Fourier transform in three-dimensional PSWFs (after an appropriate scaling and dilation). 
Let $\hat{\Psi}_{\ell,m,s}$ denote the three-dimensional Fourier transform of $\Psi_{\ell,m,s}$, then by~\eqref{eq:Basic PSWF eq} it is easy to verify that 
\begin{equation}
\hat{\Psi}_{\ell,m,s}(\omega) =  \frac{(2\pi)^3}{c^3 \alpha_{\ell,m,s}} \Psi_{\ell,m,s}(\frac{\omega}{c}) \cdot \mathbf{1}_{c\mathbf{B}}(\omega), \label{eqn:PSWF 3-D Fourier}
\end{equation}
where $\mathbf{1}_{c\mathbf{B}}(\omega)$ is the indicator function on $c\mathbf{B}$. It is evident that the Fourier transform of each three-dimensional PSWF is equal to itself up to a constant factor, a dilation by $c$, and a restriction to a ball of radius $c$. Consequently, by taking the Fourier transform of~\eqref{eqn:PSWF truncated exp space} we have
\begin{equation}
\hat{\phi}(\omega) = \sum_{\ell= 0}^{L} \sum_{m=-\ell}^{\ell} \sum_{s=0}^{S_\ell} {A}_{\ell,m,s} \Psi_{\ell,m,s}^c(\frac{\omega}{c}) /{c^{3/2}} ,
\qquad \omega\in c\mathbf{B}, 
\end{equation}
where 
\begin{equation}
{A}_{\ell,m,s} = \frac{8\pi^3}{c^{3/2}\alpha_{\ell,m,s}} \widetilde{A}_{\ell,m,s}. \label{eq: 3-D PSWF Fourier coeffs relation}
\end{equation}

We conclude this part as follows. Given a bandlimit $c$ (typically chosen as the Nyquist frequency corresponding to the projection images' resolution), we take the radial part $F_{\ell,m,s}(k)$ of~\eqref{eqn:phi_hat} as $F_{\ell,m,s}^c(k/c)/{c^{3/2}} \cdot \mathbf{1}_{c}(k)$, where $\mathbf{1}_{c}(k)$ is the indicator function on $[0,c]$, and $F_{\ell,m,s}^c(r)$ is the radial part of the three-dimensional PSWFs on $\mathbf{B}$ (the factor $1/{c^{3/2}}$ ensures that $F_{\ell,m,s}(k)$ are orthonormal over $[0,\infty)$ w.r.t the measure $k^2 dk$). Then, setting $S(\ell)$ according to~\eqref{eq:PSWFs truncation rule} for a given parameter $\delta$ allows for the controlled approximation error~\eqref{eq:Total approx err 1}.

\subsection{Projection image representation with two-dimensional PSWFs} \label{sec:appendix_2D_PSWF}
In the sequel, we are interested in providing a suitable representation for the projection images of the rotated copies of $\phi(x)$. By the Fourier slice theorem, the two-dimensional Fourier transforms of such projections are equal to slices from the three-dimensional Fourier transform of $\phi(x)$ (\ie, of $\hat{\phi}(\omega)$). Therefore, if $\phi(x)$ is $c$-bandlimted, then the projection images are bandlimited to a disk of radius $c$. Additionally, we expect the projection images to be localized in the unit disk if $\phi(x)$ is sufficiently localized in the unit ball. For such projection images, two-dimensional PSWFs are expected to provide a natural representation (see~\cite{landa2017approximation}).

We briefly summarize properties of the two-dimensional PSWFs which are used in our context. In essence, the properties of the two-dimensional PSWFs are analogous to those of the three-dimensional PSWFs when replacing the unit ball $\mathbf{B}$ with the unit disk $\mathbf{D}$.
Let $P:\mathbb{R}^2 \rightarrow \mathbb{R}$ be a square integrable function on $\mathbb{R}^2$, representing a tomographic projection of $\phi$.
We say that $P(x)$ as \textit{$c$-bandlimited} if its two-dimensional Fourier transform, denoted by $\hat{P}(\omega)$, vanishes outside a disk of radius $c$. That is, $P$ is {$c$-bandlimited} if
\begin{equation} 
P(x)= \left(\frac{1}{2\pi}\right)^{2} \int_{c\mathbf{D}} \hat{P}(\omega)e^{\imath \omega x}d\omega, \quad x \in\mathbb{R}^{2}.
\end{equation}
Among all $c$-bandlimited functions, the two-dimensional PSWFs on $\mathbf{D}$ are the most energy concentrated in $\mathbf{D}$, that is, they satisfy~\eqref{eqn:PSWFs 3-D optimal conc} when replacing $\mathbf{B}$ with $\mathbf{D}$, while constituting an orthonormal system over ${\mathcal{L}^2(\mathbf{D})}$. The two-dimensional PSWFs were derived and analyzed in~\cite{slepian1964prolate}, and were shown to be the solutions to the integral equation
\begin{equation}
\beta\psi(x)=\int_{\mathbf{D}}\psi(\omega)e^{\imath c \omega x}d\omega ,\quad x\in\mathbf{D}.\label{eq:Basic 2-D PSWF eq}
\end{equation}
We denote the PSWFs with bandlimit $c$ as $\psi_{q,t}^c(x)$, and their corresponding eigenvalues as $\beta_{q,t}^c$,
which together form the eigenfunctions and eigenvalues of~\eqref{eq:Basic 2-D PSWF eq}, with $q\in\mathbb{Z}$, and $t \in\mathbb{N}$. Furthermore, the functions $\psi_{q,t}^c(x)$ are orthogonal on
both $\mathbf{D}$ and $\mathbb{R}^{2}$ using the standard $\mathcal{L}^2$ inner products on $\mathbf{D}$ and $\mathbb{R}^{2}$, respectively, and are dense in both the
class of $\mathcal{L}^{2}(\mathbf{D)}$ functions and in the class of $c$-bandlimited
functions on $\mathbb{R}^{2}$.
In polar coordinates, the functions $\psi_{q,t}^c(x)$ agree with the form in the right-hand side of~\eqref{eqn:image_expand}, and can be expressed as
\begin{equation}
\psi_{q,t}^c(r,\varphi)=\frac{1}{\sqrt{2\pi}}f_{q,t}^c(r)e^{\imath q \varphi}, 
\end{equation}
where the eigenfunctions $\psi_{q,t}^c(x)$ are normalized to have an $\mathcal{L}^2(\mathbf{D})$ norm of $1$. 
Numerical evaluation of the two-dimensional PSWFs was considered in~\cite{shkolnisky2007prolate}.

From the properties of the two-dimensional PSWFs mentioned above, any function $P(x)\in\mathcal{L}^2(\mathbb{R}^2)$ can be expanded in $\mathbf{D}$ as
\begin{equation}
P(x) = \sum_{q= -\infty}^{\infty}\sum_{t=0}^{\infty} \widetilde{a}_{q,t} \psi_{q,t}^c(x), \quad x \in \mathbf{D}, \quad \qquad \widetilde{a}_{q,t} = \int_{\mathcal{\mathbf{D}}} P(x) \overline{\psi_{q,t}^c(x)} dx.
\end{equation}
Now, considering the truncated expansion
\begin{equation}
I(x) := \sum_{q= -Q}^{Q}\sum_{t=0}^{T(q)} \widetilde{a}_{q,t} \psi_{q,t}^c(x), \label{eqn:PSWFs 2-D truncated expansion}
\end{equation}
we are interested in controlling the error
\begin{equation}
\left\Vert P - I \right\Vert^2_{\mathcal{L}^2_\mathbf{D}} = \sum_{q=-Q}^Q \sum_{t > T(q)} \left\vert \widetilde{a}_{q,t} \right\vert^2 + \sum_{|q|>Q} \sum_{t = 0}^\infty \left\vert \widetilde{a}_{q,t} \right\vert^2. \label{eqn:2-D PSWF approx err}
\end{equation}
From~\eqref{eq:Basic 2-D PSWF eq}, the Fourier transform of $\psi_{m,k}$ can be expressed as
\begin{equation}
\hat{\psi}_{m,k}(\omega) = \frac{4\pi^2}{c^2 \beta_{m,k}} \psi_{m,k}(\frac{\omega}{c}) \cdot \mathbf{1}_{c\mathbf{D}}(\omega),
\end{equation}
where $\mathbf{1}_{c\mathbf{D}}(\omega)$ is the indicator function on $c\mathbf{D}$, which is analogous to the relation between the three-dimensional PSWFs $\Psi_{\ell,m,s}^c$ and their Fourier transforms $\hat{\Psi}_{\ell,m,s}^c$ in~\eqref{eqn:PSWF 3-D Fourier}. Continuing, taking the Fourier transform of~\eqref{eqn:PSWFs 2-D truncated expansion} gives
\begin{equation}
\hat{I}(\omega) = \sum_{q= -Q}^{Q}\sum_{t=0}^{T(q)} {a}_{q,t} \psi_{q,t}^c(\frac{\omega}{c})\sqrt{2\pi}/c, \qquad \omega\in c\mathbf{D} \label{eq:2-D Fourier PSWF expansion}
\end{equation}
where 
\begin{equation}
{a}_{q,t} = \frac{(2\pi)^{3/2}}{c \beta_{q,t}} \widetilde{a}_{q,t}. \label{eq:2-D PSWF Fourier coeffs relation}
\end{equation}
We will now relate 2D basis representation error to that of the 3D basis functions. Comparing the 2D expansion~\eqref{eq:2-D Fourier PSWF expansion} with the relation between 2-D and 3-D coefficients~\eqref{eqn:a_A_relation}, while employing~\eqref{eq:2-D PSWF Fourier coeffs relation} and~\eqref{eq: 3-D PSWF Fourier coeffs relation} we have
\begin{equation}
\widetilde{a}_{q,t} = \frac{c \beta_{q,t} }{(2\pi)^{3/2}} {a}_{q,t} = \sum_{\ell =\vert q \vert}^L \sum_{s=1}^{S(\ell)} \sum_{m=-\ell}^\ell   \widetilde{A}_{\ell,m,s}  \, U^{\ell}_{m,q}(R) \, \eta_{\ell,s}^{q, t}, 
\end{equation}
for $|q|\leq L$, where $\widetilde{a}_{q,t}=0$ for $|q|>L$, and
\begin{equation}
\eta_{\ell,s}^{q, t} = \frac{c \beta_{q,t} }{(2\pi)^{3/2}} \frac{8\pi^3}{c^{3/2}\alpha_{\ell,m,s}} \gamma_{\ell,s}^{q, t} = \frac{(2\pi)^{3/2} \beta_{q,t} }{\sqrt{c} \alpha_{\ell,m,s}} \gamma_{\ell,s}^{q, t},
\end{equation}
where $\gamma_{\ell,s}^{q, t}$ is from~\eqref{eqn:gamma_def}.
Using the Cauchy--Schwarz inequality, we can write
\begin{align}
|\widetilde{a}_{q,t}| &\leq \sum_{\ell =\vert q \vert}^L \sum_{s=1}^{S(\ell)} \sum_{m=-\ell}^\ell   | \widetilde{A}_{\ell,m,s}  \, U^{\ell}_{m,q}(R) \, \eta_{\ell,s}^{q, t}| \nonumber \\
&\leq \left( \sum_{\ell =\vert q \vert}^L \sum_{s=1}^{S(\ell)} \sum_{m=-\ell}^\ell \left\vert \widetilde{A}_{\ell,m,s} \right\vert^2 \right)^{1/2} \left(\sum_{\ell =\vert q \vert}^L \sum_{s=1}^{S(\ell)} \sum_{m=-\ell}^\ell \left\vert \, U^{\ell}_{m,q}(R) \, \eta_{\ell,s}^{q, t} \right\vert^2\right)^{1/2} \nonumber \\ 
&\leq \left\Vert \phi \right\Vert_{\mathcal{L}^2(\mathbf{B})} \cdot \left(\sum_{\ell =\vert q \vert}^L \sum_{s=1}^{S(\ell)} |\eta_{\ell,s}^{q,t}|^2\right)^{1/2}, \label{eqn:2-D PSWFs a_qt bound}
\end{align}
where we also used the fact that $U^{\ell}(R)$ is a unitary matrix. Finally, taking $Q=L$ and assuming w.l.o.g that $\left\Vert \phi \right\Vert_{\mathcal{L}^2(\mathbf{B})}=1$, we obtain from~\eqref{eqn:2-D PSWF approx err} and~\eqref{eqn:2-D PSWFs a_qt bound} that
\begin{equation}
\left\Vert P - I \right\Vert^2_{\mathcal{L}^2_\mathbf{D}} \leq \sum_{q=-L}^L \sum_{t > T(q)} \sum_{\ell =\vert q \vert}^L \sum_{s=1}^{S(\ell)} |\eta_{\ell,s}^{q,t}|^2.
\end{equation}
Given a prescribed accuracy $\epsilon$, for every $-L\leq q\leq L$ we choose $T(q)$ to be the smallest integer such that
\begin{equation} \label{eqn:epsilon_eq}
\sum_{t > T(q)} \sum_{\ell =\vert q \vert}^L \sum_{s=1}^{S(\ell)} |\eta_{\ell,s}^{q,t}|^2 \leq \frac{\epsilon}{2L+1},
\end{equation}
which results in
\begin{equation}
\left\Vert P - I \right\Vert^2_{\mathcal{L}^2_\mathbf{D}} \leq \epsilon,
\end{equation}
where $\eta_{\ell,s}^{q,t}$ are computed by evaluating $\gamma_{\ell,s}^{q,t}$ of~\eqref{eqn:gamma_def} via numerical integration (using Gauss-Legendre quadratures).
Note that the right-hand side of~\eqref{eqn:epsilon_eq} is determined by the decay rate of $\eta_{\ell,s}^{q,t}$ in $t$, which is dominated by the decay rate of the the eigenvalues of the two-dimensional PSWFs $\beta_{q,t}$. Those are known to admit a rapid decay in the form of a super-exponential decay rate following a certain transitional region (see~\cite{boulsane2018finite,serkh2015prolates}). Hence, if $T(q)$ is sufficiently large then~\eqref{eqn:epsilon_eq} can be satisfied for an arbitrarily small $\epsilon$ with a marginal increase in the number of required terms. Last, we mention that when provided with images sampled on a Cartesian grid, the coefficients $\widetilde{a}_{q,t}$ can be approximated accurately from the images by fast algorithms~\cite{landa2017approximation,landa2017steerable}.

\section{Linearizing polynomial maps with the Jacobian matrix} \label{sec:appendix_Jacobian}
In this section, we describe the linearization technique from computational algebraic geometry we used to obtain the uniqueness results in Tables~\ref{table:jac-kn-inp},~\ref{table:jac-unk-tot},~\ref{table:jac-ukn-inp} from Section~\ref{sec:algebraic geometry}.
The first paper to apply algebraic geometry techniques to cryo-EM was \cite{bandeira2017estimation}. Nevertheless, similar Jacobian tests have been used in other applications such as for testing rigidity in sensor network localization, see e.g.,~\cite{gortler2010characterizing} and testing whether a matrix can be completed into a low-rank matrix~\cite{singer2010uniqueness}. 

To state the method, we fix $\mathbb{C}^{N} =\mathbb{C}^{N'} \oplus \mathbb{C}^{N''}$, let $\pi'$ and $\pi''$ be projection onto the factors, and consider a \textit{polynomial map} $F=(F_1, \ldots, F_M): \mathbb{C}^{N} \rightarrow \mathbb{C}^{M}$  (that is, each coordinate function $F_i = F_i(x_1,\ldots,x_M)$ is a polynomial on $\mathbb{C}^{N}$).  While $F$ is generally a nonlinear map, its first derivative at $q \in \mathbb{C}^{N}$ is a linear map represented by the Jacobian matrix 
\begin{equation}
dF := \left( \, \frac{\partial F_{j}}{\partial x_{i}} \, \right)_{\substack{i=1,\ldots,M \\ j = 1, \ldots, N}}
\end{equation}
In addition, we define the \textit{fiber} in $q \in \mathbb{C}^{N}$ by 
\[ F_{q} := \{\widetilde{q} \in \mathbb{C}^{N} \mid F(\widetilde{q}) = F(q) \} \subset \mathbb{C}^{N}, \] 
and the \textit{projected fiber} by 
\[ \overline{\pi'(F_{q})} \subset \mathbb{C}^{N'}. \]
For $q' \in \mathbb{C}^{N'}$ and $q'' \in \mathbb{C}^{N''}$, define the \textit{specialized fiber} by 
$$ \left(F \vert_{\mathbb{C}^{N'} \oplus q''} \right)_{q'} = \{ \widetilde{q}' \in \mathbb{C}^{N'} \mid F(\widetilde{q}' \oplus q'') = F(q' \oplus q'') \} \subset \mathbb{C}^{N'}. $$ 

Because $F$ is described by polynomials, there is a tight relationship between the dimension of fibers of $F$ (as algebraic varieties) and the dimension of the kernels of $dF$ (as linear spaces). This is summarized by the Jacobian tests below.  Somewhat remarkably, the linear algebra tests are done at a \textit{single point} in the domain of $F$, but imply algebraic geometric statements for \textit{almost all} points in the domain of $F$.

\begin{theorem}\label{thm:jac}
Suppose it is known that, generically, the fiber, projected fiber and specialized fiber have dimensions $\geq d_1, d_2, d_3$, respectively (if we have no such knowledge, then $d_{1} = d_{2} = d_{3} = 0$).  Choose particular points $q_{0} \in \mathbb{C}^{N}$, $q'_{0} \in \mathbb{C}^{N'}$ and $q''_{0} \in \mathbb{C}^{N''}$.
\begin{itemize} \setlength\itemsep{0.4em}
\item[$-$] \textup{\underline{Vanilla Jacobian test:}} if $\textup{rank} \, dF(q_0) = N-d_1$, then generic fibers have dimension exactly $d_1$.
\item[$-$]  \textup{\underline{Projected Jacobian test:}} if $\dim \, \pi' (\textup{ker} \, dF(q_0))  = d_2$, then generic projected fibers have dimension exactly $d_2$.
\item[$-$]  \textup{\underline{Specialized Jacobian test:}} if $\textup{rank } \, d \left( F \vert_{\mathbb{C}^{N'} \oplus \,  q_0''} \right)(q_0')  = N' - d_{3}$, then generic specialized fibers have dimension exactly $d_3$.
\end{itemize}
\end{theorem}

\begin{footnotesize}
\begin{table}[ht]\setlength{\tabcolsep}{8pt}
    \centering
    \caption{Vanilla, projected and specialized Jacobian tests: these show that a system of polynomial equations generically has only finitely solutions.  Notation: $F : \mathbb{C}^{N} \rightarrow \mathbb{C}^{M}$ is a polynomial map, $\mathbb{C}^{N} = \mathbb{C}^{N'} \oplus \mathbb{C}^{N''}$ where $\pi'$, $\pi''$ are orthogonal projections onto the factors, and $d_1, d_2, d_3$ are the dimension bounds in Theorem~\ref{thm:jac}.}
    \resizebox{\textwidth}{!}{
    \begin{tabular}{l c  c  c  c }
    \toprule
         & \textup{polynomial map}  & \textup{arbitrary choices} & \textup{linearization} & \textup{rank check} \\
    \midrule
    \textup{vanilla} & $\mathbb{C}^{N} \overset{F}{\longrightarrow} \mathbb{C}^{M}$ & $q \in \mathbb{C}^{N}$ & $\mathbb{C}^{N} \overset{dF(q)}{\longrightarrow} \mathbb{C}^{M}$ & $\textup{rank}(dF(q)) = N-d_1$  \\
    \midrule
    \textup{projected} & \begin{tabular}{@{}c@{}} $\mathbb{C}^{N'} \oplus \mathbb{C}^{N''} =$ \\ $ \mathbb{C}^{N} \overset{F}{\longrightarrow} \mathbb{C}^{M} $\end{tabular}  & $q \in \mathbb{C}^{N}$ & \begin{tabular}{@{}c@{}} $\mathbb{C}^{N'} \subset \mathbb{C}^{N}$ \\ $ \overset{dF(q)}{\longrightarrow} \mathbb{C}^{M} $\end{tabular}  & $\dim \left( \pi' (\textup{ker} \, dF(q)) \right) = d_2$ \\
    \midrule
    \textup{specialized} & \begin{tabular}{@{}c@{}} $\mathbb{C}^{N'} \oplus \mathbb{C}^{N''} =$ \\ $ \mathbb{C}^{N} \overset{F}{\longrightarrow} \mathbb{C}^{M} $\end{tabular}  & \begin{tabular}{@{}c@{}} $q' \in \mathbb{C}^{N'}$  \\ $ q'' \in \mathbb{C}^{N''} $\end{tabular}  & \begin{tabular}{@{}c@{}} $\mathbb{C}^{N'} \oplus q''$ \\ $ \overset{d F (q' \oplus q'')}{\longrightarrow} \mathbb{C}^{M} $\end{tabular}  & \begin{footnotesize}\begin{tabular}{@{}c@{}} $\textup{rank } d \left( F \vert_{\mathbb{C}^{N'} \oplus \,  q''} \right)(q')$ \\ $ = N' - d_{3} $\end{tabular} \end{footnotesize} \\
    \bottomrule
    \end{tabular} }
    \label{tab:dimensions}
\end{table}
\end{footnotesize}

Several technical remarks are in order.  Firstly, in Theorem~\ref{thm:jac}, the fiber, projected fiber and specialized fiber are \textit{affine algebraic varieties} and hence a \textit{dimension} is defined for each of their \textit{irreducible components} according to \cite{CLO-book1}. The meaning of the theorem is that 
each component has dimension exactly $d_1, d_2, d_3$, respectively.
 Crucially, affine algebraic varieties have finitely many components.  Thus the theorem implies ``finitely many solutions'' up to symmetries, if the symmetries give $d_1, d_2, d_3$-dimensional ambiguities, respectively. 
 Secondly, ``generic'' in Theorem~\ref{thm:jac} is with respect to the \textit{Zariski topology}.  Concretely, there exists some polynomial $G$ on $\mathbb{C}^N$ such that for all $q \in \mathbb{C}^{N}$ with $G(q) \neq 0$ the implications in the theorem hold.  In particular, any property that holds generically holds on a Lebesgue full measure subset of points.
Thirdly, the Jacobian ranks in Theorem~\ref{thm:jac} take on generic values, as each minor of the relevant matrix is a polynomial in $q, q', q''$.  

Theorem~\ref{thm:jac} states rigorous conclusions if the Jacobian rank tests are \textit{passed}.
On the other hand, if the tests \textit{fail} for $q_0, q_0', q_0''$, and $q_0, q_0', q_0''$ were drawn randomly from \textit{any} continuous distribution on $\mathbb{F}^{N}$,
then by genericity of the Jacobian ranks, with probability $1$, the generic fibers, projected fibers, or specialized fibers of $F$ have dimension \textit{strictly} more than $d_1$, $d_2$, or $d_3$.

We applied the specialized test in subsection~\ref{sec:uni2} with $d_1=0$, the vanilla and projected tests in subsection~\ref{sec:uni3} with $d_1=d_2=3$ and the vanilla test in subsection~\ref{sec:uni4} with $d_1=3$.  The settings of $3$ reflect the fact, in the latter two subsections, that the fibers are $\textup{SO}(3)$-sets and we are interested in solutions modulo global rotation.   The bounds may be seen as instances of the orbit-stabilizer theorem, see \cite[Proposition~4.11]{bandeira2017estimation}.  
When the Jacobian rank tests were passed, this meant that, generically, there are only finitely many solutions up to global ambiguities.  

In practice, we ran the Jacobian tests in floating-point arithmetic and used SVD for robust rank estimation. Namely, we 
looked at multiplicative gaps between consecutive singular values, and regarded any gap exceeding a predefined threshold ($10^6$) as evidence that all
lower singular values should be regarded as zero.
While these computations fall short of a fully rigorous mathematical proof due to the possibility of rounding errors in floating-point arithmetic, it was typically evident which singular values ought to be counted as zero or non-zero. 

\bigskip
\bigskip

\end{document}